\documentclass{amsart}
\usepackage{txfonts}
\usepackage{mathrsfs}
\usepackage{latexsym}
\usepackage{graphicx}
\usepackage{amscd,amssymb,amsmath,amsbsy,amsthm,amsfonts}
\usepackage[all]{xy}
\usepackage[colorlinks,plainpages,backref,urlcolor=blue]{hyperref}

\newtheorem{theorem}{Theorem}[section]
\newtheorem{lemma}[theorem]{Lemma}

\newtheorem{corollary}[theorem]{Corollary}
\newtheorem{prop}[theorem]{Proposition}

\theoremstyle{definition}
\newtheorem{definition}[theorem]{Definition}
\newtheorem{example}[theorem]{Example}
\newtheorem{remark}[theorem]{Remark}

\newtheorem{thm}{Theorem}

\newcommand{\N}{\mathbb{N}}
\newcommand{\Z}{\mathbb{Z}}
\newcommand{\Q}{\mathbb{Q}}

\newcommand{\C}{\mathbb{C}}

\newcommand{\QP}{\mathbb{QP}}
\newcommand{\CP}{\mathbb{CP}}
\newcommand{\RP}{\mathbb{RP}}
\newcommand{\bP}{\mathbb{P}}

\renewcommand{\k}{\Bbbk}

\newcommand{\HH}{\mathcal{H}}
\newcommand{\p}{\mathfrak{p}}
\newcommand{\m}{\mathfrak{m}}

\DeclareMathOperator{\rank}{rank}

\DeclareMathOperator{\im}{im}
\DeclareMathOperator{\coker}{coker}
\DeclareMathOperator{\codim}{codim}

\DeclareMathOperator{\id}{id}
\DeclareMathOperator{\ab}{{ab}}
\DeclareMathOperator{\fab}{{fab}}

\DeclareMathOperator{\supp}{supp}

\DeclareMathOperator{\GL}{GL}
\DeclareMathOperator{\T}{Tors}
\DeclareMathOperator{\Hom}{{Hom}}

\DeclareMathOperator{\spn}{span}
\DeclareMathOperator{\ann}{{ann}}

\DeclareMathOperator{\lk}{lk}
\DeclareMathOperator{\Par}{P}

\DeclareMathOperator{\Lie}{Lie}

\DeclareMathOperator{\Spm}{{maxSpec}}

\DeclareMathOperator{\Aut}{Aut}

\DeclareMathOperator{\ii}{i}
\DeclareMathOperator{\ord}{ord}

\DeclareMathOperator{\lcm}{lcm}
\DeclareMathOperator{\Tors}{Tors}

\DeclareMathOperator{\Grass}{Gr}
\DeclareMathOperator{\Epi}{{Epi}}
\DeclareMathOperator{\alg}{{alg}}
\DeclareMathOperator{\group}{{group}}

\newcommand{\sV}{\mathsf{V}}
\newcommand{\sW}{\mathsf{W}}
\newcommand{\sE}{\mathsf{E}}
\newcommand{\cS}{\mathsf{p}}

\newcommand{\wG}{\widehat{G}}
\newcommand{\wH}{\widehat{H}}
\newcommand{\wA}{\widehat{A}}
\newcommand{\oH}{\overline{H}}
\newcommand{\oA}{\overline{A}}

\newcommand{\same}{\Longleftrightarrow}
\newcommand{\surj}{\twoheadrightarrow}
\newcommand{\inj}{\hookrightarrow}
\newcommand{\isom}{\!\xymatrixcolsep{10pt}\xymatrix{\ar^{\simeq}[r]&}\!}

\newcommand{\DS}{\displaystyle}
\newcommand{\abs}[1]{\left| #1 \right|}

\newcommand{\bigmid}{\:\big|  \big.\:}
\newcommand{\dv}{\! \mid \!}
\newcommand{\pt}{\{\text{pt}\}}

\newcommand{\pullbackcorner}[1][dr]{\save*!/#1-1.2pc/#1:(-1,1)@^{|-}\restore}

\newenvironment{romenum}
{

\begin{enumerate}}{\end{enumerate}}

\title[Homological finiteness of abelian covers]%
{Homological finiteness of abelian covers}

\author[A.~Suciu]{Alexander~I.~Suciu$^1$}
\address{Department of Mathematics,
Northeastern University,
Boston, MA 02115, USA}
\email{a.suciu@neu.edu}
\thanks{$^1$Partially supported by NSA grant H98230-09-1-0021
and NSF grant DMS--1010298}

\author[Y.~Yang]{Yaping~Yang}
\address{Department of Mathematics,
Northeastern University,
Boston, MA 02115, USA}
\email{yang.yap@husky.neu.edu}

\author[G.~Zhao]{Gufang~Zhao}
\address{Department of Mathematics,
Northeastern University, Boston, MA 02115, USA}
\email{zhao.g@husky.neu.edu}



\begin{document}

\begin{abstract}
We present a method for deciding when a regular abelian cover
of a finite CW-complex has finite Betti numbers. To start with, we
describe a natural parameter space for all regular covers of a
finite CW-complex $X$, with group of deck transformations a fixed
abelian group $A$, which in the case of free abelian covers
of rank $r$ coincides with the Grassmanian of $r$-planes in
$H^1(X,\Q)$. Inside this parameter space, there is a subset
$\Omega_A^i(X)$ consisting of all the covers with finite Betti
numbers up to degree $i$.

Building on work of Dwyer and Fried, we show how to
compute these sets in terms of the jump loci for homology
with coefficients in rank~$1$ local systems on $X$.  For
certain spaces, such as smooth, quasi-projective varieties,
the generalized Dwyer--Fried invariants that we introduce
here can be computed in terms of intersections of
algebraic subtori in the character group. For many
spaces of interest, the homological finiteness of
abelian covers can be tested through the corresponding
free abelian covers. Yet in general, abelian covers exhibit
different homological finiteness properties than their free
abelian counterparts.
\end{abstract}

\maketitle
\tableofcontents

\newpage

\section{Introduction}
\label{sect:intro}

By classical covering space theory, the connected, regular covers of a
CW-complex are classified by the quotients of its fundamental group.
In this paper, we investigate the set of covers with fixed deck-transformation
group (usually taken to be an abelian group) for which the Betti numbers
up to a fixed degree are finite.

\subsection{A parameter set for regular covers}
\label{subsec:covers}

Let $X$ be a connected CW-complex with finite $1$-skeleton.
Let $G=\pi_1(X,x_0)$ be the fundamental group, and
let $A$ be a quotient of $G$.  The regular covers of $X$
with group of deck transformations isomorphic to $A$ can
be parametrized by the set
\begin{equation}
\label{eq:param}
\Gamma(G,A)=\Epi(G,A)/\Aut(A),
\end{equation}
where $\Epi(G,A)$ is the set of all
epimorphisms from $G$ to $A$ and $\Aut(A)$ is the
group of automorphisms of $A$, acting on $\Epi(G,A)$
by composition.  For an epimorphism $\nu\colon G \surj A$,
we write its class in $\Gamma(G,A)$ by $[\nu]$, and the
corresponding cover by $X^{\nu}\to X$.

In the case when $A$ is abelian, the parameter set for $A$-covers
may be identified with $\Gamma(H,A)$, where $H=H_1(X,\Z)$.
Our first result identifies this set with a (set-theoretical) twisted
product over a rational Grassmannian.

More precisely, let $\overline{H}$ be the maximal, torsion-free
abelian quotient of $H$, and identify $\overline{H}=\Z^n$
and $\overline{A}=\Z^r$.  By linear algebra (see, e.g., 
\cite[\S{12}, Theorem 4.3]{Art}), any rational subspace 
of $\Z^n\otimes\Q$ has a primitive sublattice, so this 
allows us to identify $\Grass_{r}(\Z^n)$ with 
$\Grass_{r}(\Q^n)$. Let $\Par$ be a parabolic
subgroup of $\GL_n(\Z)$, such that $\GL_n(\Z)/\Par$
is the Grassmannian $\Grass_{n-r}(\Z^n)$.

\begin{thm}[Theorem~\ref{theorem:new bij}]
\label{ref:intro bij}
There is a bijection
\[
\Gamma(H, A) \longleftrightarrow \GL_{n}(\Z) \times_{\Par} \Gamma,
\]
where $\Gamma$ is the finite set $\Gamma(H/\oA,A/\oA)$, and
$\GL_{n}(\Z) \times_{\Par} \Gamma$ is the twisted
product of $\GL_{n}(\Z)$ and $\Gamma$ under the natural action
of $\Par$ on the two sets.
\end{thm}

\subsection{Dwyer--Fried sets and their generalizations}
\label{subsec:omega-sets}
In a foundational paper on the subject, \cite{DF},
Dwyer and Fried considered the regular covers
of a finite, connected CW-complex $X$, with group
of deck transformations $A=\Z^r$.  Inside the parameter
space $\Gamma(H,\Z^r)=\Grass_{r}(\Q^n)$, where $n=\rank H$,
they isolated the sets $\Omega^i_r(X)$, consisting of those covers
for which the Betti numbers up to degree~$i$ are finite.

The Dwyer--Fried sets $\Omega^i_r(X)$ have since been
studied in depth in \cite{PS-plms, Su14},
using the characteristic varieties of $X$.  These varieties,
$V^i(X)$, are Zariski closed subsets of the character group
$\widehat{H}=\Hom(H,\C^*)$; they consist of those rank $1$
local systems on $X$ for which the corresponding cohomology groups
do not vanish, for some degree less or equal to $i$.

We further develop this theory here, by first defining the
generalized Dwyer--Fried invariants of $X$ to be the subsets
$\Omega_A^i(X)$ of $\Gamma(G, A)$ consisting of those regular
$A$-covers having finite Betti numbers up to degree~$i$.
In the case when $A$ is a finitely generated (not necessarily
torsion-free) abelian group, we establish a similar formula, computing
the invariants $\Omega^i_A(X)$, viewed now as subsets of $\Gamma(H,A)$,
in terms of the characteristic varieties of $X$.

\begin{thm}[Theorem~\ref{main theorem}]
\label{thm:intro b}
Let $X$ be a connected, finite CW-complex, and let $H=H_1(X,\Z)$.
Suppose $\nu\colon H \surj A$ is an epimorphism to an abelian
group $A$.  Then
\[
\Omega^{i}_{A}(X)= \{[\nu]\in \Gamma(H, A) \mid
\text {$\im(\hat{\nu}) \cap V^{i}(X)$ is finite} \},
\]
where $V^{i}(X)\subset \widehat{H}$ is the $i$-th characteristic
variety of $X$, and $\hat\nu\colon  \wA \to \widehat{H}$
is the induced morphism between character groups.
\end{thm}

\subsection{An upper bound for the $\Omega$-sets}
\label{subsec:omega-sets bound}

In order to estimate the size of the Dwyer--Fried sets, it
is convenient to look at various analogues of the incidence
subvarieties of the Grassmannian, known as the
special Schubert varieties.

Given a subgroup $\xi\le H$ and an abelian group $A$, let
$\sigma_A(\xi)$ be the set of all $[\nu] \in \Gamma(H,A)$ for
which $\rank (\ker(\nu) + \xi) < \rank H$, and
let $U_A(\xi)$ be the subset of those $[\nu]$
for which, additionally, $\ker(\nu)\cap \overline{\xi}\subseteq \xi$,
where $\bar{\xi}$ denotes the primitive closure of $\xi$.

Each subgroup $\xi \le H$  gives rise to an algebraic subgroup,
$V(\xi)=\widehat{H/\xi}$, of the character group $\widehat{H}$,
with identity component $V(\overline{\xi})$.
Given a subvariety $W \subset \widehat{H}$, and a positive integer $d$,
let $\Xi_d(W)$ be the collection of all subgroups $\xi \le H$
for which the determinant group $\overline{\xi}/\xi$ is cyclic
of order dividing $d$, and there is a generator
$\eta$ of $\widehat{\overline{\xi}/\xi}$ such that $\eta V(\overline{\xi})$
is a maximal, positive-dimensional translated subtorus in $W$.
(The set $\Xi_1(W)$ is essentially the same as the ``exponential
tangent cone" from \cite{DPS-duke, Su14}.)

\begin{thm}[Theorem \ref{thm:omega bound}]
\label{thm:intro c}
 Let $H=H_1(X,\Z)$, and let $A$ be a quotient
of $H$. Then:
\begin{equation*}
\label{eq:intro_up_bd}
\Omega_A^i(X)\subseteq \Gamma(H,A)
\setminus \bigcup_{d\ge 1}\bigcup_{\xi\in \Xi_{d}(V^i(X))}
U_A(\xi) .
\end{equation*}
\end{thm}

In the case when $A$ has rank one, we show in
Theorem~\ref{thm:rank1 omega} that this upper bound
is reached; furthermore, we only need to use in this case the set
$\Xi_{c(A)}(V^i(X))$, where $c(A)$ is the largest order
of any element in $A$.  In other words, if $\rank A=1$, then
\begin{equation}
\label{eq:intro rank1}
\Omega_A^i(X)=\Gamma(H,A)
\setminus \bigcup_{\xi\in \Xi_{c(A)}(V^i(X))} U_A(\xi).
\end{equation}

\subsection{Translated subgroups in the characteristic varieties}
\label{subsec:omega-sets translated subgroup}

For a large class of spaces---for instance, smooth, complex quasi-projective
varieties---the characteristic varieties are union of translated algebraic
subgroups of $\widehat{H}$.
Using techniques from \cite{SYZ}, we obtain several explicit formulas
in this situation, expressing the Dwyer--Fried sets purely in terms of the
corresponding subgroups of $H$, and the associated translation factors.

\begin{thm}[Theorem \ref{thm:tt omega}]
\label{thm:intro d}
Suppose $V^i(X)=\bigcup_{j=1}^s \eta_j V(\xi_j)$, where
$\xi_1, \dots, \xi_s$ are subgroups of $H=H_1(X,\Z)$, and
$\eta_1, \dots, \eta_s$ are torsion elements in $\widehat{H}$.
Then
\[
\Omega_A^i(X)= \Gamma(H,A)
\setminus \bigcup_{\xi\in \Xi_{c}(V^i(X))}
U_A(\xi) ,
\]
where $c$ is the lowest common multiple of
$\ord(\eta_1)\cdot c(\overline{\xi}_1/\xi_1), \dots ,
\ord(\eta_s)\cdot c(\overline{\xi}_s/\xi_s)$.
\end{thm}

A formula of a different flavor is given in Theorem~\ref{thm:xik}.
A particular case of this formula is worth singling out: if
$V^i(X)=V(\xi_1)\cup \cdots \cup V(\xi_s)$ is
a union of algebraic subgroups of $\wH$, then
\begin{equation}
\label{eq:intro qoa}
\Omega^i_A(X) = \Gamma(H,A) \setminus \bigcup_{j=1}^{s}
q^{-1} \left( \sigma_{\oA}(\xi_j) \right) ,
\end{equation}
where $q$ is the canonical projection $\Gamma(H,A)$ onto
the rational Grassmannian $\Gamma(H,\overline{A})$, and the sets
$\sigma_{\oA}(\xi_j)$ are usual special Schubert varieties.

\subsection{Abelian versus free abelian covers}
\label{subsec:ab vs fab}

This last formula brings up a rather general question:
If a cover $X^{\bar{\nu}}$ has finite Betti numbers,
does $X^{\nu}$ also have finite Betti numbers?
This question can be answered by comparing the
generalized Dwyer--Fried invariants of $X$ with
their classical counterparts.

As before, let $A$ be a quotient of $H=H_1(X,\Z)$, and set
$r=\rank A$ and $n=\rank H$. The canonical projection
$q\colon \Gamma(H,A) \to \Gamma(H,\overline{A})$
restricts to a map $\Omega^{i}_{A}(X) \to \Omega^{i}_{\overline{A}}(X)$
between the respective Dwyer--Fried sets, thus yielding a
a commuting diagram,
\begin{equation}
\label{comm diag}
\xymatrix{ \Omega_A^i(X)\, \ar@{^{(}->}[r] \ar[d]^{q|_{\Omega_A^i(X)}}
& \Gamma(H, A)
\cong \GL_{n}(\Z) \times_{\Par} \Gamma\ar[d]^q\\
\Omega^i_{\,\oA}(X)\, \ar@{^{(}->}[r] &
\Gamma(H, \oA) \cong \Grass_{n-r}(\Z^n)
}
\end{equation}

If $q^{-1}\big(\Omega^i_{\oA}(X)\big)=\Omega^i_{A}(X)$,
then the finiteness of the Betti numbers of an $A$-cover can be
tested through the corresponding $\oA$-cover.  In general,
though, diagram \eqref{comm diag} is {\em not}\/ a pullback
diagram; in that case, the generalized Dwyer--Fried
invariants, when viewed as homotopy type invariants,
contain more information than the classical ones.

This dichotomy is illustrated by the following result.

\begin{thm}[Propositions~\ref{prop:pull-back2} and
\ref{prop:transverse}]
\label{thm:intro e}
Suppose the characteristic variety $V^i(X)$ is of the form
$\bigcup_{j} \rho_j T_j$, where each $T_j\subset \widehat{H}$
is an algebraic subgroup, and each $\bar{\rho}_j\in \widehat{H}/T_j$
has finite order.
\begin{romenum}
\item \label{d1}
If $\ord(\bar{\rho}_j)$ is coprime to the order of $\Tors(A)$, for each $j$,
then $q^{-1}\big(\Omega^i_{\oA}(X)\big)=\Omega^i_{A}(X)$.
\item \label{d2}
If the identity component of $T_1$ is
not contained in $\prod_{j\neq 1} T_j $, the order of $\rho_1$ divides $c(A)$,
and $\rank A < \rank H -\dim \prod_{j\neq 1} T_j$,
then $q^{-1}\big(\Omega^i_{\oA}(X)\big)\supsetneqq \Omega^i_{A}(X)$.
\end{romenum}
\end{thm}

In other words, suppose $X^{\bar\nu}$ is a regular $\oA$-cover of $X$,
with finite Betti numbers up to degree $i$.  Then, if condition \eqref{d1} holds,
all $A$-covers $X^{\nu}$ have the same finiteness property, whereas if
condition \eqref{d2} holds, one of those $A$-covers will have an infinite
Betti number in some degree less or equal to $i$.

\subsection{Toric complexes}
\label{subsec:apps toric}

A nice class of spaces to which our  theory applies is that of toric
complexes. Every simplicial complex $L$ on $n$ vertices determines
a subcomplex $T_L$ of the $n$-torus, with fundamental group the
right-angled Artin group associated to the $1$-skeleton of $L$.
Identify the group $H=H_1(T_L,\Z)$ with $\Z^n$.  Work of Papadima
and Suciu \cite{PS-toric} shows that the characteristic varieties
$V^i(T_L)$ are unions of coordinate subspaces in $\wH=(\C^*)^n$;
equations for these subspaces can be read off directly
from the simplicial complex $L$.

Formula \eqref{eq:intro qoa} computes all the generalized
Dwyer--Fried sets of a toric complex.  More precisely, if $A$ is
a quotient of $H$, then, as noted in Corollary~\ref{cor: toric_formula},
the complement of $\Omega_A^i(T_L)$ in $\Gamma(H,A)$ fibers
over the special Schubert varieties associated to the coordinate
subspaces comprising $V^i(T_L)$, with each fiber isomorphic to
$\Gamma(H/\oA,A/\oA)$.
Thus, if $T_L^{\bar\nu}$ is a free abelian cover of $T_L$, with finite
Betti numbers up to some degree $i$, then all finite abelian covers
$T_L^{\nu}\to T_L^{\bar\nu}$ have the same homological finiteness property.

\subsection{Quasi-projective varieties}
\label{subsec:apps qp}

Another important class of spaces to which our methods apply
quite well  is that of smooth, quasi-projective varieties. For such
a space $X$, work of Arapura \cite{Ar} and others shows that
\begin{equation}
\label{eq:v1x}
V^1(X)=Z\cup \bigcup_{\xi\in \Lambda}V(\xi) \cup
\bigcup_{\xi\in \Lambda'} \big( V(\xi)\setminus V(\overline{\xi}) \big),
\end{equation}
where $Z$ is a finite set, and $\Lambda$ and $\Lambda'$ are
certain (finite) collections of subgroups of $H=H_1(X,\Z)$.

\begin{thm}[Theorem~\ref{thm:omega qp}]
\label{thm:intro f}
With notation as above, let $A$ be a quotient of $H$.  Then
\[
\Omega^1_A(X)=\Gamma(H,A)\setminus \Bigg( \bigcup_{\xi\in \Lambda}
q^{-1}\big(\sigma_{\overline{A}} (\overline\xi)\big) \cup
\bigcup_{\xi\in \Lambda'}
\Big( q^{-1}\big(\sigma_{\overline{A}} (\overline\xi)\big) \cap
\theta_A(\xi)\Big) \Bigg)\, ,
\]
where $\sigma_{\overline{A}} (\overline\xi)\subseteq\Grass_{n-r}(\Z^n)
\cong\Gamma(\overline{H},\overline{A})$ is a special Schubert variety,
and $\theta_A(\xi)$ consists of those $[\nu]\in \Gamma(H,A)$ for which
there is a subgroup $\xi\le \xi'\lvertneqq\overline\xi$ such that
$\overline\xi/\xi'$ is cyclic and
$\nu(x)\neq0$ for all $x\in\overline\xi\setminus\xi'$.
\end{thm}

Theorem \ref{thm:intro e}, part \eqref{d1} now shows the
following: if the order of $\overline{\xi}/\xi$ is coprime to
$c(A)$, for each $\xi \in \Lambda'$, then
$\Omega_A^1(X)=q^{-1}\big(\Omega_{\oA}^1(X)\big)$.

In general, though, such an equality does not hold.
Examples of this sort can be constructed using
quasi-homogeneous surfaces (with the singularity at the
origin removed), which are homotopy equivalent
to Brieskorn manifolds $M=\Sigma(a_1,\dots ,a_n)$.
For instance, if  $M=\Sigma(2,4,8)$,
then the universal free abelian cover of $M$ has finite $b_1$,
whereas the universal abelian cover of $M$ has infinite $b_1$.

\subsection{Organization of the paper}
\label{subsec:org}

This paper is organized as follows:

In \S\ref{sect:structure thm} and \S\ref{sect:reinterpret},
we describe the structure of the parameter set
$\Gamma(H,A)$ for regular $A$-covers of a finite,
connected CW-complex $X$ with $H_1(X,\Z)=H$,
while in \S\ref{sect:df} we define the generalized
Dwyer--Fried invariants $\Omega^i_A(X)$, and
study their basic properties.

In \S\ref{sect:inttt}, we review the Pontryagin correspondence
between subgroups of $H$ and algebraic subgroups of the
character group $\widehat{H}$, while in \S\ref{sect:gen tcone}
we associate to each subvariety $W\subset \widehat{H}$ a
family of subgroups of $H$ generalizing the exponential
tangent cone construction.  In \S\ref{sect:schubert} and
\S\ref{sect:alg df}, we introduce and study several subsets
of the parameter set $\Gamma(H,A)$, which may be viewed
as analogues of the special Schubert varieties and the
incidence varieties from classical algebraic geometry.

In \S\ref{sect:supports} we revisit the Dwyer--Fried theory
in the more general context of (not necessarily torsion-free)
abelian covers, while in \S\ref{sect:char-var} we show how
to determine the sets $\Omega_A^i(X)$ in terms of the jump
loci for homology in rank $1$ local systems on $X$.
In \S\ref{sect:classic df}, we compare the Dwyer--Fried invariants
$\Omega_A^i(X)$ with their classical counterparts,
$\Omega_{\overline{A}}^i(X)$, while in \S\ref{sect:rank1}
we discuss in more detail these invariants
in the case when $\rank A=1$.

Finally, in \S\ref{sect:translated} we study the situation when
all irreducible components of the characteristic varieties of $X$ are
(possibly translated) algebraic subgroups of the character group,
while in \S\ref{sect:quasi-proj} we consider the particular case
when $X$ is a smooth, quasi-projective variety.

\section{A parameter set for regular abelian covers}
\label{sect:structure thm}

We start by setting up a parameter set for regular covers of a
CW-complex, with special emphasis on the case when the
deck-transformation group is abelian.

\subsection{Regular covers}
\label{subsec:regcov}
Let $X$ be a connected CW-complex with finite $1$-skeleton.
Without loss of generality, we may assume $X$ has a single
$0$-cell, which we will take as our basepoint, call it $x_0$.
Let $G=\pi_1(X,x_0)$ be the fundamental group.
Since the space $X$ has only finitely many $1$-cells,
the group $G$ is finitely generated.

Consider an epimorphism $\nu\colon G \surj A$ from $G$ to a
(necessarily finitely generated) group $A$. Such an
epimorphism gives rise to a regular
cover of $X$, which we denote by $X^{\nu}$.
Note that $X^{\nu}$ is also a connected CW-complex,
the projection map $p\colon X^{\nu} \to X$ is cellular,
and the group $A$ is the group of deck transformations
of $X^{\nu}$.

Conversely, every (connected) regular cover $p\colon (Y,y_0)\to (X,x_0)$
with group of deck transformations $A$ defines a normal subgroup
$p_{\sharp}(\pi_1(Y,y_0)) \triangleleft \pi_1(X,x_0)$, with quotient
group $A$.   Moreover, if $\nu\colon G \surj A$ is the projection
map onto the quotient, then the cover $Y$ is equivalent to $X^{\nu}$,
that is, there is an $A$-equivariant homeomorphism
$Y\cong X^{\nu}$.

Let $\Epi(G, A)$ be the set of all epimorphisms from $G$ to $A$,
and let $\Aut(A)$ be the automorphism group of $A$. The
following lemma is standard.

\begin{lemma}
\label{lem:epi aut}
Let $G=\pi_1(X,x_0)$, and let $A$ be a group.  The set of
equivalence classes of connected, regular $A$-covers of $X$
is in one-to-one correspondence with
\[
\Gamma(G,A):=\Epi(G,A)/\Aut(A),
\]
the set of equivalence classes $[\nu]$ of epimorphisms
$\nu \colon G \surj A$, modulo the right-action of $\Aut(A)$.
\end{lemma}

When the group $A$ is finite, the parameter set $\Gamma(G, A)$
is of course also finite.  Efficient counting methods for determining
the size of this set were pioneered by P.~Hall in the 1930s.
New techniques (involving, among other things, characteristic
varieties over finite fields) were introduced in \cite{MS02}.
In the particular case when $A$ is a finite abelian group, a closed
formula for the cardinality of $\Gamma(G, A)$ was given in
\cite{MS02}, see Theorem \ref{thm:ms} below.

\subsection{Functoriality properties}
\label{subsec:funct}
The above construction enjoys some (partial) functoriality
properties in both arguments. First, suppose that
$\varphi\colon G_1 \surj G_2$ is an epimorphism
between two groups. Composition with $\varphi$
gives a map $\Epi(G_2, A)\to \Epi(G_1, A)$, which
in turn induces a well-defined map,
\begin{equation}
\label{eq:fsharp}
\xymatrix{\Gamma(G_2, A)\ar[r]^{\varphi^*}
&  \Gamma(G_1, A)},
\quad [\nu] \mapsto [\nu\circ \varphi].
\end{equation}

Under the correspondence from Lemma \ref{lem:epi aut},
this map can be interpreted as follows.  Let
$f\colon (X_1,x_1) \to (X_2,x_2)$ be a basepoint-preserving
map between connected CW-complexes, and suppose
$f$ induces an epimorphism
$\varphi= f_{\sharp}\colon G_1\surj G_2$ on fundamental groups.
Then the map $\varphi^*$ sends the
equivalence class of the cover $p_{\nu}\colon X_2^{\nu}\to X_2$
to that of the pull-back cover,
$p_{\nu\circ \varphi}=f^*(p_\nu)\colon X_1^{\nu\circ \varphi} \to X_1$.

Next, recall that a subgroup $K<A$ is  {\em characteristic}\/ if
$\alpha(K)=K$, for all $\alpha\in \Aut(A)$.

\begin{lemma}
\label{lem:funct}
Suppose we have an exact sequence of groups,
$1\to K\to A\xrightarrow{\pi} B\to 1$, with $K$ a characteristic
subgroup of $A$.  There is then a well-defined map between the
parameter sets for regular $A$-covers and $B$-covers
of $X$,
\begin{equation}
\label{eq:functab}
\xymatrix{
\Gamma(G, A)  \ar[r]^{\tilde\pi}& \Gamma(G, B)
},
\end{equation}
which sends $[\nu]$ to $[\pi\circ \nu]$.
\end{lemma}

\begin{proof}
Suppose $\nu_1, \nu_2\colon G \surj A$ are two
epimorphisms so that $\alpha\circ \nu_1=\nu_2$, for some
$\alpha \in \Aut(A)$.  Since the subgroup $K=\ker(\pi)$ is
characteristic, the automorphism $\alpha$ induces and
automorphism $\bar\alpha\in \Aut(B)$ such that
$\bar\alpha\circ \pi=\pi\circ\alpha$.  Hence,
$\bar\alpha \circ (\pi \circ \nu_1) =\pi\circ \nu_2$,
showing that $q$ is well-defined.
\end{proof}

The correspondence of Lemmas \ref{lem:epi aut} and
\ref{lem:funct} is summarized in the following diagram:
\begin{equation}
\label{eq:ab}
\xymatrix{
  G \ar[dr]_(.45){\pi\circ \nu} \ar[r]^{\nu}
                & A \ar[d]^(.45){\pi}  \\
                & B
}
\qquad \longleftrightarrow \qquad
 \xymatrix{
  X^{\nu} \ar[r]^{p_\pi} \ar[dr]_(.45){p_{\nu}}
  & X^{\pi\circ \nu} \ar[d]^(.45){p_{\pi\circ \nu}}\\  & X
}
\end{equation}
Under this correspondence, the map $\tilde\pi$ sends the
equivalence class of the cover $p_{\nu}$ to that of the
cover $p_{\pi\circ \nu}$.

\subsection{Regular abelian covers}
\label{subsec:regabelcov}

Let $H=G^{\ab}$ be the abelianization of our group $G$.
Recall we are assuming $G$ is finitely generated; thus,
$H$ is a finitely generated abelian group.

Now suppose $A$ is any other (finitely generated) abelian group.
In this case, every homomorphism $G\to A$ factors through the
abelianization map, $\ab\colon G\surj H$.  Composition
with this map gives a bijection between $\Epi(H, A)$ and
$\Epi(G, A)$, which induces a bijection
\begin{equation}
\label{eq:epiha}
\xymatrix{\Gamma(H, A)\ar[r]^{\cong} &  \Gamma(G, A)},
\quad [\nu] \mapsto [\nu\circ \ab].
\end{equation}

In view of Lemma \ref{lem:epi aut}, we obtain
a bijection between the set of equivalence classes of
connected, regular $A$-covers of a CW-complex $X$
and the set  $\Gamma(H, A)$, where $H=H_1(X,\Z)$
is the abelianization of $G=\pi_1(X,x_0)$.  Note that this
parameter set is empty, unless $A$ is a quotient of $H$.

Let $\T(A)$ be the torsion subgroup, consisting of finite-order
elements in $A$; clearly, this is a characteristic subgroup of $A$.
Let $\oA=A/\T(A)$ be the quotient group, and let
$\pi\colon A \to \oA$ be the canonical projection.

Under the correspondence from \eqref{eq:epiha},
an epimorphism $\nu\colon H\surj A$ determines a regular
cover, $p_{\nu\circ \ab} \colon X^{\nu\circ \ab}\to X$,
which, for economy of notation, we will write as
$p_\nu\colon X^{\nu}\to X$.  There is also a free
abelian cover, $p_{\bar{\nu}}\colon X^{\bar{\nu}}\to X$,
corresponding to the epimorphism
$\bar{\nu}=\pi\circ \nu\colon H\to \oA$.

The projection $\pi\colon A\surj \oA$ defines a map
$\Epi(H, A) \to \Epi(H, \oA)$,
$\nu\mapsto \bar\nu$, which in turn induces maps
between the parameter spaces for $A$ and $\oA$ covers,
\begin{equation}
\label{eq:epimap}
\xymatrix{\Gamma(H, A) \ar[r]^{q_H}& \Gamma(H, \oA)},
\end{equation}
sending the cover $p_{\nu}$ to the cover $p_{\bar\nu}$.
Notice that this map is compatible with the morphism $\tilde\pi$
from \eqref{eq:functab}, induced by an epimorphism $\pi\colon A\surj B$
with characteristic kernel.

\subsection{Splitting the torsion-free part}
\label{subsec:split th}
For a topological group $G$, let $G\to EG\to BG$ be the
universal principal $G$-bundle; the total space $EG$ a contractible
CW-complex endowed with a free $G$-action, while the base space
$BG$ the quotient space under this action.   We will only
consider here the situation when $G$ is discrete, in which case
$BG=K(G,1)$ and $EG=\widetilde{K(G,1)}$.

As before, let $H$ be a finitely generated abelian group.
Let $\T(H)$ be its torsion subgroup, and let $\oH=H/\T(H)$.
Fix a splitting $\oH \to H$; then $H \cong \oH
\oplus \T(H)$.

Now consider an epimorphism $\nu\colon H\surj A$.  After
fixing a splitting $\oA \inj A$, we obtain a decomposition
$A \cong \oA \oplus \T(A)$.  We may view $\oA$ as a
subgroup of $H$ by choosing a splitting $\oA \inj H$ of
the projection $H \surj \oA$.  With these identifications,
$\nu$ induces an epimorphism $\tilde\nu\colon H/\oA\surj A/\oA$.
This observation leads us to consider the set
\begin{equation}
\label{eq:GammaHA}
\Gamma=\Gamma(H/\oA, A/\oA).
\end{equation}

Clearly, the set $\Gamma$ is finite.  Theorem 3.1 from \cite{MS02}
yields an explicit formula for the size of this set.  Given a finite
abelian group $K$, and a prime $p$, write the $p$-torsion part of $K$
as $K_p=\Z_{p^{\lambda_1}}\oplus\cdots\oplus\Z_{p^{\lambda_s}}$,
for some positive integers $\lambda_1\ge \dots \ge \lambda_s$,
where $s=0$ if $p\nmid \abs{K}$. Thus, $K_p$ determines a
partition $\pi(K_p)=(\lambda_1, \dots, \lambda_s)$.
For each such partition $\lambda$, write $l(\lambda)=s$,
$\abs{\lambda}=\sum_{i=1}^{s} \lambda_i$, and
$\langle\lambda\rangle=\sum_{i=1}^{s}(i-1)\lambda_i$.

\begin{theorem}[\cite{MS02}]
\label{thm:ms}
Set $n=\rank H$ and $r=\rank A$.
For each prime $p$ dividing the order of $A/\oA$,
let $\lambda=\pi((A/\oA)_p)$ and $\tau=\pi((\Tors(H/\oA))_p)$
be the corresponding partitions.  Then,
\begin{equation*}
\label{eq:absGamma}
\abs{\Gamma(H/\oA, A/\oA)} =\prod_{p\, \mid\, \abs{A/\oA}} \frac
{\DS{p^{(\abs{\lambda}-l(\lambda)) (n-r) + \theta(\lambda^{-},\tau)}
\prod_{i=1}^{l(\lambda)}
\left( p^{n-r+\theta_i(\lambda,\tau)-\theta_i(\lambda^{-},\tau)}-p^{i-1}\right)}}
{\DS{p^{\abs{\lambda}+2\langle \lambda \rangle}
\prod_{k\ge 1}\varphi_{m_k(\lambda)} (p^{-1})}},
\end{equation*}
where $m_k(\lambda)=\#\{ i\mid \lambda_i=k\}$,
$\varphi_m(t)=\prod_{i=1}^{m} (1-t^i)$,
$\lambda^{-}$ is  the partition with $\lambda^{-}_i=\lambda_i-1$,
$\theta_i(\lambda,\tau)=\sum_{j=1}^{l(\tau)}\min(\lambda_i,\tau_{j})$, and
$\theta(\lambda,\tau)=\sum_{i=1}^{l(\lambda)}\theta_i(\lambda,\tau)$.
\end{theorem}

For instance, $\abs{\Gamma(\Z^n, \Z_{p^{s}})}=
\frac{p^{sn}-p^{(s-1)n}}{p^{s}-p^{s-1}}$,
whereas $\abs{\Gamma(\Z^n, \Z_{p}^{s})}=
\prod_{i=0}^{s-1}\frac{p^{n}-p^{i}}{p^{s}-p^{i}}$.

Each homomorphism $H/\oA\to A/\oA$ defines an action of
$H/\oA$ on $A/\oA$.  This action yields a fiber bundle,
\begin{equation}
\label{eq:assoc}
A/\oA \to E(H/A)\times_{H/\oA}A/\oA \to B(H/\oA)
\end{equation}
associated to the principal bundle $H/\oA \to E(H/\oA) \to B(H/\oA)$.
The set $\Gamma=\Gamma(H/\oA, A/\oA)$
parameterizes all such associated bundles.

\subsection{A pullback diagram}
\label{subsec:square}
We return now to diagram \eqref{eq:epimap}, which relates
the parameter sets for regular $A$ and $\oA$ covers of
a connected CW-complex $X$.
This diagram can be further analyzed by using pullbacks
from the universal principal bundles over the classifying
spaces for the discrete groups $A$ and $\oA$.

Let $A \to EA \to BA$ and $\oA\to E\oA \to B\oA$
be the respective classifying bundles, and let $X\to BA$ and $X\to B\oA$
be classifying maps for the covers $X^{\nu}\to X$ and $X^{\bar\nu}\to X$,
respectively.  Upon identifying  $\T(A)$ with $A/\oA$, we obtain the
following diagram:
\begin{equation}
\label{eq:bigcd}
\xymatrixcolsep{18pt}
\xymatrixrowsep{18pt}
\xymatrix@!0{
& & & &  & & &  \T(A) \ar[lldd] \ar@{=}[ddd] & & & &  \\
& & & &                  & & &                 & & & & \\
& & A \ar[lldd] \ar'[dd][ddd] & & & A  \ar@{=}[lll] \ar[lldd] \ar[ddd]& & & & & &  \\
& & & & & & & \T(A) \ar[lldd] \ar[rrrr]^{\cong} & & & & A/\oA \ar[lldd] \\
\oA \ar[ddd] & & & \oA \ar@{=}[lll] \ar[ddd]& & & & & & & & \\
& & EA \ar[lldd] \ar'[dd][ddd] & & & X^{\nu} \ar'[ll][lll] \ar[lldd]_{p_\pi}
\ar'[dd][ddd]_(.45){p_{\nu}} \ar@{-->}[rrrr]&
& & & E(H/A)\times_{H/\oA}A/\oA \ar[lldd]_{\beta} & & \\
& & & &                  & & \pullbackcorner &           & & & & \\
 E\oA \ar[ddd] & & & X^{\bar\nu}
 \ar[lll] \ar[ddd]_{p_{\bar\nu}} \ar[rrrr]^(.55){\alpha}
 & & & & B(H/\oA) \\
& & BA \ar[lldd] & & & X \ar'[ll][lll] \ar@{=}[lldd]& & & & & &  \\
& & & &                  & & &          & & & & \\
 B\oA & & & X \ar[lll]\ar[rrrr]& & & & BH \ar[uuu]& & & & \\
}
\end{equation}

Here, the map $X\to BH$ realizes the abelianization
morphism, $\ab\colon \pi_1(X,x_0)\to H$, while
$\alpha$ denotes the composite $X^{\bar\nu} \to X
\to BH\to B(H/\oA)$.

\begin{prop}
\label{prop:bha}
The marked square in diagram \eqref{eq:bigcd} is a pullback square.
That is, the cover $p_\pi\colon X^{\nu} \to X^{\bar\nu}$ is the
pullback along the map $\alpha\colon X^{\bar\nu} \to B(H/\oA)$
of the cover
$\beta\colon E(H/A)\times_{H/\oA}A/\oA \to B(H/\oA)$
corresponding to the epimorphism $\tilde\nu\colon H/\oA\surj A/\oA$.
\end{prop}

\begin{proof}
Clearly, $(\alpha\circ {p_\pi})_{\sharp} (\pi_1(X^\nu))=
\im ( \ker \nu \inj \ker \bar{\nu})$, while $\beta_{\sharp}
(\pi_1(E(H/A)\times_{H/\oA}A/\oA))=
\ker (\tilde{\nu}\colon H/\oA \surj A/\oA$).
After picking a splitting $\oA \inj H$, and
identifying the group $H$ with $\ker \bar{\nu} \oplus
\oA$, we see that
\[
(\alpha\circ {p_\pi})_{\sharp} (\pi_1(X^\nu))
\subseteq \beta_{\sharp}
(\pi_1(E(H/A)\times_{H/\oA}A/\oA)).
\]
The existence of the dashed arrow in the diagram follows then
from the lifting criterion for covers. It is readily seen that
this arrow is equivariant with respect to the  actions on
source and target by $\T(A)$ and $A/\oA$; thus,
a morphism of covers.  This completes the proof.
\end{proof}

Using this proposition, and the discussion from \S\ref{subsec:split th},
we obtain the following corollary.

\begin{corollary}
\label{cor:gha}
With notation as above,
\[
\xymatrix{
\Gamma(H/\oA, A/\oA) \ar[r] &
\Gamma(H, A) \ar[r]^{q_H} &  \Gamma(H, \oA)
}
\]
is a set fibration; that is, all fibers of $q_H$ are in bijection
with the set  $\Gamma(H/\oA, A/\oA)$.
\end{corollary}

Let us identify topologically the fiber of $q_H$.
A regular, $\oA$-cover of our space $X$
corresponds to an epimorphism $\bar{\nu}\colon H \surj \oA$.
A regular, $\T(A)$-cover of $X^{\bar{\nu}}$
corresponds to an epimorphism $\ker \bar{\nu} \surj \T(A)$. Given
these data, and the chosen splitting $\oA\inj A$,
we can find an epimorphism $\nu\colon H \surj A$,
such that the following diagram commutes:
\begin{equation}
\label{eq:cd tors}
\xymatrix{ \ker \bar{\nu}\: \ar@{^{(}->}[r] \ar@{->>}[d]
& H \ar@{->>}[r]^{\bar{\nu}} \ar@{-->}[d]^{\nu}
& \oA \ar@{=}[d]\\
 \T(A)\: \ar@{^{(}->}[r] & A \ar@{->>}[r]^{\pi} & \oA
}
\end{equation}

Thus, any regular, $\T(A)$-cover $X^{\nu}\to X^{\bar\nu}$ defines
a regular $A$-cover $X^{\nu} \to X$, whose corresponding free
abelian cover is $X^{\bar{\nu}}$.  Consequently, the fiber of
$[\bar{\nu}]$ under the map
$q_H\colon \Gamma(H, A)\to \Gamma(H, \oA)$
coincides with the set
\begin{equation}
\label{eq:fiber q}
\big\{\, [\nu] \in \Gamma(H, A)\mid \text{ $X^{\nu}$  is a regular
$\T(A)$-cover of $X^{\bar{\nu}}$} \big\}.
\end{equation}

\section{Reinterpreting the parameter set for $A$-covers}
\label{sect:reinterpret}
In this section, we give a geometric description of
the parameter set for regular abelian covers of a
space.

\subsection{Splittings}
\label{subsec:split}
As before, let $H$ and $A$ be finitely generated abelian groups,
and assume there is an epimorphism $H\surj A$.

\begin{lemma}
\label{lem:split}
The action $\Aut(A)$ on the set of all splittings $A/\T(A) \inj A$
induced by the natural action of $\Aut(A)$ on $A$ is transitive.
\end{lemma}

\begin{proof}
Set $\oA=A/\T(A)$, and fix a splitting $s\colon \oA  \inj A$.
Using this splitting, we may decompose the group $A$ as
$\oA \oplus \T(A)$, and view $s\colon \oA  \inj
\oA \oplus \T(A)$ as the map $a \mapsto (a,0)$.

An arbitrary splitting $\sigma\colon  \oA  \inj \oA  \oplus \T(A)$
is given by $a \mapsto (a, \sigma_2(a))$, for some homomorphism
$\sigma_2\colon \oA  \to \T(A)$.   Consider the automorphism
of $\alpha\in \Aut(\oA  \oplus \T(A))$ given by the matrix
$\left(\begin{smallmatrix}
\id & 0 \\
\sigma_2 & \id
\end{smallmatrix}\right)$.
Clearly, $\alpha\circ s=\sigma$, and we are done.
\end{proof}

Denote by $n$ the rank of $H$, and by $r$ the rank of
$A$.  Fixing splittings $\oH\inj H$ and  $\oA\inj A$,
we have $H=\oH\oplus \T(H)$, with
$\oH=\Z^n$, and $A=\oA\oplus \T(A)$, with
$\oA=\Z^r$.

Now identify the automorphism group $\Aut(\oH)$ with
the general linear group $\GL_n(\Z)$.  Let $\Par$ be the parabolic
subgroup of $\GL_n(\Z)$, consisting all matrices of the form
$\left(\begin{smallmatrix}
         *_1 & *_2 \\
         0 & *_3
\end{smallmatrix}\right)$,
where $*_1$ is of size $(n-r) \times (n-r)$. Then $\GL_n(\Z)/\Par$
is isomorphic to the Grassmannian $\Grass_{n-r}(\Z^n)$.
It is readily checked that the left action of $\Par$ on $\Z^{n-r}$,
given by multiplication of $\{*_1\} \cong \GL_{n-r}(\Z)$ on
$\Z^{n-r}$, induces an action of $\Par$ on the set
$\Gamma=\Gamma(H/\oA,A/\oA)$.
Also note that, if $A$ is torsion-free, then the set
$\Gamma$ is a singleton.

\subsection{A fibered product}
\label{subsec:bijection}
We are now ready to state and prove the main result
of this section.
\begin{theorem}
\label{theorem:new bij}
There is a bijection
\[
\Gamma(H,A) \longleftrightarrow \GL_{n}(\Z) \times_{\Par} \Gamma
\]
between the parameter set $\Gamma(H,A)=\Epi(H,A)/\Aut(A)$
and the twisted product of $\GL_{n}(\Z)$ with the set
$\Gamma=\Gamma(H/\oA,A/\oA)$ over the parabolic subgroup $\Par$.
Under this bijection, the map
$q \colon \Gamma(H, A) \to \Gamma(H, \oA)$ induced
by the projection $\pi\colon A\to \oA$ corresponds to the
canonical projection
\[
\GL_{n}(\Z) \times_{\Par} \Gamma \to \GL_{n}(\Z)/\Par =
\Grass_{n-r}(\Z^n).
\]
\end{theorem}

\begin{proof}
Define a map
$\theta\colon \GL_{n}(\Z) \times \Gamma \to \Gamma(H, A)$ as follows.
Given an element $(M, [\gamma])$ of $\GL_{n}(\Z) \times \Gamma$,
let $(\gamma_1, \gamma_2)$
be a representative of $[\gamma]$, with
$\gamma_1\colon \Z^{n-r} \surj \T(A)$ and
$\gamma_2\colon \T(H) \surj \T(A)$. Let $\alpha_1, \dots, \alpha_n$ be the column vectors
of the matrix $M$, which forms a basis of
$\oH \cong \Z^n$, we can write $H=\Z^{n-r} \oplus \Z^r
\oplus \T(H)$, where $\Z^{n-r}$ is the subspace of $\oH$ generated
by the first $n-r$ column vectors of $M$.
Now define  $\theta(M, [\gamma]) = [\nu]$, where
$\nu \colon \Z^{n-r} \oplus \Z^r \oplus \T(H) \to \Z^r \oplus \T(A)$
is the homomorphism given by the matrix $N=\left(\begin{smallmatrix}
  0 & \id & 0 \\
  \gamma_1 & 0 & \gamma_2
  \end{smallmatrix}\right)$.
It is straightforward to check that the map $\theta$ is well-defined,
i.e., $\theta$ is independent of the splitting and representative we chose.

Now let's check that the map $\theta$ factors through $\GL_{n}(\Z)
\times_{\Par} \Gamma$.  Suppose we have two elements
$(M, [\gamma_1, \gamma_2])$ and
$(M',[\gamma'_1, \gamma'_2])$ of $\GL_{n}(\Z) \times \Gamma$
which are equivalent, that is, there is a matrix
$Q=\left( \begin{smallmatrix} Q_1 & Q_2 \\
0 & Q_3  \end{smallmatrix}\right)\in \Par$ such that
$M=M' Q$ and $(\gamma_1 Q_1^{-1}, \gamma_2)=(\gamma_1', \gamma_2')$.
By definition, the map $\theta$ takes the pair $(M, [\gamma_1, \gamma_2])$
to the homomorphism $\nu$ given by the matrix $N$  above.
Changing the basis of $\oH$ to the basis given by the column
vectors of $M'$, the map $\nu$ is given by the matrix
\[
NQ^{-1} =
\begin{pmatrix}
0 & Q_3^{-1} & 0 \\
\gamma_1 Q_1^{-1} & \gamma_1Q_4 & \gamma_2
\end{pmatrix}
= \begin{pmatrix}
              Q_3^{-1} & 0 \\
              \gamma_1 Q_4 & \id
              \end{pmatrix}
           \begin{pmatrix}
                    0 & \id & 0 \\
                    \gamma_1' & 0 & \gamma_2'
 \end{pmatrix}.
\]
Clearly, the matrix $\left(\begin{smallmatrix}
             Q_3^{-1} & 0 \\
              \gamma_1 Q_4 & \id \end{smallmatrix}\right)$
defines an automorphism of $A$.  Thus, the map $\theta$ factors
through a well-defined map,
$\theta\colon \GL_{n}(\Z) \times_{\Par} \Gamma \to \Gamma(H, A)$,
which is readily seen to be a bijection.  It is now a straightforward
matter to verify the last assertion.
\end{proof}

\subsection{Further discussion}
\label{subsec:further}
A particular case of the above theorem is worth singling out.
Recall $\oH=\Z^n$ and $\oA=\Z^r$.

\begin{corollary}
\label{cor:pgrass}
Suppose $\T(H)=\Z_p^s$ and $\T(A)=\Z_p^t$, for some prime $p$.
Then, the parameter set $\Gamma(H, A)$ is in bijective correspondence
with $\GL_{n}(\Z) \times_{\Par} \Grass_t(\Z_p^{n-r+s})$.
\end{corollary}

\begin{proof}
In this case, the set $\Gamma=\Gamma(H/\oA,A/\oA)$
is in bijection with
$\Epi(\Z_p^{n-r} \oplus \Z_p^s, \Z_p^t) /\Aut(\Z_p^t)$.
This bijection is established using the diagram
\begin{equation}
\xymatrix{
\Z^{n-r} \oplus \Z_p^s \ar@{->>}[d] \ar@{->>}[r] & \Z_p^t \\
             \Z_p^{n-r} \oplus \Z_p^s \ar@{->>}[ur] }
\end{equation}
Therefore, $\Gamma= \Grass_t(\Z_p^{n-r+s})$, and we are done.
\end{proof}

\begin{remark}
\label{rem:grass}
Consider the projection map $q\colon \GL_{n}(\Z)
\times_{\Par} \Gamma \to \GL_{n}(\Z)/\Par = \Grass_{n-r}(\Z^n)$
from Theorem \ref{theorem:new bij}. It is readily seen that, for each
subspace $Q \in \Grass_{n-r}(\Z^n)$, the cardinality of the fiber
$q^{-1}(Q)$ is the same as the the cardinality of the set
$\Gamma$. In particular, $q^{-1}(Q)$ is finite,
for all $Q \in \Grass_{n-r}(\Z^n)$.
\end{remark}

\section{Dwyer--Fried sets and their generalizations}
\label{sect:df}

In this section, we define a sequence of subsets $\Omega^i_A(X)$
of the parameter set for regular $A$-covers of $X$. These sets, which
generalize the Dwyer--Fried sets $\Omega^i_r(X)$, keep
track of the homological finiteness properties of those covers.

\subsection{Generalized Dwyer--Fried sets}
\label{subsec:abel covers}
Throughout this section, $X$ will be a connected CW-complex
with finite $1$-skeleton, and $G=\pi_1(X,x_0)$ will denote its
fundamental group.

\begin{definition}
\label{def:gral df}
For each group $A$ and integer $i\ge 0$, the corresponding
{\em Dwyer--Fried set}\/ of $X$ is defined as
\[
\Omega^{i}_{A}(X)= \big\{\, [\nu] \in \Gamma(G, A) \mid
b_{j} (X^{\nu})< \infty, \text{ for all $0 \leq j \leq i$ } \big\}.
\]
\end{definition}

In other words, the sets $\Omega^{i}_{A}(X)$ parameterize those
regular $A$-covers of $X$ having finite Betti numbers up to
degree ~$i$.  In the particular case when $A$ is a free abelian
group of rank $r$, we recover the standard Dwyer--Fried sets,
$\Omega^i_r(X)=\Omega^i_{\Z^r}(X)$, viewed as subsets
of the Grassmannian $\Grass_r(\Q^n)$, where $n=b_1(X)$.

By our assumption on the $1$-skeleton on $X$, the group
$G=\pi_1(X,x_0)$ is finitely generated.  Thus, we may assume
$A$ is also finitely generated, for otherwise $\Epi(G, A)=\emptyset$,
and so $\Omega^{i}_{A}(X)=\emptyset$, too.

The $\Omega$-sets are invariant under homotopy
equivalence. More precisely, we have the following lemma,
which generalizes the analogous lemma for free abelian
covers, proved in \cite{Su14}.

\begin{lemma}
\label{lem:homotopy equiv}
Let $f\colon X\to Y$ be a (cellular) homotopy equivalence.
For any group $A$, the homomorphism
$f_{\sharp}\colon \pi_1(X,x_0) \to \pi_1(Y,y_0)$
induces a bijection
$f_{\sharp}^* \colon  \Gamma(\pi_1(Y,y_0), A)\to
\Gamma(\pi_1(X,x_0), A)$,
sending each subset $\Omega^i_A(Y)$ bijectively
onto $\Omega^i_A(X)$.
\end{lemma}

\begin{proof}
Since $f$ is a homotopy equivalence, the induced homomorphism
on fundamental groups, $f_{\sharp}$, is a bijection.
Thus, the corresponding map between parameter sets,
$f_{\sharp}^*$, is a bijection.  To finish the proof, it remains
to verify that $f_{\sharp}^*(\Omega^i_A(Y))= \Omega^i_A(X)$.

Let $\nu\colon \pi_1(Y,y_0) \surj A$ be an epimorphism.
Composing with $f_{\sharp}$, we get an epimorphism
$\nu \circ f_{\sharp}\colon \pi_1(X,x_0) \surj A$.
By the lifting criterion, $f$ lifts to a map $\bar{f}$ between
the respective $A$-covers. This map fits into the following
pullback diagram:
\begin{equation}
\label{eq:lift}
\xymatrix{
X^{\nu\circ f_{\sharp}} \ar^{\bar{f}}[r] \ar[d]& Y^{\nu} \ar[d] \\
X \ar^{f}[r] & Y .}
\end{equation}

Clearly, $\bar{f}\colon X^{\nu\circ f_{\sharp}} \to Y^{\nu}$
is also a homotopy equivalence.
Thus, $b_j(Y^{\nu})< \infty$ if and only if $b_j(X^{\nu\circ f_{\sharp}})< \infty$,
which means that $f_{\sharp}^*(\Omega^i_A(Y))= \Omega^i_A(X)$.
\end{proof}

Based on this lemma, we may define the $\Omega$-sets
of a (discrete, finitely generated) group $G$ as
$\Omega^i_A(G):=\Omega^i_A(BG)$, where $BG=K(G,1)$
is a classifying space for $G$.

\subsection{Naturality properties}
\label{subsec:natural}

The Dwyer--Fried sets (or their complements) enjoy certain
naturality properties in both variables, which we now describe.

\begin{prop}
\label{prop:surj omega}
Let $\varphi\colon G \surj Q$ be an epimorphism of groups. Then,
for each group $A$, there is an inclusion
$\Omega^1_A(Q)^c \inj \Omega^1_A(G)^c$.
\end{prop}

\begin{proof}
Let $\nu\colon Q \surj A$ be an epimorphism. Composing with $\varphi$,
we get an epimorphism $\nu \circ \varphi\colon G \surj A$. So there is an
epimorphism $\varphi\colon \ker(\nu \circ \varphi) \surj \ker(\nu)$.
Taking abelianizations, we get an epimorphism
$\varphi_{\ab}\colon\ker(\nu \circ \varphi)_{\ab} \surj \ker(\nu)_{\ab}$.
Thus, if $\ker(\nu)_{\ab}$ has infinite rank, then
$\ker(\nu \circ \varphi)_{\ab}$ also has infinite
rank.  The desired conclusion follows.
\end{proof}

Before proceeding, let us recall a well-known result regarding the
homology of finite covers, which can be proved via a standard transfer
argument (see, for instance, \cite{Hatcher}).

\begin{lemma}
\label{lem:transfer}
Let $p\colon Y \to X$ be a regular cover defined
by a properly discontinuous action of a finite group
$A$ on $Y$, and let $\k$ be a coefficient field of
characteristic $0$, or a prime not dividing the order of $A$.
Then, the induced homomorphism in cohomology,
$p^*\colon H^{*} (X; \k) \to H^{*} (Y; \k)$, is injective,
with image the subgroup $H^*(Y; \k)^{A}$ consisting of those
classes $\alpha$ for which $\gamma^*(\alpha)=\alpha$, for
all $\gamma \in A$.
\end{lemma}

\begin{corollary}
\label{cor:transfer}
Let $p\colon Y \to X$ be a finite, regular cover.
Then $b_{i}(X)\leq b_{i}(Y)$, for all $i\ge 0$.
\end{corollary}

Now fix a CW-complex $X$ as above, with fundamental group
$G=\pi_1(X,x_0)$.  Suppose $1 \to K  \to A \xrightarrow{\pi} B \to 1$
is a short exact sequence of groups, with $K$ a characteristic
subgroup of $A$.  As noted in Lemma \ref{lem:funct}, the
homomorphism $\pi$ induces a map
$\tilde{\pi} \colon \Gamma(G, A) \to \Gamma(G, B)$,
$[\nu]\mapsto [\pi\circ \nu]$, between the parameter sets for
regular $A$-covers and $B$-covers of $X$.

\begin{prop}
\label{prop:natb}
Suppose $K=\ker(\pi\colon A\surj B)$ is a finite, characteristic
subgroup of $A$.  Then the map
$\tilde{\pi} \colon \Gamma(G, A) \to \Gamma(G, B)$
restricts to a map
$\tilde{\pi} \colon \Omega^{i}_{A}(X) \to \Omega^{i}_{B}(X)$
between the respective Dwyer--Fried sets.
\end{prop}

\begin{proof}
Let $\nu\colon G \surj A$ be an epimorphism,
and suppose $[\nu]\in \Omega^{i}_{A}(X)$, that is,
$b_j(X^{\nu})<\infty$, for all $j\le i$. Then
$X^{\nu}\to X^{\pi\circ \nu}$ is a regular $K$-cover.
By Corollary \ref{cor:transfer}, we have that $b_j(X^{\pi\circ \nu})<\infty$,
for all $j\le i$; in other words, $[\pi\circ \nu]\in \Omega^{i}_{B}(X)$.
\end{proof}

Proposition \ref{prop:natb} may be summarized in the following
commuting diagram:
\begin{equation}
\label{eq:omega ab}
\xymatrix{ \Omega_A^i(X)\, \ar@{^{(}->}[r] \ar[d]^{\tilde\pi |_{\Omega_A^i(X)}} &
\Gamma(G, A)\ar[d]^{\tilde\pi}\\
\Omega^i_{B}(X)\, \ar@{^{(}->}[r] &
\Gamma(G, B)}
\end{equation}
This diagram is a pullback diagram precisely when
\begin{equation}
\label{eq:pull}
\tilde{\pi}^{-1}(\Omega^i_{B}(X)) = \Omega^i_A(X).
\end{equation}

As we shall see later on, this condition is not always satisfied.
For now, let us just single out a simple situation
when \eqref{eq:omega ab} is tautologically
a pullback diagram.

\begin{corollary}
\label{cor:ab empty}
With notation as above, if\/ $\Omega^{i}_{B}(X)=\emptyset$,
then $\Omega^{i}_{A}(X)=\emptyset$.
\end{corollary}

\subsection{Abelian versus free abelian covers}
\label{subsec:fab-df}
Let us now consider in more detail the case when $A$ is
an abelian group.  As usual, we are only interested in the
case when $A$ is a quotient of the (finitely generated)
group $H=G_{\ab}$, and thus we may assume $A$ is
also finitely generated.

Consider the exact sequence $0\to \T(A)\to A \to \oA \to 0$.
Clearly, $\T(A)$ is a finite, characteristic subgroup of $A$.
Thus, Proposition \ref{prop:natb} applies, giving a map
\begin{equation}
\label{eq:qmap}
q\colon \Omega^i_{A}(X) \to \Omega^i_{\oA}(X).
\end{equation}

In particular, if $\Omega^{i}_{\oA}(X)=\emptyset$, then
$\Omega^{i}_{A}(X)=\emptyset$.

\begin{example}
\label{ex:surfaces}
Let $\Sigma_g$ be a Riemann surface of genus $g\ge 2$.
It is readily seen that $\Omega^{i}_{r}(\Sigma_g)=\emptyset$,
for all $r\ge 1$ and $i\ge 1$, cf.~\cite{Su14}.  Thus, if $A$ is
any finitely generated abelian group with $\rank A \geq 1$, then
$\Omega^{i}_{A}(\Sigma_g)=\emptyset$, for all $i\ge 1$.
\end{example}

Suppose now we have a short exact sequence
$1 \to K  \to A \xrightarrow{\pi} B \to 1$, with $K$ characteristic.
Let  $\bar\pi\colon \oA \surj \overline{B}$ be the induced
epimorphism between maximal torsion-free
quotients.  Since $K=\ker(\pi)$ is finite, $\bar\pi$ is an isomorphism.
Using Proposition \ref{prop:natb} again, and the identification
from \eqref{eq:epiha}, we obtain the following commutative diagram:
\begin{equation}
\label{eq:natural diag}
\xymatrixcolsep{8pt}
\xymatrixrowsep{12pt}
\xymatrix{
&& \Omega_B^i(X)\ar[rr] \ar'[d][dd] & &
\Gamma(H, B) \ar[dd]^(.38){q_B}\\
\Omega_A^i(X) \ar[rrr] \ar[rru] \ar[dd]  & & &
\Gamma(H, A)\ar[ru]^(.42){\tilde{\pi}} \ar[dd]^(.38){q_A} \\
&& \Omega_{\overline{B}}^i(X) \ar'[r][rr] & &
\Gamma(H, \overline{B}) \\
\Omega_{\oA}^i(X) \ar[rrr] \ar@{=}[rru] & & &
\Gamma(H, \oA)\ar@{=}[ru]^{\tilde{\bar\pi}}
}
\end{equation}

\begin{prop}
\label{cyclic to arbitrary}
Assume the function
$\tilde{\pi} \colon \Gamma(H, A) \to \Gamma(H, B)$ is surjective.
Then, if the front square in diagram (\ref{eq:natural diag}) is a pullback
square, so is the back square; that is,
\[
q_A^{-1}\Big(\Omega_{\oA}^i\Big)=\Omega_A^i \implies
q_B^{-1}\Big(\Omega_{\overline{B}}^i\Big)=\Omega_B^i.
\]
\end{prop}

\begin{proof}
Suppose the back square is not pullback square.  Then there exist
elements $[\bar{\nu}] \in \Omega^i_{\overline{B}}$ and
$[\nu] \in  \Gamma(H, B) \setminus \Omega^i_{B}$
such that $q_B([\nu])=\bar{\nu}$. By assumption, the map $\tilde{\pi}$
is surjective; thus, $\tilde{\pi}^{-1}([\nu])$ is nonempty.
Pick an element $[\sigma] \in \tilde{\pi}^{-1}([\nu])$. Then
$[\sigma] \in \Gamma(H,A)\setminus \Omega_A^i$, for otherwise
$[\nu]=\tilde{\pi}([\sigma]) \in \Omega^i_B$.
On the other hand,
\[
q_A([\sigma])=q_B(\tilde{\pi}([\sigma]))=q_B([\nu])=
[\bar{\nu}] \in \Omega^i_{\overline{B}}=\Omega^i_{\oA}\, .
\]
Thus, the front square is not a pullback diagram, either.
\end{proof}

\subsection{The comparison diagram}
\label{subsec:compare}
Now fix a splitting $\oA \inj A$, which gives rise to an isomorphism
$A \cong \oA \oplus \T(A)$.  Similarly, after fixing a splitting $\oH \inj H$,
the abelianization $H=G_{\ab}$ also decomposes as
$H \cong \oH\oplus \T(H)$. Theorem \ref{theorem:new bij} yields
an identification
\begin{equation}
\label{eq:isos}
\Gamma(H, A) \cong \GL_{n}(\Z) \times_{\Par} \Gamma,
\end{equation}
where $n=\rank H$, the group $\Par$ is a parabolic subgroup
of $\GL_{n}(\Z)$ so that $\GL_{n}(\Z)/\Par=\Grass_{n-r}(\Z^n)$,
and $\Gamma=\Gamma(H/\oA, A/\oA)$.

Putting things together, we obtain a commutative diagram,
which we shall refer to as the {\em comparison diagram},
\begin{equation}
\label{commutative diagram}
\xymatrix{ \Omega_A^i(X)\, \ar@{^{(}->}[r] \ar[d]^{q|_{\Omega_A^i(X)} }
& \Gamma(H, A)
\cong \GL_{n}(\Z) \times_{\Par} \Gamma\ar[d]^{q}\\
\Omega^i_{\,\oA}(X)\, \ar@{^{(}->}[r] &
\Gamma(H, \oA) \cong \Grass_{n-r}(\Z^n)
}
\end{equation}

The next result reinterprets the condition that this diagram
is a pull-back  in terms of Betti numbers of abelian covers.

\begin{prop}
\label{prop:pullbetti}
The following conditions are equivalent:

\begin{romenum}
\item \label{c1}
Diagram \eqref{commutative diagram} is a
pull-back diagram.
\item \label{c2}
$q^{-1}\big(\Omega^i_{\,\oA}(X)\big) = \Omega_A^i(X)$.
\item \label{c3}
If $X^{\bar{\nu}}$ is a regular $\oA$-cover with finite
Betti numbers up to degree $i$, then any regular
$\T(A)$-cover of $X^{\bar{\nu}}$ has the same finiteness
property.
\end{romenum}
\end{prop}

\begin{proof}
The equivalence \eqref{c1} $\Leftrightarrow$ \eqref{c2}
is immediate.  To prove \eqref{c2} $\Leftrightarrow$ \eqref{c3},
consider an epimorphism $\nu\colon G \surj A$, and let
$\bar{\nu}=\pi\circ \nu\colon G \surj \oA$.  We
know from \eqref{eq:fiber q} that $q^{-1}([\bar\nu])$
coincides with the set of equivalence classes of
regular $\T(A)$-covers $X^{\nu}\to X^{\bar{\nu}}$.
The desired conclusion follows.
\end{proof}

In other words, \eqref{commutative diagram} is a
pull-back diagram if and only if the homological finiteness
of an arbitrary abelian cover of $X$ can be tested through the
corresponding free abelian cover.

\section{Pontryagin duality}
\label{sect:inttt}

Following the approach from \cite{Hi, SYZ}, we now discuss
a functorial correspondence between finitely generated
abelian groups and abelian, complex algebraic reductive groups.

\subsection{A functorial correspondence}
\label{subsec:functor}

Let $\C^*$ be the multiplicative group of units in the field of complex
numbers.  Given a group $G$, let $\wG=\Hom(G, \C^{*})$ be the group
of complex-valued characters of $G$, with pointwise multiplication
inherited from $\C^*$, and identity the character taking constant
value $1 \in \C^*$ for all $g \in G$. If the group $G$ is finitely generated,
then $\wG$ is an abelian, complex reductive algebraic group.
Given a homomorphism $\varphi \colon G_1\to G_2$, let
$\hat{\varphi}\colon \widehat{G}_2 \to \widehat{G}_1$,
$\rho\mapsto \rho\circ \varphi$ be
the induced morphism between character groups.
Since the group $\C^*$ is divisible, the functor
$G\leadsto \wG=\Hom(H,\C^*)$ is exact.

Now let $H=G_{\ab}$ be the maximal abelian quotient of $G$.
The abelianization map, $\ab\colon G\to H$, induces an
isomorphism $\widehat{\ab} \colon \wH\isom \wG$.
If $H$ is torsion-free, then $\wH$ can be identified
with the complex algebraic torus $(\C^{*})^{n}$,
where $n=\rank(H)$.  If $H$ is a finite abelian group,
then $\wH$ is, in fact, isomorphic to $H$.

More generally,
let $\oH$ be the maximal torsion-free quotient of $H$.
Fixing a splitting $\oH \to H$ yields a decomposition
$H\cong \oH\oplus \Tors(H)$, and thus an isomorphism
$\wH \cong \widehat{\oH} \times \Tors(H)$.
For simplicity, write $T=\wH$, and $T_0$ for the
identity component of this abelian, reductive, complex algebraic group;
clearly, $T_0=\widehat{\oH}$ is an algebraic torus.

Conversely, we can associate to $T$ its weight group,
$\check{T}=\Hom_{\alg}(T,\C^*)$,
where the hom set is taken in the category of algebraic groups.
The (discrete) group $\check{T}$ is a finitely generated abelian group.
Let $\C[\check{T}]$ be its group algebra.  We then have
natural identifications,
\begin{equation}
\label{eq:check}
\Spm\, (\C[\check{T}]) = \Hom_{\alg}(\C[\check{T}], \C)
= \Hom_{\group}(\check{T}, \C^*)= T.
\end{equation}

The correspondence $H \leftrightsquigarrow T$ extends
to a duality
\begin{equation}
\label{eq:pont dual}
\xymatrix{
\text{Subgroups of $H$}
\ar@/^1pc/@<1ex>[r]|-{\ V\ }
&
\text{Algebraic subgroups of $T$}
\ar@/^1pc/@<1ex>[l]|-{\ \epsilon\ }
}
\end{equation}
where $V$ sends a subgroup $\xi \le H$ to $\Hom(H/\xi,\C^*) \subseteq T$,
while $\epsilon$ sends an algebraic subgroup $C \subseteq T$
to $\ker (\check{T} \surj \check{C}) \le H$.

Both sides of \eqref{eq:pont dual} are partially ordered
sets, with naturally defined meets and joins. As
showed in \cite{SYZ}, the above correspondence
is an order-reversing equivalence of lattices.

\subsection{Primitive lattices and connected subgroups}
\label{subsec:primitive}

Given a subgroup $\xi\le H$, set
\begin{equation}
\label{eq:primitive closure}
\overline{\xi}:=\big\{ x\in H \mid \text{$m x\in \xi$ for some $m\in\N$}\big\}.
\end{equation}
Clearly, $\overline{\xi}$ is again a subgroup of $H$, and
$H/\overline{\xi}$ is torsion-free.
By definition, $\xi$ is a finite-index subgroup of $\overline{\xi}$; in
particular, $\rank(\xi) = \rank(\overline{\xi})$.  The quotient
group, $\overline{\xi}/\xi$, called the {\em determinant group}\/
of $\xi$, fits into the exact sequence
\begin{equation}
\label{eq:ex}
\xymatrix{0 \ar[r] & H/\overline{\xi}  \ar[r] & H/\xi
\ar[r]&\overline{\xi}/\xi  \ar[r]& 0}.
\end{equation}
The inclusion $\overline{\xi} \inj H$
induces a splitting $\overline{\xi}/\xi \inj H/\xi$,
showing that $\overline{\xi}/\xi \cong \Tors(H/\xi)$.
Since the (abelian) group $\overline{\xi}/\xi$ is finite, it is
isomorphic to its character group,
$\widehat{\overline{\xi}/\xi}$, which in turn can be
viewed as a (finite) subgroup of $\widehat{H}=T$.

The subgroup $\xi$ is called {\em primitive}\/ if $\overline{\xi}=\xi$.
Under the correspondence $H \leftrightsquigarrow T$,
primitive subgroups of $H$ correspond to connected
algebraic subgroups of $T$.   For an arbitrary subgroup $\xi\le H$,
we have an isomorphism of algebraic groups,
\begin{equation}
\label{eq:vdecomp}
V(\xi) \cong \widehat{\overline{\xi}/\xi} \cdot V(\overline{\xi}).
\end{equation}
In particular, the irreducible components of $V(\xi)$ are indexed
by the determinant group, $\overline{\xi}/\xi$, while the identity
component is $V(\overline\xi)$.

\subsection{Pulling back algebraic subgroups}
\label{subsec:pullback}

Now let  $\nu\colon H \surj A$ be an epimorphism,
and let $\bar\nu \colon \oH \surj \oA$ be the
induced epimorphism between maximal torsion-free
quotients.   Applying the $\Hom(-,\C^*)$ functor
to the left square of \eqref{eq:two sq} yields the
commuting right square in the display below:
\begin{equation}
\label{eq:two sq}
\xymatrix{
H \ar@{->>}[r]^{\nu} \ar[d] & A \ar[d]\\
\oH \ar@{->>}[r]^{\bar{\nu}} & \oA
}
\xymatrix{ \qquad \ar@{}[d]^{\DS{\leadsto}}&  \\ \: }
\xymatrix{
\wH & \:\wA\ar@{_{(}->}[l]_{\hat\nu} \\
\widehat{\oH}  \ar[u]
& \:\widehat{\oA} \ar@{_{(}->}[l]_{\hat{\bar\nu}} \ar[u]
}
\end{equation}

The morphism $\hat{\nu}\colon \wA\to \wH$
sends the identity component $\wA_0$
to the identity component $\wH_0$, thereby
defining a morphism $\hat{\nu}_0\colon \wA_0\to \wH_0$.
Fixing a splitting $\oA \to A$ yields an isomorphism
$\wA \cong \widehat{\oA} \times \Tors(A)$.
The following lemma is now clear.

\begin{lemma}
\label{lem:nunubar}
Let  $\nu\colon H \surj A$ be an epimorphism.
Upon identifying $\widehat{\oA}=\wA_0$ and $\widehat{\oH}=\wH_0$,
we have:
\begin{romenum}
\item \label{nu1}
$\hat{\bar\nu} = \hat{\nu}_0$.
\item \label{nu2}
$\im(\hat{\nu} ) = V(\ker (\nu) )$.
\end{romenum}
\end{lemma}
Consequently, $\im(\hat{\bar\nu})=V(\ker (\nu) )_0$.

\subsection{Intersections of translated subgroups}
\label{subsec:itt}

Before proceeding, we need to recall some results from \cite{SYZ},
which build on work of Hironaka \cite{Hi}. In what follows,
$T$ will be an abelian, reductive complex algebraic group.

\begin{prop}[\cite{SYZ}]
\label{prop:int2}
Let $\xi_1$ and $\xi_2$ be two subgroups of $H$,
and let $\eta_1$ and $\eta_2$ be two elements in
$T=\widehat{H}$.  Then
$\eta_1 V(\xi_1) \cap \eta_2 V(\xi_2) \ne \emptyset $
if and only if $\eta_1 \eta_2^{-1} \in V(\xi_1\cap\xi_2)$,
in which case
\[
\dim \eta_1 V(\xi_1) \cap \eta_2 V(\xi_2)=
\rank H - \rank (\xi_1+\xi_2).
\]
\end{prop}

\begin{prop}[\cite{SYZ}]
\label{prop:dimension equal}
Let $C$ and $V$ be two algebraic subgroups of $T$.
\begin{romenum}
\item \label{e1}
Suppose $\alpha_1$, $\alpha_2$, and $\eta$ are  torsion
elements in $T$ such that $\alpha_i C \cap \eta V \neq \emptyset$,
for $i=1, 2$.  Then
\[
\dim\, (\alpha_1 C \cap \eta V)=\dim\, (\alpha_2 C \cap \eta V).
\]
\item \label{e2}
Suppose $\alpha$
and $\eta$ are torsion elements in $T$, of coprime order. Then
\[
C \cap \eta V = \emptyset \implies \alpha C \cap  \eta V= \emptyset.
\]
\end{romenum}
\end{prop}

Here is a corollary, which will be useful later on.

\begin{corollary}
\label{cor:empty}
Let $C$ and $V$ be two algebraic subgroups of $T$.
Suppose $\alpha$ and $\rho$ are torsion
elements in $T$, such that
$\rho \notin CV$, $\alpha^{-1}\rho \in CV$, and $\dim(C \cap V)> 0$.
Then $C \cap \rho V=\emptyset$ and $\dim(\alpha C \cap \rho V) > 0$.
\end{corollary}

\begin{proof}
By Proposition \ref{prop:int2},
\[
\rho \notin C \cdot V \Leftrightarrow C \cap \rho V=\emptyset \text{ and }
\alpha^{-1}\rho \in C\cdot V  \Leftrightarrow \alpha C \cap \rho V \neq\emptyset.
\]
By Proposition \ref{prop:dimension equal}, $\dim(\alpha C \cap \rho V)=
\dim(C \cap V)> 0$.  The conclusion follows.
\end{proof}

\section{An algebraic analogue of the exponential tangent cone}
\label{sect:gen tcone}

We now associate to each subvariety $W \subset T$
and integer $d\ge 1$ a finite collection, $\Xi_d(W)$,
of subgroups of the weight group $H=\check{T}$,
which allows us to generalize the exponential tangent
cone construction from \cite{DPS-duke}.

\subsection{A collection of subgroups}
\label{subsec:tau1}
Let $T$ be an abelian, reductive, complex algebraic group,
and consider a Zariski closed subset $W\subset T$.
The translated subtori contained in $W$ define an
interesting collection of subgroups of the discrete
group $H=\check{T}$.

\begin{definition}
\label{def:tau1}
Given a subvariety $W \subset T$, and a positive integer $d$,
let $\Xi_d(W)$ be the collection of all subgroups $\xi \le H$
for which the following two conditions are satisfied:
\begin{romenum}
\item \label{xi1}
The determinant group
$\overline{\xi}/\xi$ is cyclic of order dividing $d$.
\item \label{xi2}
There is a generator $\eta\in \widehat{\overline{\xi}/\xi}$ such that
$\eta\cdot V(\overline{\xi})$ is a maximal, positive-dimensional,
torsion-translated subtorus in $W$.
\end{romenum}
\end{definition}

Clearly, if $d\mid m$, then $\Xi_d(W) \subseteq \Xi_m(W)$.
Although this is not a priori clear from the definition, we shall see in
Proposition \ref{prop:bigxi} that $\Xi_d(W)$ is finite,
for each $d\ge 1$.

To gain more insight into this concept, let us work out
what the sets $\Xi_d(W)$ look like in the case when $W$
is a coset of an algebraic subgroup of $T$.

\begin{lemma}
\label{lem:eta vxi}
Suppose $W=\eta V(\chi)$, where $\chi$ is a subgroup of $H$
and $\eta\in \widehat{H}$ is a torsion element.
Write $V(\chi)=\bigcup_{\rho\in \widehat{\overline\chi/\chi}}
\rho V(\overline\chi)$.
Then
\[
\Xi_d(W) = \Big\{\, \xi \le H \,\big|\, \text{$\exists\,
\rho\in \widehat{\overline\chi/\chi}$ such that
$\ord(\eta \rho)\, |\, d$ and $\widehat{H/\xi} = \bigcup_{m \ge 1}
(\eta\rho)^m V(\overline\chi)$}\, \Big.
\Big\} .
\]
\end{lemma}

\begin{corollary}
\label{cor:eta vxi}
If $\chi$ is a primitive subgroup of $H$ and $\eta$ is an element of
order $d$ in $\wH$, then $\Xi_d(\eta V(\chi))$ consists of the single
subgroup $\xi\le \chi$ for which $\overline{\xi}=\chi$ and
$\widehat{\chi/\xi}=\langle \eta \rangle$.
\end{corollary}

Now note that $\Xi_d$ commutes with unions:
if  $W_1$ and $W_2$ are two subvarieties of $T$, then
\begin{equation}
\label{eq:xid union}
\Xi_{d}(W_1\cup W_2)= \Xi_{d}(W_1) \cup  \Xi_{d}(W_2).
\end{equation}
Lemma \ref{lem:eta vxi}, then, provides an algorithm for computing
the sets $\Xi_d(W)$, whenever $W$ is a (finite) union of torsion-translated
algebraic subgroups of $T$.

\begin{example}
\label{ex:xidu}
Let $H=\Z^2$, and consider the subvariety
$W=\{ (t,1) \mid t \in \C^*\} \cup \{ (-1,t) \mid t \in \C^*\}$ inside $T=(\C^*)^2$.
Note that $W=V(\xi_1) \cup \eta V(\xi_2)$, where $\xi_1=0\oplus \Z$,
$\xi_2=\Z\oplus 0$, and $\eta=(-1,1)$. Hence,
$\Xi_d(W)=\{\xi_1\}$ if $d$ is odd, and $\Xi_d(W)=\{\xi_1,2 \xi_2\}$
if $d$ is even.
\end{example}

\subsection{The exponential map}
\label{subsec:vexp}

Consider now the lattice
\begin{equation}
\label{eq:dual}
\HH=H^{\vee} :=\Hom(H, \Z).
\end{equation}
Evidently, $\HH \cong H/\Tors(H)$.  Moreover, each subgroup
$\xi\le H$ gives rise to a sublattice $(H/\xi)^{\vee} \le H^{\vee}$.

Let $\Lie(T)$ be the Lie algebra of the complex algebraic group $T$.
The exponential map $\exp\colon \Lie(T) \to T$
is an analytic map, whose image is $T_0$.
Let us identify $T_0= \Hom(\HH^{\vee}, \C^*)$ and
$\Lie(T) =\Hom (\HH^{\vee},\C)$.  Under these
identifications, the corestriction to the image of
the exponential map can be written as
\begin{equation}
\label{eq:exp}
\exp=\Hom(-, e^{2 \pi \ii z})\colon\Hom(\HH^{\vee}, \C) \to
\Hom(\HH^{\vee}, \C^*),
\end{equation}
where $z\mapsto e^z$ is the usual exponential map from $\C$ to $\C^*$.
Finally, upon identifying $\Hom (\HH^{\vee},\C)$ with $\HH\otimes \C$,
we see that $T_0=\exp(\HH\otimes \C)$.

The correspondence $T\leadsto  \HH=(\check{T})^{\vee}$
sends an algebraic subgroup $W$ inside $T$ to the
sublattice $\chi=(\check{W})^{\vee}$
inside $\HH$.  Clearly, $\chi=\Lie(W) \cap \HH$  is a
primitive lattice; furthermore,
$\exp(\chi \otimes \C)=W_0$.
As shown in \cite{SYZ}, we have
\begin{equation}
\label{eq:chi}
V((\HH/\chi)^{\vee})= \exp(\chi \otimes \C),
\end{equation}
where both sides are connected algebraic subgroups
inside $T_0=\exp(\HH \otimes \C)$.

\subsection{Exponential interpretation}
\label{subsec:tau exp}

The construction from \S\ref{subsec:tau1} allows us to associate
to each subvariety $W\subset T$
and each integer $d\ge 1$ a subset $\tau_d(W) \subseteq H^{\vee}$,
given by
\begin{equation}
\label{eq:tau1w}
\tau_d(W) = \bigcup_{\xi \in \Xi_d(W)} (H/\xi)^{\vee}.
\end{equation}

The next lemma reinterprets the set $\tau_1(W)$ in terms of the ``exponential
tangent cone" construction introduced in \cite{DPS-duke} and studied in detail
in \cite{Su14}.

\begin{lemma}
\label{lemma:tau1}
For every subvariety $W\subset T$,
\begin{equation}
\label{eq:tau1 again}
\tau_1(W)= \{ x\in H^{\vee} \mid \exp(\lambda x) \in W,
\text{ for all $\lambda\in \C$} \}.
\end{equation}
\end{lemma}

\begin{proof}
Denote by $\tau$ the right-hand side of \eqref{eq:tau1 again}.
Given a non-zero homomorphism $x\colon H\to \Z$
such that $x\in \tau$, the subgroup $\ker(x)\le H$ is primitive
and $V(\ker(x))=\exp(\C x)\subseteq W$. Hence, we can
find a subgroup $\xi\le H$ such that $\xi$ is primitive,
$V(\ker(x)) \subseteqq V(\xi)$, and $V(\xi)\subseteq W$
is a maximal subtorus.  By \eqref{eq:chi}, we have that
$V(\xi)=\exp((H/\xi)^{\vee}\otimes \C)$, which implies
$x \in \tau_1(W)$.

Conversely, for any non-zero element $x\in \tau_1(W)$,
there is a subgroup $\xi\le H$ such that $\xi \in \Xi_1(W)$ and
$x \in (H/\xi)^{\vee}$.  Thus, $\C x \subseteq (H/\xi)^{\vee} \otimes \C$,
and so $\exp(\C x) \subseteq \exp((H/\xi)^{\vee} \otimes \C)=V(\xi) \subseteq W$.
Since $x \neq 0$, the map $x\colon H \to \Z$ is surjective; thus,
$V(\ker(x))=V((\HH/\chi)^{\vee})$, where $\chi$ is the rank $1$
sublattice of $\HH$ generated by $x$. Hence,
$V(\ker(x))=\exp(\chi \otimes \C) \subseteq W$, and so $x\in \tau$.
\end{proof}

Using now the characterization of exponential tangent cones
given in \cite{DPS-duke, Su14}, we obtain the following
immediate corollary.

\begin{corollary}
\label{cor:exp tcone}
$\tau_1(W)$ is a finite union of subgroups of $H^{\vee}$.
\end{corollary}

Thus, the set $\tau_1^{\Q}(W) =
\bigcup_{\xi\in \Xi_1(W)} (H/\xi)^{\vee} \otimes \Q$
is a finite union of linear subspaces in the vector space $\Q^n$,
where $n=\rank H$.

\begin{example}
\label{ex:om1}
Let $T=(\C^*)^n$, and
suppose $W = Z(f)$, for some Laurent polynomial $f$
in $n$ variables.
Write $f(t_1,\dots,t_n)= \sum_{a\in S} c_a t_1^{a_1}\cdots t_n^{a_n}$,
where $S$ is a finite subset of $\Z^n$, and $c_a\in \C^*$
for each $a=(a_1,\dots,a_n)\in S$. We say a partition
$\cS=(\cS_1 \dv \cdots \dv \cS_q)$ of the support $S$ is
admissible if $\sum_{a\in \cS_j} c_a =0$, for each
$1\le j\le q$.  To such a partition, we associate the subgroup
\begin{equation}
\label{eq:rat sub}
L(\cS)= \{ x\in \Z^n \mid (a-b)\cdot x =0,
\; \forall a,b\in \cS_j, \; \forall\, 1\le j\le q \}.
\end{equation}
Then $\tau_1 (W)$ is the union of all subgroups $L(\cS)$,
where $\cS$ runs through the set of admissible partitions of $S$.
In particular, if $f(1)\ne 0$, then $\tau_1 (W)=\emptyset$.
\end{example}

\begin{prop}
\label{prop:bigxi}
For each $d\ge 1$, the set $\Xi_d(W)$ is finite.
\end{prop}

\begin{proof}
Fix an integer $d\ge 1$. For any torsion point $\eta \in T$
whose order divides $d$, consider the set $\Xi_d(W,\eta)$ of
subgroups $\xi\le H$ for which
$\widehat{\overline{\xi}/\xi}=\langle \eta\rangle$ and
$\eta\cdot V(\overline{\xi})$ is a maximal, positive-dimensional,
torsion-translated subtorus in $W$.  Then
\begin{equation}
\label{eq:xidw}
\Xi_d(W)=\bigcup_{\eta} \Xi_d(W,\eta),
\end{equation}
where the
union runs over the (finite) set of torsion points $\eta\in T$
whose order divides $d$.

For each such point $\eta$,
we have a map $\Xi_d(W,\eta)\to  \Xi_1(\eta^{-1}W)$,
$\xi \mapsto \overline{\xi}$.  Clearly, this map is an injection.
Now, Corollary \ref{cor:exp tcone} insures that the set
$\Xi_1(W)$ is finite. Thus, the set $\Xi_1(\eta^{-1}W)$
is also finite, and we are done.
\end{proof}

\section{The incidence correspondence for subgroups of $H$}
\label{sect:schubert}

We now single out certain subsets $\sigma_A(\xi)$
and $U_A(\xi)$ of the parameter set
$\Gamma(H,A)$, which may be viewed as analogues of
the special Schubert varieties in Grassmann geometry.

\subsection{The sets $\sigma_A(\xi)$}
\label{subsec:sigma xi}

We start by recalling a classical geometric construction.
Let $V$ be a variety in $\Q^n$ defined by homogeneous 
polynomials.  Set $m=\dim V$, and assume $m>0$. Consider 
the locus of $r$-planes in $\Q^n$ intersecting $V$ non-trivially,
\begin{equation}
\label{eq:incident}
\sigma_r(V) = \big\{\, P \in \Grass_{r}(\Q^n)
\bigmid \dim( P \cap  V )>0 \,\big\}.
\end{equation}
This set is a Zariski closed subset of
$\Grass_{r}(\Q^n)$, called the {\em variety of incident
$r$-planes}\/ to $V$. For all $0<r< n-m$, this is an irreducible
subvariety, of dimension $(r-1)(n-r)+m-1$.

Particularly simple is the case when $V$ is
a linear subspace $L\subset \Q^n$.  The corresponding
incidence variety, $\sigma_r(L)$, is known as the
{\em special Schubert variety}\/ defined by $L$.
Clearly, $\sigma_1(L)=\bP(L)$, viewed as a projective
subspace in $\QP^{n-1}:=\bP(\Q^n)$.

Now let $H$ be a finitely generated abelian group, let
$A$ be a factor group of $H$, and let 
$\Gamma(H,A)=\Epi(H,A)/\Aut(A)$.

\begin{definition}
\label{def:sigmaxi}
Given a subgroup $\xi\le H$, let
$\sigma_A(\xi)$ be the set of all $[\nu] \in \Gamma(H,A)$ for
which $\rank (\ker(\nu) + \xi) < \rank H$.
\end{definition}

When $A$ is torsion-free, we recover the classical definition
of special Schubert varieties.  More precisely, set $n=\rank H$
and $r=\rank A$.  We then have the following  lemma.

\begin{lemma}
\label{lem:schubert}
Under the natural isomorphism
$\Gamma(H,\oA) \cong \Grass_r(\Q^n)$, the set
$\sigma_{\oA}(\xi)$  corresponds to the special Schubert variety
$\sigma_r((H/\xi)^{\vee} \otimes \Q)$.
\end{lemma}

\begin{proof}
Let $\Grass_r(H^{\vee}\otimes \Q)$ be the Grassmannian of $r$-dimensional
subspaces in the vector space $H^{\vee}\otimes \Q \cong \Q^n$.
Given an epimorphism $\nu\colon H\surj \overline{A}$ and a
subgroup $\xi\le H$, we have
\[
\rank (\ker(\nu) + \xi) < \rank H \same
\dim((H/\ker (\nu))^{\vee}\otimes\Q\cap (H/\xi)^{\vee}\otimes\Q)>0.
\]
Thus, the isomorphism
\[
\Gamma(H,\oA) \isom \Grass_r(H^{\vee}\otimes \Q),
\quad [\nu] \mapsto (H/\ker (\nu))^{\vee}\otimes \Q
\]
establishes a one-to-one correspondence between
$\sigma_{\oA}(\xi)$ and $\sigma_r((H/\xi)^{\vee} \otimes \Q)$.
\end{proof}

For instance, if $A$ is infinite cyclic, then the parameter set $\Gamma(H,\Z)$
may be identified with the projective space $\bP(H^{\vee})$,
while the set  $\sigma_{\Z}(\xi)$ coincides with the
projective subspace $\bP((H/\xi)^{\vee})$.

\begin{example}
\label{ex:rk1}
Let $\xi\le \Z^2$ be the sublattice
spanned by the vector $(a,b)\in \Z^2$.  Then
$\sigma_{\Z}(\xi)\subset \Gamma(\Z^2,\Z)$
corresponds to the point $(-b,a)\in \QP^1$.
\end{example}

The sets $\sigma_A(\xi)$ can be reconstructed from the
classical Schubert varieties $\sigma_{\oA}(\xi)$ associated
to the lattice $\oA=A/\Tors(A)$ by means
of the set fibration described in Theorem \ref{theorem:new bij}.
More precisely, we have the following proposition.

\begin{prop}
\label{prop:comp schubert}
Let $q\colon \Gamma(H,A) \to \Gamma(H,\oA)$ be the
natural projection map.  Then
\begin{romenum}
\item $q(\sigma_A(\xi) )= \sigma_{\oA}(\xi)$.
\item $q^{-1}(\sigma_{\oA}(\xi)) =\sigma_A(\xi)$.
\end{romenum}
Therefore, $\sigma_A(\xi)$ fibers over the Schubert variety
$\sigma_{\oA}(\xi)$, with each fiber isomorphic to the set
$\Gamma(H/\oA, A/\oA)$.
\end{prop}

\subsection{The sets $U_A(\xi)$}
\label{subsec:uxi}

Although simple to describe, the sets $\sigma_A(\xi)$ do not
behave too well with respect to the correspondence
between subgroups of $H$ and algebraic subgroups of
$T=\wH$.  This is mainly due to the fact that the $\sigma_A$-sets
do not distinguish between a subgroup $\xi\le H$ and its primitive
closure, $\overline{\xi}$.  To remedy this situation, we
identify certain subsets $U_A(\xi)\subseteq \sigma_A(\xi)$
which turn out to be better suited for our purposes.

\begin{definition}
\label{def:uxi}
Given a subgroup $\xi\le H$, let
$U_A(\xi)$ be the set of all $[\nu] \in \sigma_A(\xi)$ for which
$\ker(\nu)\cap \overline{\xi}\subseteq \xi$.
\end{definition}

In particular, if $\overline{\xi}=\xi$, then $U_A(\xi)=\sigma_A(\xi)$.
In general, though, $U_A(\xi)\subsetneqq \sigma_A(\xi)$.
In order to reinterpret this definition in more geometric terms,
we need a lemma.

\begin{lemma}
\label{lemma:equiva}
Let $\xi \le H$ be a subgroup such that $\overline{\xi}/\xi$ is cyclic,
and let $\chi\le H$ be another a subgroup.
Then the following conditions are equivalent.
\begin{romenum}
\item \label{w1}
$\chi \cap \overline{\xi}\subseteq \xi$.
\item \label{w2}
$V(\chi) \cap \eta V(\bar{\xi}) \neq \emptyset$, for some
generator $\eta$ of $\widehat{\overline{\xi}/\xi}$.

\item \label{w3}
$V(\chi) \cap \eta V(\bar{\xi}) \neq \emptyset$, for any
generator $\eta$ of $\widehat{\overline{\xi}/\xi}$.

\item \label{w4}
$\eta\in V(\ker(\nu) \cap \overline{\xi})$,  for some
generator $\eta$ of $\widehat{\overline{\xi}/\xi}$.

\item \label{w5}
$\eta\in V(\ker(\nu) \cap \overline{\xi})$,  for any
generator $\eta$ of $\widehat{\overline{\xi}/\xi}$.
\end{romenum}
\end{lemma}

\begin{proof}
Let $\eta$ be a generator of the finite cyclic group $\widehat{\overline{\xi}/\xi}$.
We then have
\begin{equation}
\label{eq:epxi}
\epsilon(\langle \eta \rangle) \cap \overline{\xi} = \xi.
\end{equation}

By Proposition \ref{prop:int2}, the intersection $V(\chi) \cap \eta V(\bar{\xi})$
is non-empty if and only if $\eta\in V(\chi \cap \bar{\xi})$,
that is, $\langle \eta \rangle \subseteq V(\chi \cap \bar{\xi})$, which
in turn is equivalent to
\begin{equation}
\label{eq:inc}
\epsilon(\langle \eta \rangle) \supseteq \chi \cap \bar{\xi}.
\end{equation}
In view of equality \eqref{eq:epxi}, inclusion \eqref{eq:inc} is
equivalent to $\chi \cap \overline{\xi}\subseteq \xi$.  This shows
\eqref{w1} $\Leftrightarrow$ \eqref{w2}.

The other equivalences are proved similarly.
\end{proof}

\begin{corollary}
\label{cor:uxi}
Let $\xi\le H$ be a subgroup, and assume $\overline{\xi}/\xi$ is cyclic.
Let $\nu\colon H\surj A$ be an epimorphism. Then
\[
[\nu] \in U_A(\xi) \: \same  \:
\dim\, ( V(\ker(\nu))\cap \eta V(\overline\xi))>0
\]
for any (or, equivalently, for some)
generator $\eta\in \widehat{\overline{\xi}/\xi}$.
\end{corollary}

Despite their geometric appeal, the sets $U_A(\xi)$ do not
enjoy a naturality property analogous to the
one from Proposition \ref{prop:comp schubert}.
In particular, the projection map
$q\colon \Gamma(H,A) \to \Gamma(H,\oA)$
may {\em not}\/ restrict to a map
$U_A(\xi) \to U_{\oA}(\xi)$. Here is a simple
example.

\begin{example}
\label{ex:rk3}
Let $\nu\colon H\surj A$ be the
canonical projection from $H=\Z^2$ to  $A=\Z\oplus \Z_2$,
and let  $\xi=\ker(\nu)$.  Then $[\nu]\in U_A(\xi)$,
but $[\bar\nu]\notin U_{\oA}(\xi)$.
\end{example}

\section{The incidence correspondence for subvarieties of $\wH$}
\label{sect:alg df}

In this section, we introduce and study certain subsets
$\Upsilon_A(W)\subseteq \Gamma(H,A)$, which can be
viewed as the toric analogues of the classical incidence
varieties $\sigma_r(V)\subseteq \Grass_r(\Q^n)$.

\subsection{The sets $U_{A}(W)$}
\label{subsec:lub}
Let us start by recalling some constructions we discussed
previously. In \S\ref{subsec:tau1}, we associated to
each subvariety $W$ of the algebraic group $T=\wH$
and each integer $d\ge 1$ a
certain collection $\Xi_{d}(W)$ of subgroups of $H$.
In \S\ref{subsec:uxi}, we associated
to each subgroup $\xi\le H$ and each abelian group $A$
a certain subset $U_A(\xi)$ of the parameter set
$\Gamma(H,A)=\Epi(H,A)/\Aut(A)$.
Putting together these two constructions, we associate
now to $W$ a family of subsets of $\Gamma(H,A)$, as follows.

\begin{definition}
\label{def:uad w}
Given a subvariety $W\subset T$, an abelian group $A$,
let
\begin{equation}
\label{eq:ua}
U_A(W)= \bigcup_{d\ge 1} U_{A,d} (W) ,
\end{equation}
where
\begin{equation}
\label{eq:uad}
U_{A,d} (W) = \bigcup_{\xi\in \Xi_{d}(W)} U_A(\xi) .
\end{equation}
\end{definition}

By Proposition \ref{prop:bigxi}, the union in \eqref{eq:uad} is a finite union.

\begin{lemma}
\label{lem:uaw}
The set $U_{A,d} (W)$ consists of all $[\nu]\in \Gamma(H,A)$
for which there is a subgroup $\xi\le H$ and an element $\eta\in \wH$
of order dividing $d$ such that $\eta V(\overline{\xi})$ is a maximal,
positive-dimensional translated subtorus in $W$, and
$\dim\, ( V(\ker \nu)\cap \eta V(\overline\xi))>0$.
\end{lemma}

\begin{proof}
Let $\nu\colon H \surj A$ be an epimorphism such that
$[\nu] \in U_A(\xi)$, for some $\xi\in \Xi_{d}(W)$.
According to Definition \ref{def:tau1}, this means that
the group $\overline{\xi}/\xi$ is cyclic of order dividing $d$,
and there is a generator $\eta\in \widehat{\overline{\xi}/\xi}$
such that $\eta V(\overline{\xi})$ is a maximal, positive-dimensional
translated subtorus in $W$.  In view of Corollary \ref{cor:uxi},
the fact that $[\nu] \in U_A(\xi)$ insures that
$V(\ker(\nu))\cap \eta V(\overline\xi)$ has positive dimension.
\end{proof}

The case $d=1$ is worth singling out.

\begin{corollary}
\label{cor:ua1w}
Let $W\subset T$ be a subvariety.
Set $n=\rank H$ and $r=\rank A$. Then:
\begin{romenum}
\item
$U_{A,1}(W)=\bigcup_{\xi\in \Xi_1(W)} \sigma_A(\xi)$.
\item
Under the isomorphism
$\Gamma(H,\oA) \cong \Grass_r(\Q^n)$, the set
$U_{\oA,1}(W)$  corresponds to the incidence variety
$\sigma_r(\tau_1^{\Q}(W))$.
\end{romenum}
\end{corollary}

In general, the set $U_{A}(W)$ is larger than $U_{A,1}(W)$.
Here is a simple example; a more general situation will be
studied in \S\ref{subsec:z2}.

\begin{example}
\label{ex:xidu again}
Let $H=\Z^2$ and let $W\subset (\C^*)^2$ be the subvariety
from Example \ref{ex:xidu}. Pick  $A=\Z\oplus \Z_2$, and identify
$\Gamma(H,A)$ with $\QP^1$.  Then $U_{A,d}(W)=\{(1,0)\}$ or
$\{(1,0),(0,1)\}$, according to whether $d$ is odd or even.
\end{example}

\subsection{The sets $\Upsilon_A(W)$}
\label{subsec:omw}

As before, let $H$ be a
finitely-generated abelian group, and let $T=\wH$ be
its Pontryagin dual.  The next definition will prove to
be key to the geometric interpretation of the (generalized)
Dwyer--Fried invariants.

\begin{definition}
\label{def:omega w}
Given a subvariety $W\subset T$, and an abelian
group $A$, define a subset $\Upsilon_A(W)$ of the
parameter set $\Gamma(H,A)$ by setting
\begin{equation}
\label{eq:upsa}
\Upsilon_{A}(W)= \big\{\, [\nu]\in \Gamma(H,A) \mid
\dim( V(\ker \nu) \cap W) >0\, \big\}.
\end{equation}
\end{definition}

Roughy speaking, the set $\Upsilon_A(W)\subset \Gamma(H,A)$
associated to a variety $W\subset T$ is  the toric analogue
of the incidence variety $\sigma_r(V)\subset \Grass_r(\Q^n)$
associated to a homogeneous variety $V\subset \Q^n$.

It is readily seen that $\Upsilon_A$ commutes with unions:
if  $W_1$ and $W_2$ are two subvarieties of $T$, then
\begin{equation}
\label{eq:uaw union}
\Upsilon_{A}(W_1\cup W_2)= \Upsilon_{A}(W_1) \cup  \Upsilon_{A}(W_2).
\end{equation}
Moreover, $\Upsilon_A(W)$ depends only on the positive-dimensional
components of $W$.  Indeed, if $Z$ is a finite algebraic set, then
$\Upsilon_A(W\cup Z)=\Upsilon_A(W)$.

The next result gives a convenient lower bound
for the $\Upsilon$-sets.

\begin{prop}
\label{prop:om bd}
Let $A$ be a quotient of $H$.  Then
\begin{equation}
\label{eq:upsw}
U_A(W)\subseteqq \Upsilon_A(W) .
\end{equation}
\end{prop}

\begin{proof}
Let $\nu\colon H \surj A$ be an
epimorphism such that  $[\nu] \in U_{A,d}(W)$,
for some $d\ge 1$.
By Lemma \ref{lem:uaw}, we have that
$\dim (V(\ker \nu)\cap W)>0$.  Thus,
$[\nu]\in \Upsilon_A(W)$.
\end{proof}

As we shall see in Example \ref{example:ruled},
inclusion \eqref{eq:upsw} may well be strict.

\subsection{Translated subgroups}
\label{subsec:omw tt}

If the variety $W$ is a torsion-translated algebraic
subgroup of $T$, we can be more precise.

\begin{theorem}
\label{thm:uprho1}
Let $W=\eta V(\xi)$, where $\xi \le H$ is a subgroup,
and $\eta\in \wH$ has finite order.  Then
$\Upsilon_{A}(W) = U_{A,c}(W)$, where
$c=\ord(\eta)\cdot c(\overline{\xi}/\xi)$.
\end{theorem}

\begin{proof}
Inclusion $\supseteq$ follows from Proposition \ref{prop:om bd},
so we only need to prove the opposite inclusion.
Write
\[
V(\xi)=\bigcup_{\rho \in \widehat{\overline\xi/\xi}}\rho V(\overline\xi).
\]

Let $\nu\colon H\surj A$ be an epimorphism such that
$[\nu] \in \Upsilon_A(\eta V(\xi))$.
Hence, there is a character $\rho \colon \overline{\xi}/\xi \to \C^*$
such that $\dim\, (V(\ker(\nu)) \cap \eta\rho V(\overline\xi)) >0$.
Consider the subgroup
\[
\chi = \epsilon\, \bigg( \bigcup_m (\eta\rho)^m V(\overline\xi)\bigg).
\]
Lemma \ref{lem:eta vxi} implies that $\chi\in \Xi_{d}(\eta V(\xi))$,
where $d=\ord (\eta\rho)$.  Using
Corollary \ref{cor:uxi}, we conclude that $[\nu]\in U_{A}(\chi)$.

Finally, set $c=\ord(\eta)\cdot c(\overline{\xi}/\xi)$.
Then clearly $[\nu]\in U_{A,c}(W)$, and we are done.
\end{proof}

Here is an alternate description of the set $\Upsilon_{A}(W)$,
in the case when $W$ is an algebraic subgroup of $T$, translated
by an element $\eta\in T$, not necessarily of finite order.

\begin{theorem}
\label{thm:uprho2}
Let $\xi \le H$ be a subgroup, and let $\eta\in \wH$.
Then
\[
\Upsilon_{A}(\eta V(\xi))
=\sigma_A(\xi) \cap \big\{[\nu] \in \Gamma(H,A) \mid
 \eta \in V(\ker(\nu)\cap \xi)
\big\}.
\]
In particular, $\Upsilon_{A}(V(\xi))=\sigma_A(\xi)$.
\end{theorem}

\begin{proof}
By Proposition \ref{prop:int2}, we have
\begin{align*}
&\{[\nu] \mid \dim ( V(\ker(\nu))\cap \eta V(\xi))>0  \} \\
&\qquad =\{[\nu] \mid V(\ker(\nu))\cap \eta V(\xi) \ne
\emptyset\} \cap \{[\nu] \mid \rank(\ker(\nu)+ \xi)< \rank H\}.
\end{align*}
Moreover, $V(\ker(\nu))\cap \eta V(\xi) \ne
\emptyset \same \eta \in V(\ker(\nu)\cap \xi)$, and we are done.
\end{proof}

\begin{remark}
\label{rem:srxieta}
In the case when $A$ is a free abelian group of rank $r$
and $\xi$ is a primitive subgroup of $H=\Z^n$, the set
$\Upsilon_A(\eta V(\xi))$ coincides with the set $\sigma_r(\xi, \eta)$
defined in \cite{Su14}.
\end{remark}

When the translation factor $\eta$ from Theorem \ref{thm:uprho2}
has finite order, a bit more can be said.
\begin{corollary}
\label{cor:uprho3}
 Let $W=\eta V(\xi)$ be a torsion-translated subgroup of $T$.
Then
\[
\Upsilon_{A}(W)
=\sigma_A(\xi) \cap \big\{[\nu] \in \Gamma(H,A) \mid
 \epsilon(\langle \eta \rangle) \supseteq \ker(\nu)\cap \xi
\big\}.
\]
\end{corollary}

\subsection{Deleted subgroups}
\label{subsec:xixi}
We now analyze in more detail the case when the variety
$W$ is obtained from an algebraic subgroup of $T$
by deleting its identity component.
First, we need to introduce one more bit
of notation.

\begin{definition}
\label{def:theta}
Given a subgroup $\xi\le H$, and a quotient
$A$ of $H$, consider the subset $\theta_A(\xi)\subseteq \Gamma(H, A)$
given by
\begin{equation}
\label{eq:theta}
\theta_A(\xi)=\bigcup_{\xi\le \xi'\lvertneqq\overline\xi \: : \:
\text{$\overline\xi/\xi'$ is cyclic}}
\big\{ [\nu] \in \Gamma(H,A) \mid
\nu(x)\neq0\hbox{ for all }
x\in\overline\xi\setminus\xi'
\big\}.
\end{equation}
\end{definition}

Note that the indexing set for this union is a finite set, which is empty
if $\xi$ is primitive. On the
other hand, the condition that $\nu(x)\neq0$ depends on the actual
element $x$ in the (typically) infinite set $\overline{\xi} \setminus \xi'$,
not just on the class of $x$ in the finite group $\overline\xi/\xi'$.
Thus, even when $A=\Z^r$, the set $\theta_A(\xi)$ need not be open
in the Grassmannian $\Gamma(H,A)=\Grass_r(\Q^n)$,
where $n=\rank(H)$, although each of the sets
$\{ [\nu]  \mid\nu(x)\neq0\}$ is open.

\begin{prop}
\label{prop:uprho4}
Suppose $W=V(\xi)\setminus V(\overline{\xi})$, for
some non-primitive subgroup $\xi\le H$.
Then
\begin{equation}
\label{eq:gxi}
\Upsilon_{A}(W)
=\sigma_A(\xi) \cap
\theta_A(\xi).
\end{equation}
\end{prop}

\begin{proof}
Write $W=\bigcup_{\eta \in \widehat{\overline\xi/\xi} \setminus\{1\}}
\eta V(\overline{\xi})$. By Theorem \ref{thm:uprho2} and Lemma \ref{lemma:equiva},
we have
\begin{align*}
\Upsilon_{A}(W) &= \Upsilon_{A}\Bigg(\bigcup_{\eta \in \widehat{\overline\xi/\xi}
\setminus\{1\}} \eta V(\overline{\xi}) \Bigg)\\
& =\bigcup_{\eta \in \widehat{\overline\xi/\xi} \setminus\{1\}}
\Big( \sigma_A(\overline\xi) \cap \big\{[\nu] \in \Gamma(H,A) \mid
  \eta \in V(\ker(\nu)\cap \overline\xi)
\big\} \Big) \\
& =\sigma_A(\overline\xi) \cap \bigcup_{\eta \in \widehat{\overline\xi/\xi}
\setminus\{1\}} \big\{[\nu] \in \Gamma(H,A) \mid
V(\ker(\nu)) \cap \eta V(\overline{\xi}) \ne \emptyset
\big\}\\
& =\sigma_A(\xi) \cap \bigcup_{\xi\le \xi'\lvertneqq\overline\xi \: : \:
\text{$\overline\xi/\xi'$ is cyclic}}
\big\{[\nu] \in \Gamma(H,A) \mid
\ker(\nu) \cap\overline{\xi}\le \xi'
\big\}.
\end{align*}
The desired conclusion follows at once.
\end{proof}

\subsection{Comparing the sets $\Upsilon_A(W)$ and $\Upsilon_{\oA}(W)$}
\label{subsec:comp ups}

Fix a decomposition $A=\oA\oplus \Tors(A)$.
Clearly, the projection map $q=q_A\colon \Gamma(H,A) \to \Gamma (H,\oA)$
sends $\Upsilon_A(W)^c$ to $\Upsilon_{\oA}(W)^c$.
On the other hand, as we shall see in Example \ref{example:ruled},
the map $q$ does {\em not}\/ always send $\Upsilon_A(W)$ to $\Upsilon_{\oA}(W)$.
Nevertheless, in some special cases it does.
Here is one such situation.

\begin{prop}
\label{prop:not pback1}
Suppose $W=\rho T$, where $T\subset \widehat{H}$
is an algebraic subgroup, and $\bar{\rho}\in \widehat{H}/T$
has finite order, coprime to the order of $\Tors(A)$.  Then
\[
q(\Upsilon_A(W))= \Upsilon_{\oA}(W)\quad \text{and} \quad
q^{-1}(\Upsilon_{\oA}(W)) =\Upsilon_A(W).
\]
Therefore, $\Upsilon_A(W)$ fibers over
$\Upsilon_{\oA}(W)$, with each fiber isomorphic
to $\Gamma(H/\oA, A/\oA)$.
\end{prop}

Here, $\bar{\rho}$ is the image of
$\rho$ under the quotient map $\widehat{H}\to \widehat{H}/T$.

\begin{proof}
Let $\nu \colon H\surj A$ be an epimorphism such that
$[\bar{\nu}]=q([\nu])$ does not belong to $\Upsilon_{\oA}(W)$, that is,
the subtorus $\im(\hat{\bar{\nu}})=\im(\hat{\nu})_1$ intersects
$W$ in only finitely many points.
We want to show that $ \im(\hat{\nu})_\alpha \cap W$
is also finite, for all $\alpha \in \T(A)$.

First assume $\im(\hat{\nu})_1 \bigcap W$ is non-empty.
If $\im(\hat{\nu})_\alpha \bigcap W = \emptyset$, we are done.
Otherwise, using
Proposition~\ref{prop:dimension equal}(\ref{e1}) with $C=\im(\hat{\nu})_1$,
$\alpha_1=1$, $\alpha_2=\alpha$ and $\eta V = \rho T $,
we infer that
$\dim\, (\im(\hat{\nu})_\alpha \cap \rho T)=
\dim\, (\im(\hat{\nu})_1 \cap \rho T)$,
and the desired conclusion follows.

Now assume $\im(\hat{\nu})_1 \bigcap W$ is empty.
Using Proposition~\ref{prop:dimension equal}(\ref{e2})
with $C=\im(\hat{\nu})_1$ and $\eta V = \rho T$,
our assumption that
$\im(\hat{\nu})_1 \bigcap \rho T=\emptyset$ implies that
$\im(\hat{\nu})_\alpha \bigcap \rho T=\emptyset$.
Thus, the desired conclusion follows in this case, too,
and we are done.
\end{proof}

\begin{corollary}
\label{cor:pback omega}
For every subgroup $\xi\le H$, the set
$q(\Upsilon_A(V(\xi)))$ is contained in $\Upsilon_{\oA}(V(\xi))$.
\end{corollary}

In general, though, the projection map
$q\colon \Gamma(H,A) \to \Gamma (H,\oA)$
does not restrict to a map $\Upsilon_A(W)\to \Upsilon_{\oA}(W)$.
Proposition \ref{prop:not pback2} below describes a situation
when this happens.  First, we need a lemma, whose proof
is similar to the proof of Proposition \ref{cyclic to arbitrary}.

\begin{lemma}
\label{lem:fun a}
Let $\pi\colon A\surj B$ be an epimorphism, and let
$\tilde\pi\colon \Gamma(H,A)\surj \Gamma(H,B)$ be
the induced homomorphism. Then
\[
q_A(\Upsilon_A(W)) \subseteq \Upsilon_{\overline{A}}(W) \implies
q_B(\Upsilon_B(W)) \subseteq \Upsilon_{\overline{B}}(W).
\]
\end{lemma}

\begin{prop}
\label{prop:not pback2}
Let $H$ be a finitely generated, free abelian group, and
let $A$ be a quotient of $H$ such that $\rank A< \rank H$.
Let $W$ be a subvariety of $\wH$ of the form $\rho T\cup Z$,
where $Z$ is a finite set, $T$ is an algebraic subgroup, and
$\rho$ is a torsion element whose order divides $c(A)$.
Then $q\big(\Upsilon_{A}(W)\big)\not\subseteqq \Upsilon_{\oA}(W)$.
\end{prop}

\begin{proof}
We need to construct an epimorphism $\nu\colon H \surj A$
such that $[\nu] \in \Upsilon_A(W)$, yet
$[\bar{\nu}] \notin \Upsilon_{\oA}(W)$.
\par
\smallskip
{\bf Step 1.} First, we assume $\Tors(A)$ is a cyclic group.
In this case, we claim there exists a subtorus $C$ of
$\widehat{H}$, and a torsion element $\alpha \in \widehat{H}$,
such that $\ord{\alpha}=\ord{\rho}$
and $\dim(C \cap \rho T)\leq 0$, yet $\dim(\alpha C \cap \rho T) > 0$.

To prove the claim, set $\epsilon(\langle\rho\rangle)=L$ and $\xi=\epsilon(T)$.
Since $\rho \notin T$, we have that $\xi \nsubseteq L$.
Thus, there exist a sublattice $\chi$ of rank $1$, such that
$\chi \subseteq \xi$ and $\chi \nsubseteq L$.
Set $T'=V(\chi)$. Then $T'$ is a codimension $1$ subgroup with
$T \subseteq T' \subset \widehat{H}$, and $\rho \notin T'$.
Thus, $T_0 \subseteq T'_0 \subset \widehat{H}$.
Let $C$ be any dimension $r$ subtorus of $T'_0$
intersecting $T$ with positive dimension.
Then $\rho \notin CT \subseteq T'$ and $\dim(C \cap T)>0$.
Choose an element $\alpha \in \widehat{H}$ such that
$\alpha^{-1}\rho=1 \in CT$. Clearly, $\ord(\alpha)=\ord(\rho)\mid c(A)$.
Using Corollary \ref{cor:empty}, we conclude that $C \cap \rho T=\emptyset$
and $\dim(\alpha C \cap \rho T) > 0$, thus finishing the proof of the claim.

Now, the algebraic
subgroup $\bigcup_k\alpha^k C$ corresponds to an
epimorphism $H \surj \overline{A} \oplus \Z_d$, where $d=\ord(\alpha)$.
Since $H$ is torsion-free,  $\Tors(A)$ is cyclic, and
$d$ divides $c(A)$, this epimorphism can be lifted to
an epimorphism $\nu\colon H \surj A$.

\smallskip
{\bf Step 2.} For the general case, let $B$ be the cyclic subgroup
of $\Tors(A)$ for which $\abs{B}=c(A)$ and $\ord(\rho)\mid \abs{B}$.
Notice that $B$ is a direct summand of $\Tors(A)$.  We then have
the following commuting diagram:
\[
\xymatrix{
\Gamma(H/\oA, B) \ar[r] & \Gamma(H, \overline{A} \oplus B)
\ar[r]^{q_{\overline{A} \oplus B}}
&  \Gamma(H, \oA)\\
\Gamma(H/\oA, \T(A))\ar[u]^{\pi_1} \ar[r]
& \Gamma(H, A)\ar[u]^{\pi_2} \ar[r]^{q_A} &  \Gamma(H, \oA)\ar@{=}[u]
}
\]

Since $H/\oA$ is torsion-free, the group $\Aut(H/\oA)$
acts transitively on $\Gamma(H/\oA, B)$.  Using the
assumption that $\Gamma(H/\oA, \T(A))\neq \emptyset$,
we deduce that the map $\pi_1$ is surjective.
Thus, the map $\pi_2$ is surjective.
From Step 1, we know that
$q\big(\Upsilon_{\oA\oplus B}(W)\big)$ is not contained in
$\Upsilon_{\oA}(W)$.  Using Lemma \ref{lem:fun a}, we conclude that
$q\big(\Upsilon_{A}(W)\big)$ is not contained in
$\Upsilon_{\oA}(W)$, either.
\end{proof}

\section{Support varieties for homology modules}
\label{sect:supports}

We now switch gears, and revisit the Dwyer--Fried theory
in a slightly more general context.  In particular, we show
that the support varieties of the homology modules of two
related chain complexes coincide.

\subsection{Support varieties}
\label{subsec:supports}
Let $H$ be a finitely generated abelian group, and
let $F$ be a finitely generated $\C$-algebra.  Then the group ring
$R=F[H]$ is a Noetherian ring.  Let $\Spm(R)$
be the set of maximal ideals in $R$, endowed with the
Zariski topology.

Given a module $M$ over $F[H]$, denote by $\supp M$
its support, consisting of those maximal ideals $\m\in \Spm(F[H])$ for
which the localization $M_{\m}$ is non-zero.

Now let $A$ be another finitely generated abelian group,
and let $\nu\colon  H \surj A$ be an epimorphism.  Denote
by $S=F[A]$ the group ring of $A$.  The extension of $\nu$
to group rings, $\nu\colon  R \surj S$, is a ring epimorphism.
Let $\nu^*\colon  \Spm(S) \inj \Spm(R)$ be the induced
morphism between the corresponding affine varieties.

In the case when $F=\C$, the group ring $R=\C[H]$ is the
coordinate ring of the algebraic group $\wH=\Hom(G,\C^*)$,
and $\Spm(R)=\wH$.  Furthermore, if $M$ is an $R$-module, then
\begin{equation}
\label{eq:vann}
\supp(M) = Z(\ann M),
\end{equation}
where $\ann M \subset R$ is the annihilator of $M$, and
$Z(\ann M) \subset \wH$ is its zero-locus.

\begin{lemma}
\label{lem:suppz}
If $\nu\colon H\surj A$ is an epimorphism, then
\begin{equation}
\label{eq:cap}
(\nu^*)^{-1} (\supp(M)) \cong  \im(\hat\nu)\cap Z(\ann M).
\end{equation}
\end{lemma}

\begin{proof}
From the definitions, we see that the diagram
\begin{equation}
\label{eq:spec}
\xymatrix{
\Spm(S)\, \ar@{^{(}->}[r]^{\nu^*} \ar[d]^{\cong} &
\Spm(R) \ar[d]^{\cong}\\
\wA\: \ar@{^{(}->}[r]^{\hat{\nu}} & \wH
}
\end{equation}
commutes.  The conclusion readily follows.
\end{proof}

\subsection{Homology modules}
\label{subsec:hom mod}
We are now ready to state and prove the main result of this section.
An abbreviated proof was given by Dwyer and Fried in \cite{DF},
in the special case when $A$ is free abelian. For the convenience
of the reader, we give here a complete proof, modeled on the
one from \cite{DF}.

\begin{theorem}
\label{thm:supp}
Let $F$ be a finitely generated $\C-$algebra, let $C_\bullet$ be a chain complex
of finitely generated free modules over $F[H]$, and let $\nu\colon H\surj A$
be an epimorphism.   Viewing $F[A]$ as a module over $F[H]$ by extension
of scalars via $\nu$, we have
\begin{equation}
\label{conclusion}
\supp H_* (C_\bullet \otimes_{F[H]} F[A])=
(\nu^*)^{-1} \big(\!\supp H_* (C_\bullet) \big).
\end{equation}
\end{theorem}

\begin{proof}
Set $n=\rank H$ and $r=\rank A$.  There are three cases to consider.
\par
\smallskip
{\bf Case 1: $H$ is torsion-free.}
We use induction on $n-r$ to reduce to the case $r=n-1$, in which
case $\T(A)= \Z_q$, for some $q\geq1$ (if $A$ is torsion-free, $q=1$).
We then have a short exact sequence of chain complexes,
\[
\xymatrix{0 \ar[r] & C_\bullet \ar[r]^{\phi=x^q-1} & C_\bullet \ar[r]
& C_\bullet \otimes_{F[H]} F[A] \ar[r] & 0},
\]
which yields a long exact sequence of homology groups.  Consider
the map $\phi_*\colon M \to M$, where $M=H_*(C_\bullet)$, viewed
as a module over $F[H]$.  Localizing at a maximal ideal $\m$,
we obtain an endomorphism $\phi_{\m}$ of the finitely-generated module
$M_{\m}$  over the Noetherian ring $F[H]_{\m}$.

As a standard fact, if
$\phi_{\m}$ is surjective, then $\phi_{\m}$ is injective. Using the exact
sequence
\[
0 \to \ker \phi_{\m} \to M_{\m} \to M_{\m} \to \coker \phi_{\m} \to 0,
\]
we see that $\coker \phi_{\m}=0 \Rightarrow \ker \phi_{\m} =0$.
Therefore, $\supp \ker \phi_* \subseteq \supp \coker \phi_*$,
and so
\begin{align*}
   \supp H_*(C_{\bullet} \otimes_{F[H]} F[A])
                             & =  \supp \coker \phi_* \cup \supp \ker \phi_* \\
                             & =  \supp \coker \phi_* \\
                             & =  \supp  M/ (x^q-1) M \\
                             & =  (\supp M) \cap Z(x^q-1) \\
                             & =  (\supp M) \cap \im (\hat\nu).
\end{align*}
\par
When $n-r > 1$, one can change the basis of $H$
and $\oA$ so that, the epimorphism $\nu\colon H \to A$ is the composite
\[
\xymatrix{
H'\oplus \Z \ar@{->>}[r]^{\nu_1} & H'\oplus \Z_q \ar@{->>}[r]^{\nu_2}
& A'\oplus \Z_q},
\]
where $\nu_1$ is of the form
$\left(
         \begin{smallmatrix}
         \id & 0 \\
          0 & \nu|_{\Z}\\
         \end{smallmatrix}
         \right)$, and $\nu_2$ is of the form
$\left(
         \begin{smallmatrix}
         \nu|_{H'} & 0 \\
          0 & \id \\
         \end{smallmatrix}
         \right)$.
By the induction hypothesis, equality \eqref{conclusion} holds for
$\nu_1$ and $\nu_2$.  Thus, the theorem holds for the map
$\nu=\nu_2\circ \nu_1$.

\smallskip
{\bf Case 2: $H$ is finite.} In this situation, $\nu\colon  H \surj A$ is an
epimorphism between two finite abelian groups.  As above, $\nu$ induces
a ring epimorphism $\nu\colon F[H] \surj F[A]$. The corresponding
map, $i=\nu^*\colon \Spm F[A] \inj \Spm F[H]$,
is a closed immersion.  Consider the  commuting diagram
\begin{equation}
\label{eq:spec com}
\xymatrix{ \Spm F[A] \, \ar@{^{(}->}[r]^i \ar[d] & \Spm F[H] \ar[d]\\
           \Spm \C[A]\, \ar@{^{(}->}[r]^j & \Spm \C[H],
}
\end{equation}
where the morphism $j$ is induced by $\nu\colon H \surj A$.
Clearly, $j$ is an open immersion. By commutativity of \eqref{eq:spec com},
we have that $\Spm F[A]$ is an open subset of $\Spm F[H]$.

It suffices to show that
\begin{equation}
\label{eq:supphk}
\supp( H_k(i^*\tilde{C}_\bullet))= i^{-1}(\supp H_k (\tilde{C}_\bullet) )
\end{equation}
for any $k \in \Z$,
where $\tilde{C}_\bullet$ is the sheaf of modules over $\Spm (F[H])$
corresponding to the module $C_\bullet$ over $F[H]$,
and $i^*\tilde{C}_\bullet$ is the sheaf of modules over $\Spm (F[A])$
obtained by pulling back the sheaf  $\tilde{C}_\bullet$.
\par
For any $\m \in \Spm(F[A])$, we have $(H_k(\tilde{C}_{\bullet}))_\m=
H_k ((\tilde{C}_{\bullet})_\m)$,
since localization is an exact functor, and also
$(\tilde{C}_{\bullet})_{\m}=(i^* \tilde{C}_{\bullet})_\m$,
since $i$ is an open immersion.
Thus,
\begin{align*}
   \supp H_k(i^*\tilde{C}_\bullet) & =   \{\m \in \Spm(F[A])
   \mid  (H_k (i^*\tilde{C}_\bullet))_{\m} \neq 0  \} \\
                  & =  \{\m \in \Spm(F[A])
                  \mid  ( i^*\tilde{C}_\bullet)_{\m} \text{ is not exact at $k$}  \}\\
                  & =  \{ \m \in \Spm(F[A])
                  \mid (H_k (\tilde{C}_\bullet))_{\m} \neq 0 \} \\
                  & =  \supp H_k(\tilde{C}_\bullet) \cap \Spm(F[A])\\
                  & =  i^{-1}(\supp H_k(\tilde{C}_\bullet)).
\end{align*}

\smallskip
{\bf Case 3: $H$ is arbitrary.}
As in the proof of Theorem~\ref{theorem:new bij},
we can choose splittings $H=\oH\oplus \Tors(H)$ and
$A=\oA\oplus \Tors(A)$ such that $\oH=\oA\oplus H'$,
and the epimorphism $\nu\colon H \surj A$ is the composite
\[
\xymatrix{
\oA \oplus H' \oplus \T(H) \ar@{->>}[r]^(.46){\nu_1}
& \oA \oplus \nu(H') \oplus \T(H) \ar@{->>}[r]^(.56){\nu_2}
& \oA \oplus \T(A)},
\]
where $\nu(H')$ is finite,
$\nu_1$ is of the form
$\left(  \begin{smallmatrix}
         \id & 0 & 0\\
          0 & \nu|_{H'} & 0\\
          0 & 0 & \id
         \end{smallmatrix}
         \right)$,
and $\nu_2$ is of the form
$\left(
         \begin{smallmatrix}
         \id & 0 &  0\\
         0 & i & \nu|_{\T(H)}
         \end{smallmatrix}
         \right)$,
with $i\colon \nu(H') \inj \T(A)$ the
inclusion map.

Set $F= \C[\oA \oplus \T(H)]$; by Case 1, equality \eqref{conclusion}
holds for $\nu_1$.  Now set $F= \C[\oA]$; by Case 2, equality
\eqref{conclusion} holds for $\nu_2$.  Thus, the theorem holds
for the map $\nu=\nu_2\circ \nu_1$.
\end{proof}

\subsection{Finite supports}
\label{subsec:fin supp}
We conclude this section with a result which is presumably
folklore.  For completeness, we include a proof.

\begin{prop}
\label{prop:supp finite}
Let $A$ be a finitely generated abelian group, and let
$M$ be a finitely generated module over the group ring $S=\C[A]$.
Then $M$ as a $\C$-vector space is finite-dimensional if
and only if $\supp(M)$ is finite.
\end{prop}

\begin{proof}
Since $M$ is a finitely generated $S$-module,
we can argue by induction on the number of
generators of $M$.  Using the short exact sequence
$0 \to  \langle m \rangle \to  M \to M/ \langle m \rangle \to  0$,
where $m$ is a generator of $M$, and the fact that
$\supp(M)=\supp( \langle m \rangle) \cup \supp(M/ \langle m \rangle)$,
we see that it suffices to consider the case when $M$ is a cyclic module.
In this case, $M=S/\ann(M)$ and $\supp(M)=Z(\ann(M))=\Spm(S/\ann(M))$.
From the assumption that $\supp{M}$ is finite, and
using the Noether Normalization lemma, we infer that
$S/\ann(M)$ is an integral extension of $\C$.  Thus,
$\dim_{\C} (S/\ann(M))<\infty$.

Conversely, suppose $\supp(M)$ is infinite. Then $\Spm(S/\ann(M))$ is infinite,
which implies $\Spm(S/\ann(M))$ has positive dimension.  Choose a prime ideal $\p$
containing $\ann(M)$, such that the Krull dimension of $S/\p$
is positive.  From the condition that $\dim_{\C}S/\ann(M)<\infty$,
we deduce that $\dim_{\C}S/\p<\infty$.  By the Noether Normalization
lemma, $S/\p$ is an integral extension of $\C[x_1, \dots , x_n]$,
with $n>0$. Thus, $\dim_{\C}S/\p=\infty$. This is a contradiction,
and so we are done.
\end{proof}

\section{Characteristic varieties and generalized Dwyer--Fried sets}
\label{sect:char-var}

In this section, we finally tie together several strands,
and show how to determine the sets $\Omega^i_A(X)$
in terms of the jump loci for homology in rank $1$ local systems
on $X$.

\subsection{The equivariant chain complex}
\label{seubsec:cx}

Let $X$ be a connected CW-complex. As usual, we will
assume that $X$ has finite $k$-skeleton, for some $k\ge 1$.
Without loss of generality, we may assume that $X$ has a single
$0$-cell $e^0$, which we will take as our basepoint $x_0$.
Moreover, we may assume that all attaching maps
$(S^i, *) \to (X^i, x_0) $ are basepoint-preserving.  Let
$G =\pi_1(X, x_0)$ be the fundamental group of $X$, and denote by
$(C_i(X, \C), \partial_i)_{i \geq 0}$ the cellular chain complex of $X$,
with coefficients in $\C$.

Let $p\colon X^{\ab} \to X$ be the universal abelian cover.
The cell structure on $X$ lifts in a natural
fashion to a cell structure on $X^{\ab}$. Fixing a lift
$\tilde{x}_0 \in p^{-1}(x_0)$ identifies the group
$H=G_{\ab}$ with the group of deck transformations of
$X^{\ab}$, which permute the cells. Therefore, we may view
the cellular chain complex $C_{\bullet}=C_{\bullet} (X^{\ab}, \C)$
as a chain complex of left-modules over the group algebra $R=\C[H]$.
This chain complex has the form
\begin{equation}
\label{eq:equiv cc}
\xymatrixcolsep{16pt}
\xymatrix{\cdots \ar[r]
& C_{i} \ar^(.45){\tilde{\partial}_{i}}[r]
& C_{i-1} \ar[r] & \cdots  \ar[r]
& C_{2} \ar^(.45){\tilde{\partial}_{2}}[r]
& C_{1} \ar^(.45){\tilde{\partial}_{1}}[r]
& C_0
},
\end{equation}

The first two boundary maps can be written down explicitly.
Let $e^1_1, \dots , e^1_m$ be the $1$-cells of $X$. Since
we have a single $0$-cell, each $e^1_i$ is a loop,
representing an element $x_i\in G$.
Let $ \tilde{e}^0= \tilde{x}_0$, and let $\tilde{e}^1_i$ be the lift of
$e^1_i$ at $\tilde{x}_0$; then $\tilde{\partial}_1(\tilde{e}^1_i)= (x_i-1) \tilde{e}^0$.
Next, let $e^2$ be a $2$-cell, and let $\tilde{e}^2$ be its lift at $\tilde{x}_0$;
then
\begin{equation}
\label{eq:fox}
\tilde{\partial}_2 (\tilde{e}^2) =
\sum_{i=1}^{m} \phi \big(\partial r/\partial x_i \big) \cdot  \tilde{e}^1_i,
\end{equation}
where $r$ is the word in the free group $F_m=\langle x_1,\dots , x_m\rangle$
determined by the attaching map of the $2$-cell, $\partial r/\partial x_i\in \C[F_m]$
are the Fox derivatives of $r$, and $\phi\colon \C[F_m] \to \C[H]$
is the extension to group rings of the projection map
$F_m \to G \xrightarrow{\ab} H$, see  \cite{Fo}.

\subsection{Characteristic varieties}
\label{seubsec:cvs}
Since $X$ has finite $1$-skeleton, the group $H=H_1(X,\Z)$
is finitely generated, and its dual, $\wH=\Hom(H,\C^*)$, is a
complex algebraic group.
As is well-known, the character group $\wH$ parametrizes rank $1$
local systems on $X$: given a character $\rho\colon H \to \C^*$,
denote by $\C_{\rho}$ the $1$-dimensional $\C$-vector space,
viewed as a right $R$-module
via $a \cdot g = \rho(g)a$, for $g \in H$ and $a \in \C$.
The homology groups of $X$ with coefficients in $\C_{\rho}$
are then defined as
\begin{equation}
\label{eq:twist hom}
 H_i(X, \C_{\rho}) := H_i \big(C_{\bullet} (X^{\ab}, \C)
 \otimes_{R}\C_{\rho}\big).
\end{equation}

\begin{definition}
\label{def:cv}
The {\em characteristic varieties}\/ of $X$ (over $\C$) are the  sets
\[
V^i (X) = \left\{ \, \rho \in \Hom(H, \C^*) \mid \dim_{\C} H_j(X,
\C_{\rho}) \ne 0\: \text{ for some $1 \leq j\leq i$}\, \right\},
\]
\end{definition}

The identity component of the character group $T=\wH$
is a complex algebraic torus, which we will denote by $T_0$.
Let $\oH=H/\Tors(H)$ be the maximal torsion-free quotient of $H$.
The projection map $\pi\colon H \surj \oH$
induces an identification $\hat\pi \colon \widehat{\oH}  \isom \wH_0$.
Denote by $W^i(X)$ the intersection of $V^i(X)$ with $T_0=\wH_0$.
If $H$ is torsion-free, then $W^i(X)=V^i(X)$; in general,
though, the two varieties differ.

For each $1 \le i \le k$, the set $V^i(X)$ is a Zariski closed
subset of the complex algebraic group $T$, and $W^i(X)$
is a Zariski closed subset of the complex algebraic torus $T_0$.
Up to isomorphism, these varieties depend only on the homotopy
type of $X$.  Consequently, we may define the characteristic
varieties of a group $G$ admitting a classifying space $K(G,1)$
with finite $k$-skeleton as $V^i(G)=V^i(K(G,1))$, for $i\le k$.
It is readily seen that $V^1(X)=V^1(\pi_1(X))$.
For more details on all this, we refer to \cite{Su14}.

The characteristic varieties of a space can be reinterpreted
as the support varieties of its Alexander invariants,
as follows.

\begin{theorem}[\cite{PS-plms}]
\label{thm:supp cv}
For each $1\le i\le k$, the characteristic variety $V^{i}(X)$
coincides with the support of the $\C[H]$-module
$\bigoplus_{j=1}^i H_{j} (X^{\ab}, \C)$, while
$W^{i}(X)$ coincides with the support of the
$\C[\oH]$-module $\bigoplus_{j=1}^i H_{j} (X^{\fab}, \C)$.
\end{theorem}

\subsection{The first characteristic variety of a group}
\label{subsec:v1g}
Let $G$ be a finitely presented group.
The chain complex \eqref{eq:equiv cc} corresponding to
a presentation
$G= \langle x_1, \dots, x_q \mid r_1, \dots, r_m \rangle$
has second boundary map, $\tilde\partial_2$, an $m$ by $q$
matrix, with rows given by \eqref{eq:fox}.  Making use of
Theorem \ref{thm:supp cv}, we see that
$V^1(G)$ is defined by the vanishing of the
codimension~$1$ minors of the Alexander matrix
$\tilde\partial_2$, at least away from the trivial
character $1$.  This interpretation allows us to construct
groups with fairly complicated characteristic varieties.

\begin{lemma}
\label{lem:realize}
Let $f=f(t_1,\dots, t_n)$ be a Laurent polynomial with integral coefficients. 
There is then a finitely presented group $G$ with
$G_{\ab}=\Z^n$ and $V^1(G)= \{z\in (\C^{*})^n \mid f(z)=0\} \cup \{1\}$.
\end{lemma}

\begin{proof}
Let $F_n=\langle x_1,\dots, x_n\rangle$ be the
free group of rank $n$, with abelianization map
$\ab\colon F_n\to \Z^n$, $x_k \mapsto t_k$.
Recall the following result of R.~Lyndon (as
recorded in \cite{Fo}):  if $v_1, \dots ,v_n$ are elements
in the ring $\Z[\Z^n]=\Z[t_1^{\pm 1},\dots ,t_n^{\pm 1}]$,
satisfying the equation $\sum_{k=1}^n (t_k -1) v_k=0$, then
there exists an element $r\in F_n'$ such that
$v_k= \ab(\partial r/\partial x_k)$, for $1\le k \le n$.

Making use of this result, we may find elements $r_{i,j}\in F_n'$,
$1\le i<j\le n$ such that
\[
\ab(\partial r_{i,j}/\partial x_k) =
\begin{cases}
f \cdot (t_i-1), & \text{if $k=i$}\\
f \cdot (1-t_j), & \text{if $k=j$}\\
0, & \text{otherwise}.
\end{cases}
\]
It is now readily checked that the group
$G$ with generators $x_1, \dots, x_n$ and relations
$r_{ij}$ has the prescribed first characteristic variety.
\end{proof}

In certain situations, one may realize a Laurent
polynomial as the defining equation for the characteristic
variety by a more geometric construction.

\begin{example}
\label{ex:links}
Let $L$ be an $n$-component link in $S^3$, with
complement $X$.  Choosing orientations on the link
components yields a meridian basis for $H_1(X,\Z)=\Z^n$.
Then
\begin{equation}
\label{eq:v1link}
V^1(X)= \{z\in (\C^{*})^n \mid \Delta_L(z)=0\} \cup \{1\},
\end{equation}
where $\Delta_L=\Delta_L(t_1,\dots ,t_n)$ is
the (multi-variable) Alexander polynomial of the link.
\end{example}

\subsection{A formula for the generalized Dwyer--Fried sets}
\label{subsec:gen df}

Recall that, in Definition \ref{def:omega w} we associated to
each subvariety $W\subset \widehat{H}$, and each abelian
group $A$ a subset
\begin{equation}
\label{eq:ups again}
\Upsilon_{A}(W)= \big\{\, [\nu]\in \Gamma(H,A) \mid
\dim( \im(\hat\nu) \cap W) >0\, \big\}.
\end{equation}
The next theorem expresses the Dwyer--Fried sets
$\Omega^{i}_{A}(X)$ in terms of the $\Upsilon$-sets
associated to the $i$-th characteristic variety of $X$.

\begin{theorem}
\label{main theorem}
Let $X$ be a connected CW-complex, with finite $k$-skeleton.
Set $G=\pi_1(X,x_0)$ and $H=G_{\ab}$. For any abelian
group $A$, and for any $i\le k$,
\[
\Omega^{i}_{A}(X)= \Gamma(H,A) \setminus \Upsilon_A (V^{i}(X)) .
\]
\end{theorem}

\begin{proof}
Fix an epimorphism $\nu \colon H\surj A$, and let $X^{\nu}\to X$
be the corresponding cover.  Recall the cellular chain complex
$C_{\bullet}=C_{\bullet}(X^{\ab},\C)$ is a chain complex of left modules
over the ring $R=\C[H]$.  If we set $S=\C[A]$, the cellular
chain complex $C_{\bullet}(X^{\nu},\C)$ can be written as
$C_{\bullet} \otimes_R S$, where $S$ is viewed as a right
$R$-module via extension of scalars by $\nu$.

Consider the $S$-module
\[
M=\bigoplus_{j=1}^i H_j(X^{\nu},\C)=
\bigoplus_{j=1}^i H_j(C_\bullet \otimes_R S).
\]
By definition, $[\nu]$ belongs to $\Omega^{i}_{A}(X)$ if and only if the
Betti numbers $b_1(X^{\nu}), \dots , b_i(X^{\nu})$ are all finite, i.e.,
$\dim_{\C} M< \infty$.  By Proposition \ref{prop:supp finite}, this
condition is equivalent to $\supp M$ being finite.

Now let $\nu^*\colon  \Spm(S) \inj \Spm(R)$ be the induced morphism
between the corresponding affine schemes.  We then have
\begin{align*}
\supp M & =
(\nu^*)^{-1} \supp\Big( \bigoplus_{j=1}^i H_j (C_{\bullet}) \Big)
&&  \text{by Theorem \ref{thm:supp}}\\
&\cong \im (\hat{\nu})\cap Z\Big(\ann\Big(\bigoplus_{j=1}^i
H_j (X^{\ab}, \C)\Big)\Big)
&&  \text{by Lemma \ref{lem:suppz}} \\
& = \im (\hat{\nu})\cap V^i(X)
&&  \text{by Theorem \ref{thm:supp cv}}.
\end{align*}
This ends the proof.
\end{proof}

\begin{remark}
\label{rem:tors}
If $H$ is torsion-free, then $W^{i}(X)=V^{i}(X)$;
thus, the set $\Omega^{i}_{A}(X)$ depends only the variety $W^{i}(X)$
and the abelian group $A$.
On the other hand, if $\Tors(H)\ne 0$, the variety
$W^{i}(X)$ may be strictly included in $V^{i}(X)$, in which case
the set $\Omega^{i}_{A}(X)$ may depend on information not carried
by $W^{i}(X)$.   We shall see examples of this phenomenon in
\S\ref{subsec:comparison}.
\end{remark}

\subsection{An upper bound for the $\Omega$-sets}
\label{subsec:upper bound}

We now give a computable ``upper bound" for
the generalized Dwyer-Fried sets $\Omega^i_A(X)$,
in terms of the sets introduced in \S\ref{subsec:lub}.

\begin{theorem}
\label{thm:omega bound}
Let $X$ be a connected CW-complex with finite $k$-skeleton.
Set $H=H_1(X,\Z)$, and fix a degree $i\le k$.
Let $A$ be a quotient of $H$.  Then
\begin{equation}
\label{eq:omega bound}
\Omega_A^i(X)\subseteq \Gamma(H,A)
\setminus  U_{A} (V^i(X)) .
\end{equation}
\end{theorem}

\begin{proof}
Follows at once from Proposition \ref{prop:om bd} and
Theorem \ref{main theorem}.
\end{proof}

In general, inclusion \eqref{eq:omega bound}  is not an
equality.  Indeed, the fact that $V(\ker\nu)\cap V^i(X)$ is infinite
cannot guarantee there is an algebraic subgroup in the intersection.
Here is a concrete instance of this phenomenon.

\begin{example}
\label{example:ruled}
Using Lemma \ref{lem:realize}, we may find a $3$-generator, $3$-relator
group $G$ with abelianization $H=\Z^3$ and characteristic variety
\[
V^1(G)=\big\{ (t_1, t_2, t_3) \in (\C^*)^3 \mid (t_2-1)=(t_1+1)(t_3-1) \big\}.
\]
This variety has a single irreducible component, which
is a complex torus passing through the origin; nevertheless,
this component does not embed as an algebraic subgroup in $(\C^*)^3$.

There are $4$ maximal, positive-dimensional torsion-translated subtori
contained in $V^1(G)$, namely
$\eta_1 V(\overline{\xi_1}), \dots,
\eta_4 V(\overline{\xi_4})$,
where $\xi_1,\dots, \xi_4$ are the subgroups of $\Z^3$ given by
\[
\xi_1=
\im\left(\begin{smallmatrix}
  2 & 0  \\
  0 & 1  \\
  0 & 0  \\
\end{smallmatrix}\right), \quad
 \xi_2=
\im\left(\begin{smallmatrix}
  0 & 0  \\
  1 & 0  \\
  0 & 1  \\
\end{smallmatrix}\right), \quad
 \xi_3=
\im\left(\begin{smallmatrix}
  2 & -1  \\
  -1 & 0  \\
  0 & 1  \\
\end{smallmatrix}\right),\quad
 \xi_4=
\im\left(\begin{smallmatrix}
  2 & -2  \\
  -1 & 0  \\
  0 & 2  \\
\end{smallmatrix}\right),
\]
and $\eta_1=(-1, 1, 1)$, $\eta_2=\eta_3=(1, 1, 1)$, $\eta_4=(-1, 1, -1)$.

We claim that, for $A=\Z^2\oplus \Z_2$,
inclusion \eqref{eq:omega bound} from
Theorem \ref{thm:omega bound} is strict, i.e.,
\[
\Omega^1_{A}(G) \subsetneqq \Gamma(H,A) \setminus U_{A} (V^1(G)).
\]

To prove this claim, consider the epimorphism
$\nu\colon \Z^3 \surj \Z^2 \oplus \Z_2$ given by the matrix
$\left(\begin{smallmatrix}
  1 & 0  & 0\\
  0 & 0 & 1  \\
  0 & 1 & 0  \\
\end{smallmatrix}\right)$. Note that $\ker(\nu) = \im\left(\begin{smallmatrix}
  0   \\
  2   \\
  0  \\
\end{smallmatrix}\right)$, and so
$\im(\hat{\nu})=\{t \in (\C^{*})^{3} \mid t_2=\pm 1\}$.  The intersection
of $\im(\hat{\nu})$ with $V^1(G)$ consists of all points of the form
$(t_1,\pm 1, t_3)$ with $(t_1+1)(t_3-1)$ equal to $0$ or $-2$.
Clearly, this is an infinite set; therefore, $[\nu] \notin \Omega_{A}^1(G)$.

On the other hand, $\im(\hat{\nu}) \cap \eta_1 V(\overline{\xi_1})=
\im(\hat{\nu}) \cap \eta_2 V(\overline{\xi_2})=\emptyset$, and
$\rank(\ker(\nu)+ \xi_3)=\rank(\ker(\nu)+ \xi_4)=\rank(H)$.
Hence, $\im(\hat{\nu}) \cap \eta_j V(\overline{\xi_j})$ is finite,
for all $j$.  In view of Lemma \ref{lem:uaw}, we conclude that
$[\nu] \notin U_{A} (V^1(G))$.
\end{example}

\section{Comparison with the classical Dwyer--Fried invariants}
\label{sect:classic df}

\subsection{The Dwyer--Fried invariants $\Omega^{i}_r(X)$}
\label{subsec:standard df}

In the case of free abelian covers and the usual $\Omega$-sets,
Theorem \ref{main theorem} allows us to recover
the following result from \cite{DF}, \cite{PS-plms}, \cite{Su14}.

\begin{corollary}
\label{cor:df x}
Set $n=b_1(X)$.  Then, for all $r\ge 1$,
\[
\Omega^{i}_r(X)=\big\{ [\nu] \in \Grass_r(\Z^n) \mid
\text{$\im (\hat\nu) \cap W^i(X)$ is finite} \big\}.
\]
\end{corollary}

Using now the identification from Example \ref{ex:rk1},
and taking into account Lemma \ref{lemma:tau1},
Theorem \ref{thm:omega bound}
yields the following corollary.

\begin{corollary}[\cite{PS-plms, Su14}]
\label{cor:bptau1}
Let $X$ be a connected CW-complex with finite $k$-skeleton, and
set $n=b_1(X)$.  Then $\Omega^i_r(X) \subseteq \QP^{n-1} \setminus
\bP(\tau_1(W^i(X)))$, for all $i\le k$ and all $r\le n$.
\end{corollary}

\begin{remark}
\label{rem:open}
As noted in \cite{Su14}, if the variety $W^i(X)$ is
a union of algebraic subtori, then the Dwyer--Fried sets
$\Omega_r^i(X)$ are open subsets of $\Grass_r(\Z^n)$,
for all $r\ge 1$. In general, though, examples from \cite{DF},
\cite{Su14} show that the sets $\Omega_r^i(X)$ with $r>1$
need not be open.
On the other hand, as noted in \cite{DF, PS-plms, Su14},
the sets $\Omega^i_1(X)$ are always open subsets
of $ \Grass_1(\Z^n) =\QP^{n-1}$.
We will come back to this phenomenon in \S\ref{sect:rank1},
in a more general context.
\end{remark}

\subsection{The comparison diagram, revisited}
\label{subsec:comp again}

As may be expected, the generalized Dwyer--Fried invariants
carry more information about the homotopy type of a space
and the homological finiteness properties of its regular abelian
covers than the classical ones.  To make this more precise,
let $A$ be a quotient of the group $H=H_1(X,\Z)$. Fix a decomposition
$A=\oA\oplus \Tors(A)$, and identify $\Omega^i_{r}(X)
= \Omega^i_{\oA}(X)$, where $r=\rank (A)$.
As we saw in \S\ref{subsec:compare}, for each $i\le k$
we have a comparison diagram
\begin{equation}
\label{eq:comp}
\xymatrix{ \Omega_A^i(X)\, \ar@{^{(}->}[r] \ar[d]^{q|_{\Omega_A^i(X)}}
& \Gamma(H, A) \ar[d]^{q} \\
\Omega^i_{r}(X)\, \ar@{^{(}->}[r] &
\Gamma(H, \oA)
}
\end{equation}
between the respective Dwyer--Fried invariants, viewed as
subsets of the parameter sets for regular $A$-covers
and $\oA$-covers, respectively.

We are interested in describing conditions under which
the set $\Omega^i_A(X)$ contains more information
than $\Omega^i_r(X)$.
This typically happens when the comparison diagram
\eqref{eq:comp} is not a pull-back diagram, i.e.,
there is a point
$[\bar\nu]\in \Omega^{i}_{r}(X)$ for which the fiber
$q^{-1}([\bar\nu])$ is not included in $\Omega^{i}_{A}(X)$.
In fact, the number of points in the fiber which lie in
$\Omega^{i}_{A}(X)$ may vary as we move about
$\Omega^{i}_{r}(X)$.

In view of Theorem \ref{main theorem}, we have the following
criterion.

\begin{prop}
\label{prop:cv comp}
Diagram \eqref{eq:comp} fails to be a pull-back diagram
if and only if $q(\Upsilon_A(V^i(X))$ is not included in
$\Upsilon_r(V^i(X)$, i.e., there is an epimorphism $\nu\colon H\surj A$
such that
\[
\dim ( \im \hat\nu \cap V^i(X) )>0, \quad\text{yet}\quad
\dim ( \im \hat{\bar\nu} \cap V^i(X)) =0.
\]
\end{prop}

From Proposition \ref{prop:pullbetti}, we know diagram  \eqref{eq:comp}
is a pull-back diagram precisely when the homological
finiteness of an arbitrary $A$-cover of $X$ can be tested through the
corresponding $\overline{A}$-cover.
In order to quantify the discrepancy between these two types of
homological finiteness properties, let us define the ``singular set"
\begin{equation}
\label{eq:sigma}
\Sigma_A^i(X)=\big\{ [\bar\nu] \in \Omega_{\oA}^i(X) \mid
\#(q^{-1}([\bar\nu]) \cap \Omega_A^i(X) ) < \#(q^{-1}([\bar\nu])) \big\}.
\end{equation}
We then have:
\begin{equation}
\label{eq:sigma again}
\Sigma_A^i(X)= q\left( \Omega_A^i(X) ^c \right) \cap
 \Omega_{\oA}^i(X).
\end{equation}

\subsection{Maximal abelian versus free abelian covers}
\label{subsec:comparison}

We now investigate the relationship between the finiteness of the
Betti numbers of the maximal abelian cover $X^{\ab}$ and the
finiteness of the Betti numbers of the corresponding free abelian
cover $X^{\fab}$ of our space $X$.

As before, write $H=H_1(X,\Z)$
and identify the character group $\wH$ with $(\C^*)^n \times \Tors(H)$,
where $n=b_1(X)$.

\begin{prop}
\label{prop:w1v1}
Suppose $\Tors(H) \ne 0$.
Furthermore, assume that $W^{i}(X)$ is finite, whereas
$V^{i}(X)=W^i(X)\cup (\wH \setminus \wH_{0})$.  Then
\begin{romenum}
\item \label{g1}
$\Omega^{i}_{r}(X)=\Grass_{r}(\Z^n)$, for all $r\ge 1$.
\item \label{g2}
If $\rank(A)=\rank(H)$ and $\Tors(A)\ne 0$, then
$\Omega^{i}_{A}(X)=\emptyset$.
\end{romenum}
\end{prop}

\begin{proof}
By Theorem \ref{main theorem}, an element $[\nu] \in \Grass_{r}(\Z^n)$
belongs to $\Omega^{i}_{r}(X)$ if and only if $\im(\hat{\nu}) \cap W^i(X)$
is finite. By assumption, $W^{i}(X)$ is finite; thus, the intersection
$\im(\hat{\nu}) \cap W^i(X)$ is also finite. This proves \eqref{g1}.

Again by Theorem \ref{main theorem}, an element $[\nu] \in \Gamma(H,A)$
belongs to $\Omega^{i}_{A}(X)$ if and only if $\im(\hat{\nu}) \cap V^i(X)$
is finite. By assumption,
$V^{i}(X)=W^i(X)\cup (\wH \setminus \wH_{0})$; moreover,
$\rank(A)=\rank(H)$ and $\Tors(A)\ne 0$.  Thus,
the intersection $\im(\hat{\nu}) \cap V^i(X)$ contains at least one
component of $\wH \setminus \wH_{0}$, which is infinite.
This proves \eqref{g2}.
\end{proof}

A similar argument yields the following result.
\begin{prop}
\label{prop:fab f ab inf}
Suppose $H_1(X, \Z)$ has non-trivial torsion, $W^1(X)$ is finite,
and $V^1(X)$ is infinite. Then $b_1(X^{\fab})< \infty$, yet
$b_1(X^{\ab})=\infty$.
\end{prop}

\begin{prop}
\label{prop:vee}
Suppose $X=X_1\vee X_2$, where $H_1(X_1,\Z)$ is free abelian
and non-trivial, and $H_1(X_2,\Z)$ is finite and non-trivial.  If
$V^1(X_1)$ is finite, then $b_1(X^{\fab})< \infty$, yet $b_1(X^{\ab})=\infty$.
\end{prop}

\begin{proof}
Let $H=H_1(X,\Z)$; then $\oH=H_1(X_1,\Z)\ne 0$ and
$\Tors(H)=H_1(X_2,\Z)\ne 0$. Thus, the character group
$\wH$ decomposes as $\wH_0 \times \Tors(H)$, with
both factors non-trivial.

Now,  $W^1(X) =V^1(X_1)\times \{1\}$ is finite, and thus
$\Omega^1_{\oH}(X)$ is a singleton.  On the other hand,
$V^1(X) = W^1(X) \cup \wH_0 \times (\Tors(H) \setminus\{1\})$
is infinite, and thus $\Omega^1_{H}(X)=\emptyset$.
\end{proof}

\begin{example}
\label{ex:s1rp2}
Consider the CW-complexes $X=S^1\vee \RP^2$ and
$Y=S^1\times \RP^2$.  Then $H_1(X, \Z)\cong
H_1(Y, \Z) \cong \Z \times \Z_2$. Clearly, the free
abelian covers $X^{\fab}$ and $Y^{\fab}$ are
rationally acyclic; thus both
$\Omega^i_1(X)$ and $\Omega^i_1(Y)$ consist
of a single point, for all $i\geq 0$.
On the other hand, the free abelian cover $X^{\ab}$ has
the homotopy type of a countable wedge of $S^1$'s and
$\RP^2$'s, whereas $Y^{\ab} \simeq S^2$. Therefore,
$\Omega^1_{\Z\oplus\Z_2}(X)=\emptyset$, while
$\Omega^1_{\Z\oplus \Z_2}(Y)=\{\text{point}\}$.
\end{example}

This example shows that the generalized Dwyer--Fried
invariants $\Omega^{i}_{A}(X)$ may contain
more information than the classical ones.

\section{The rank $1$ case}
\label{sect:rank1}

In this section, we discuss in more detail the invariants
$\Omega^i_A(X)$ in the case when $A$ has rank $1$,
and push the analysis even further in some
particularly simple situations.

\subsection{A simplified formula for $\Omega_A^i(X)$}
\label{subsec:omega rank1}
Recall that, for a finitely generated abelian group $A$,
the integer $c(A)$ denotes the largest order of an
element in $A$.  Recall also that, for every subvariety
$W\subset \wH$ and each index $d\ge 1$, we have a
subset $U_{A,d} (W) \subset \Gamma(H,A)$, described
geometrically in Lemma \ref{lem:uaw}.

\begin{theorem}
\label{thm:rank1 omega}
Let $X$ be a connected CW-complex with finite $k$-skeleton.
Set $H=H_1(X,\Z)$, and fix a degree $i\le k$.
If $\rank(A)=1$, then
\begin{equation}
\label{eq:rank1 omega}
\Omega_A^i(X)=\Gamma(H,A)
\setminus  U_{A,c(A)}(V^i(X)).
\end{equation}
\end{theorem}

\begin{proof}
The inclusion $\subseteq$ follows from Theorem \ref{thm:omega bound},
so we only need to prove the opposite inclusion.

Let $\nu\colon H \surj A$ be an epimorphism such that
$[\nu] \notin \Omega_A^i(X)$.  By Theorem \ref{main theorem},
the variety $\im \hat{\nu} \cap V^i(X)$ has  positive dimension.
Since $\rank(A)=1$, there exists a component of the
$1$-dimensional algebraic subgroup $\im \hat{\nu}$
contained in $ V^i(X)$; that is, there
exists a character $\rho \in \widehat{\T(A)}$ such that
\[
\hat{\nu}(\rho) \cdot \im \hat{\bar\nu} \subseteq V^i(X).
\]

Now, we may find a primitive subgroup $\chi\le H$
and a torsion character $\eta\in \wH$ such that
$\eta V(\chi)=\hat{\nu}(\rho) V(\chi)$
is a maximal translated subtorus in $V^i(X)$ which contains
$\hat{\nu}(\rho) \cdot \im \hat{\bar\nu}$.
Set
\[
\xi=\epsilon\, \bigg(\bigcup_{m\ge 1} \eta^m V(\chi)\bigg).
\]
Clearly, the subgroup $\xi\le H$ belongs to $\Xi_d(V^i(X))$,
where $d:=\ord(\rho)$ divides $c(A)$.
Since $\hat{\nu}(\rho) V(\chi)\supseteq \hat{\nu}(\rho) \cdot V(\ker(\bar\nu))$,
we must also have
$V(\chi) \supseteq V(\ker(\bar\nu))$, and so $[\nu]\in U_A(\xi)$.
Therefore, $[\nu]\in U_{A,c(A)}(V^i(X))$, and we are done.
\end{proof}

\begin{corollary}[\cite{PS-plms, Su14}]
\label{cor:omegai1}
Let $X$ be a connected CW-complex with finite $k$-skeleton, and
set $n=b_1(X)$.  Then $\Omega^i_1(X) = \QP^{n-1} \setminus
\bP(\tau_1(W^i(X)))$, for all $i\le k$.
\end{corollary}

In particular, $\Omega^i_1(X)$ is an open subset of the
projective space $\QP^{n-1}$.

\subsection{The singular set}
\label{subsec:sing set}

Recall from \S\ref{subsec:comp again} that we measure the discrepancy
between the generalized Dwyer--Fried invariant $\Omega_A^i(X)$ and
its classical counterpart, $\Omega_{\oA}^i(X)$, by means of the
``singular set,"
$\Sigma_A^i(X)= q\big( \Omega_A^i(X) ^c \big) \cap
\Omega_{\oA}^i(X)$.  When the group $A$ has rank $1$, this set
can be expressed more concretely, as follows.

\begin{prop}
\label{prop:rank1}
Suppose $H$ is torsion-free, and $\rank(A)=1$.  Then $\Sigma_A^i(X)$
consists of all $[\sigma] \in \Gamma(H,\overline{A})$ satisfying the
following two conditions:
\begin{romenum}
\item \label{s1}
$\ker(\sigma)\supseteq \overline{\xi}$, for some $\xi\in \Xi_{c(A)} (V^i(X))$,
and
\item \label{s2}
$V(\ker(\sigma))\nsubseteq V^i(X)$.
\end{romenum}
\end{prop}

\begin{proof}
Without loss of generality, we may assume that
$\Gamma(H, A)\neq \emptyset$.  From Theorem \ref{thm:rank1 omega},
we know that $\Omega^i_A(X)=\Gamma(H,A)\setminus U$, where
$U=U_{A,c(A)}(V^i(X))$.
It follows that $\Sigma_A^i(X)=q(U)\cap \Omega_1^i(X)$.

Now let $S$ be the set of all $[\sigma] \in \Gamma(H,\overline{A})$
satisfying conditions \eqref{s1} and \eqref{s2}.
It suffices to show that $q(U)\cap \Omega_1^i(X)=S$.

The inclusion $S \supseteq q(U)\cap \Omega_1^i(X)$ is straightforward.
To establish the reverse inclusion, let $\sigma\colon H \surj \Z$
represent an element in $S$. Condition \eqref{s2} implies that
$[\sigma] \in \Omega_1^i(X)$.
Recall that $\pi\colon A\to \oA$ is the natural projection. To prove
that $[\sigma]\in q(U)$, it is enough to find an epimorphism
$\nu\colon H \surj A$ such that $\pi\circ\nu=\sigma$ and
$V(\ker(\nu)) \cap \eta V(\overline{\xi}) \neq \emptyset$, where
$\eta$ is a generator of $\widehat{\overline{\xi}/\xi}$.

Set $d:=\ord(\eta)$.  Writing $A = \Z_{d_1} \oplus\cdots\oplus \Z_{d_k}$,
with $d_1 \dv d_2 \dv \cdots \dv d_k$, we have that
$d \dv d_k$. Denote by $\iota$ the embedding of the cyclic group
$\langle\eta\rangle$ into $\wH$.
Under the correspondence from \eqref{eq:pont dual},
there is a map $\check{\iota}\colon H \surj \Z_d$; clearly,
this map factors as the composite
\begin{equation*}
\xymatrixcolsep{25pt}
\xymatrix{ H \ar[r]^(.4){f_1} & \Z^{n-1}\oplus \Z\ar@{->>}[r]^{\id \oplus \kappa_1}
           & \Z^{n-1}\oplus \Z_{d_k}\ar@{->>}[r]^(.6){(0\:\: \kappa_2)}&\Z_d},
\end{equation*}
for some isomorphism $f_1\colon H \to \Z^n$, where $\kappa_1$
and $\kappa_2$ are the canonical projections.

Recall we are assuming $\Gamma(H,A)\neq\emptyset$; thus, there is
an epimorphism $\gamma\colon H \surj A$. We then have a
commuting diagram
\begin{equation*}
\xymatrix{ H \, \ar@{->>}[r]^{\gamma} \ar[d]^{f_2} & A \ar@{->>}[d] \\
           \Z^{n-1}\oplus \Z\, \ar@{->>}[r]^(.57){(0\:\: \kappa_1)} & \Z_{d_k}}
\end{equation*}
for some isomorphism $f_2\colon H \to \Z^n$.
The composite $\nu=\gamma \circ f_2^{-1}\circ f_1\colon H \surj A$, then,
is the required epimorphism.
\end{proof}

As we shall see in Examples \ref{ex:421} and \ref{example:ruled surface},
the singular set $\Sigma^i_A(X)$ may be non-empty; in fact,
as we shall see in Example~\ref{strata infinite}, this set
may even be infinite.

\subsection{A particular case}
\label{subsec:z2}

Perhaps the simplest situation when such a phenomenon
may occur is the one when $H=\Z^2$ and $A=\Z\oplus \Z_2$.
In this case, the set
\[
\Gamma(H/\oA, A/\oA)=\Epi(\Z, \Z_2)/ \Aut(\Z_2)
\]
is a singleton, and so the map
$q_H\colon \Gamma(H,A)\to \Gamma(H,\oA)$ is a bijection.
In other words, if $H_1(X,\Z)=\Z^2$, there is a
one-to-one correspondence between regular
$\Z\oplus \Z_2$-covers of $X$ and regular
$\Z$-covers of $X$, both parametrized by
 the projective line $\QP^1=\Grass_1(\Z^2)$.
The comparison diagram, then, takes the form
\begin{equation}
\label{eq:z2}
\xymatrix{ \Omega^i_{\Z\oplus \Z_2}(X)\,
\ar@{^{(}->}[r] \ar@{^{(}->}[d]_{q |_{\Omega^i_{\Z\oplus \Z_2}(X)}} & \QP^1 \ar^{=}[d]\\
\Omega^i_{1}(X)\, \ar@{^{(}->}[r] &\QP^1
}
\end{equation}

Let $V^i(X)\subset (\C^*)^2$ be the $i$-th characteristic variety of $X$.
For each pair $(a,b)\in \Z^2$, consider the (translated) subtori
$T_{a,b}^{\pm}=\{(t_1,t_2)\in(\C^*)^2 \mid t_1^at_2^b=\pm 1\}$.

\begin{prop}
\label{prop:z2z2}
Suppose $H=\Z^2$ and $A=\Z\oplus \Z_2$.
Then,
\begin{align*}
\Omega^i_A(X) &=\{ (a,b)\in \QP^1 \mid  \text{$T_{-b,a}^{+} \subset V^i(X)$
or  $T_{-b,a}^{-} \subset V^i(X)$}\} ^c\\
\Omega^i_1(X) &=\{ (a,b)\in \QP^1 \mid  T_{-b,a}^{+} \subset V^i(X) \} ^c.
\end{align*}
\end{prop}

\begin{proof}
Let $(a,b)\in H$ and let $\xi\le H$ be the subgroup generated by $(-b,a)$.
Then $\xi\in \Xi_1 (V^i(X))$ if and only if $\xi=\overline{\xi}$ and $V^i(X)$
contains a component of the form $ t_1^{-b} t_2^{a} =1$,
whereas $\xi\in \Xi_2 (V^i(X))$ if and only if $\xi$ has index at most $2$ in
$\overline{\xi}$ and $V^i(X)$
contains a component of the form $ t_1^{-b} t_2^{a} =\pm 1$.
The conclusions follow from Theorem \ref{thm:rank1 omega}.
\end{proof}

In particular, $\Sigma^i_{\Z\oplus \Z_2}(X) =
 \Omega^i_1(X) \setminus \Omega^i_{\Z\oplus \Z_2}(X)$, and
diagram \eqref{eq:z2} is not a pull-back diagram if and only if
$V^i(X)$ has a component of the form $t_1^{a} t_2^{b} + 1=0$.

A nice class of examples is provided by $2$-components
links.  Let $L=(L_1,L_2)$ be such a link, with complement $X_L$.
As we saw in Example \ref{ex:links}, the characteristic variety
$V^1(X_L)$ consists of the zero-locus in $(\C^*)^2$ of the Alexander
polynomial $\Delta_L(t_1,t_2)$, together with the identity.

\begin{example}
\label{ex:421}
Let $L$ be the $2$-component link denoted $4^2_1$ in
Rolfsen's tables, and let $X_L$ be its complement.
Then $\Delta_L=t_1+t_2$, and so
$V^1(X_L)=\{1\} \cup \{(t_1, t_2)\in  (\C^*)^2\mid  t_1t_2^{-1}=-1\}$.
Hence, $\Omega^1_{1}(X_L)= \QP^1$, but
$\Omega^1_{\Z\oplus \Z_2}(X_L)=\QP^1 \backslash \{(1, 1)\}$.
\end{example}

\subsection{Another particular case}
\label{subsec:lyndon}

The next simplest situation is the one when $H=\Z^3$
and $A=\Z\oplus \Z_2$. In this case, the set
\[
\Gamma(H/\oA, A/\oA)=\Epi(\Z^2, \Z_2)/ \Aut(\Z_2)=(\Z_2\oplus \Z_2)^*
\]
consists of $3$ elements.   In other words, if
$H_1(X,\Z)=\Z^3$, there is a three-to-one correspondence
between the regular $\Z\oplus \Z_2$-covers of $X$ and
the regular $\Z$-covers of $X$.

\begin{example}
\label{example:ruled surface}
Let $G$ be the group from Example \ref{example:ruled}.
With notation as before, we have that
$\Xi_1(V^1(G))=\{\xi_2, \xi_3 \}$ and
$\Xi_2(V^1(G))=\{ \xi_1, \xi_2, \xi_3, \xi_4\}$.  Consequently, 
$\bP(\tau_1(V^1(G)))=\{(1, 0, 0)$, $(1, 2, 1)\}$, and so
$\Omega^1_{1}(G)=\QP^2 \backslash \{(1, 0, 0), (1, 2, 1)\}$.

Let  $A=\Z\oplus \Z_2$.
Given an element $[\nu] \in \Gamma(H,A)$ such that
$\overline{\ker(\nu)}\supseteq \overline{\xi_1}$,
we have that $[\bar{\nu}]=(0, 0, 1) \in \Omega^1_{1}(G)$.  Furthermore,
$q^{-1}([\bar{\nu}])$ consists of $3$ representative classes:
$\nu_1=
\left(\begin{smallmatrix}
  0 & 0 & 1 \\
  1 & 0 & 0
\end{smallmatrix}\right)$,
$\nu_2=
\left(\begin{smallmatrix}
  0 & 0 & 1 \\
  1 & 1 & 0
\end{smallmatrix}\right)$, and
$\nu_3=
\left(\begin{smallmatrix}
  0 & 0 & 1 \\
  0 & 1 & 0
\end{smallmatrix}\right)$.
By calculation $\nu_1 \in U$, but $\nu_2, \nu_3 \notin U$.
Thus,
\[
\Omega^1_{A}(G)=\Gamma(H, A)
\setminus \{q^{-1}(1, 0, 0), q^{-1}(1, 2, 1), \nu_1\}.
\]
\end{example}

\begin{example}
\label{strata infinite}
Consider the group from \cite[Example 8.7]{Su14}, with
presentation
\[
G=\langle x_1, x_2, x_3\mid
 [x_1^2,x_2], \
[x_1,x_3], \
x_1[x_2,x_3]x_1^{-1}[x_2,x_3] \rangle.
\]
The characteristic variety $V^1(G)\subset (\C^*)^3$ consists
of the origin, together with the translated torus $\{ (t_1, t_2, t_3) \in
(\C^*)^3 \mid t_1=-1 \}$; hence,  $\Omega_1^1(G)=\QP^2$.

Let $A=\Z\oplus \Z_2$. The singular set $\Sigma = \Sigma^1_A(G)$
consists of those points $[(0, b, c)]\in\QP^2$ with  $b$ and $c$ coprime;
thus, $\Sigma \cong\QP^1$ is infinite.  Moreover, the restriction
$q\colon \Omega^1_A(G)\setminus q^{-1}(\Sigma)
\to \Omega^1_1(G)\setminus \Sigma$ is three-to-one.
On the other hand, if $[\nu] \in q^{-1}(\Sigma)$, then either
$\nu$ is of the form $\left(\begin{smallmatrix}
0 & b & c \\
1 & \epsilon_2 & \epsilon_3
\end{smallmatrix}\right)$,
in which case $[\nu] \notin \Omega_{A}^1(G)$,
or $\nu$ is of the form $\left(\begin{smallmatrix}
0 & b & c \\
0 & \epsilon_2 & \epsilon_3
\end{smallmatrix}\right)$, in which case $[\nu] \in \Omega_{A}^1(G)$.
Thus,  the restriction
$q\colon \Omega^1_A(G) \cap q^{-1}(\Sigma) \to \Sigma$
is one-to-one.
\end{example}

\section{Translated tori in the characteristic varieties}
\label{sect:translated}

Throughout this section, we assume all irreducible components of
the characteristic varieties under consideration are (possibly translated)
algebraic subgroups of the character group, a condition
satisfied by large families of spaces.

\subsection{A refined formula for the $\Omega$-sets}
\label{subsec:cvtt}

In Theorem \ref{main theorem} we gave a general
description of the Dwyer--Freed invariants $\Omega_A^i(X)$ in
terms of the characteristic variety $W=V^i(X)$, while in
Theorem \ref{thm:omega bound} we gave a upper bound
for those invariants, in terms of certain sets $\Xi_{d}(W)$
and $U_A(\xi)$, introduced in Definitions \ref{def:tau1} and
\ref{def:sigmaxi}, respectively.

We now refine those results in the special case when all
components of the characteristic variety are torsion-translated
subgroups of the character group.  The next theorem
shows that inclusion \eqref{eq:omega bound} from
Theorem \ref{thm:omega bound} holds as equality in
this case, with the union of all sets $U_{A,d} (W)$
with $d\ge 1$ replaced by a single constituent
$U_{A,c} (W)$, for some integer $c$ depending
only on $W$ (and not on $A$).

\begin{theorem}
\label{thm:tt omega}
Let $X$ be a connected CW-complex with finite $k$-skeleton.
Set $H=H_1(X,\Z)$, and fix a degree $i\le k$.
Suppose $V^i(X)$ is a union of torsion-translated subgroups
of $\wH$.  There is then an integer $c>0$ such that,
for every abelian group $A$,
\begin{equation}
\label{eq:tt omega}
\Omega_A^i(X)= \Gamma(H,A)
\setminus U_{A,c} (V^i(X)) .
\end{equation}
\end{theorem}

\begin{proof}
By assumption, $V^i(X)=\bigcup_{j=1}^s \eta_j V(\xi_j)$,
for some subgroups $\xi_j\le H$ and torsion elements
$\eta_j\in \wH$.  In the special case when $s=1$, the
required equality is proved in Theorem \ref{thm:uprho1},
with $c=\ord(\eta_1)\cdot c(\overline{\xi}_1/\xi_1)$.

The general case follows from a similar argument, with $c$
replaced by the lowest common multiple of
$\ord(\eta_1)\cdot c(\overline{\xi}_1/\xi_1), \dots ,
\ord(\eta_s)\cdot c(\overline{\xi}_s/\xi_s)$.
\end{proof}

\subsection{Another formula for the $\Omega$-sets}
\label{subsec:alt cvtt}

We now present an alternate formula for computing the
sets $\Omega^{i}_{A}(X)$ in the case when the $i$-th
characteristic variety of $X$ is a union of torsion-translated
subgroups of the character group.  Although somewhat
similar in spirit to Theorem \ref{thm:tt omega}, the next
theorem uses different ingredients to express the answer.

\begin{theorem}
\label{thm:xik}
Let $X$ be a connected CW-complex with finite $k$-skeleton.
Suppose there is a degree $i\le k$
such that  $V^i(X)=\bigcup_{j=1}^s \eta_j V(\xi_j)$, where
$\xi_1, \dots, \xi_s$ are subgroups of $H=H_1(X,\Z)$, and
$\eta_1, \dots, \eta_s$ are torsion elements in $\widehat{H}$.
Then, for each abelian group $A$,
\begin{equation}
\label{eq:omegaiax}
\Omega^{i}_{A}(X)=
 \bigcap_{j=1}^{s} \Bigg( \sigma_A(\xi_j)^c  \cup
 \big\{[\nu] \in \Gamma(H,A) \mid
\epsilon(\langle \eta_j \rangle) \nsupseteq \ker(\nu)\cap \xi_j \big\} \Bigg).
\end{equation}
\end{theorem}

\begin{proof}
The result  follows from Theorems \ref{main theorem}
and \ref{thm:uprho2}, as well as formula \eqref{eq:uaw union}.
\end{proof}

The simplest situation in which the above theorem applies is that in
which there are no translation factors in the subgroups
comprising the characteristic variety.

\begin{corollary}
\label{cor:pull-back1}
Suppose $V^i(X)=V(\xi_1)\cup \cdots \cup V(\xi_s)$ is
a union of algebraic subgroups of $\wH$.
Then
\begin{equation}
\label{eq:qoa}
\Omega^i_A(X) = \Gamma(H,A) \setminus \bigcup_{j=1}^{s}
q^{-1} ( \sigma_{\oA}(\xi_j)) .
\end{equation}
In particular if $X^{\nu}$ is a free abelian cover with finite Betti
numbers up to degree $i$, then any finite regular
abelian cover of $X^{\nu}$ has the same finiteness property.
\end{corollary}

\begin{proof}
By Theorem \ref{thm:xik},
\begin{equation}
\label{eq:qoa prime}
\Omega^i_A(X) = \Gamma(H,A) \setminus \bigcup_{j=1}^{s}
\sigma_{A}(\xi_j).
\end{equation}
Indeed, for each $j$ we have $\eta_j=1$, and thus
$\{[\nu] \mid \epsilon(\langle \eta_j \rangle) \nsupseteq \ker(\nu)\cap \xi_j\}$
is the empty set. Applying now Proposition \ref{prop:comp schubert}
ends the proof.
\end{proof}

\subsection{Toric complexes}
\label{subsec:tc}

We illustrate the above corollary with a class of spaces arising
in toric topology.  Let $L$ be a simplicial complex with $n$ vertices,
and let $T^n=S^1\times \cdots \times S^1$ be the $n$-torus,
with the standard product cell decomposition.  The
\textit{toric complex}\/ associated to $L$, denoted $T_L$,
is the union of all subcomplexes of the form
\begin{equation}
\label{eq:tsigma}
T^{\sigma}=\{(x_1,\dots, x_n)\in T^n \mid
\text{$x_i = *$ if $i \notin \sigma$} \},
\end{equation}
where $\sigma$ runs through the simplices of $L$, and
$*$ is the (unique) $0$-cell of $S^1$.
Clearly, $T_L$ is a connected CW-complex, with unique
$0$-cell corresponding to the empty simplex $\emptyset$.
The fundamental group of $T_L$ is the right-angled Artin group
\begin{equation}
\label{eq:raag}
G_L=\langle v \in \sV \mid \text{$vw=wv$ if
$\{v, w\} \in E$} \rangle,
\end{equation}
where $\sV$ and $\sE$ denote the $0$-cells and $1$-cells of $L$.
Furthermore, a classifying space for $\pi_1(T_L)$ is the toric
complex $T_{\Delta(L)}$, where $\Delta(L)$ is the flag complex
associated to $L$.

Evidently, $H_1(T_L,\Z)=\Z^n$; thus, we may identify the character
group of $\pi_1(T_L)$ with the algebraic torus $(\C^*)^{\sV}:=(\C^*)^n$.
For any subset $\sW \subseteq \sV$, let
$(\C^*)^{\sW} \subseteq (\C^*)^{\sV}$ be the corresponding
subtorus; in particular, $(\C^*)^{\emptyset}=\{1\}$.
Note that $(\C^*)^{\sW} = V(\xi_{\sW})$, where $\xi_{\sW}$
is the sublattice of $\Z^n$ spanned by the basis vectors
$\{e_i \mid i\notin \sW\}$.
From \cite{PS-toric}, we have the following description
of the characteristic varieties of our toric complex:
\begin{equation}
\label{eq:cvtoric}
V^i(T_L)=\bigcup_{\sW}\: V(\xi_{\sW}),
\end{equation}
where the union is taken over all subsets $\sW \subseteq \sV$
for which there is a simplex $\sigma \in L_{\sV\setminus \sW}$
and an index $j\le i$ such that
$\tilde{H}_{j-1-\abs{\sigma}}(\lk_{L_{\sW}}(\sigma), \C)\ne 0$.
Here,  $L_{\sW}$ denotes the subcomplex induced by $L$ on $\sW$, and
$\lk_K(\sigma)$ denotes the link of a simplex $\sigma\in L$ in a
subcomplex $K \subseteq L$.

From the above, we see that the assumptions from
Corollary~\ref{cor:pull-back1} are true for the
classifying space of a right-angled Artin group, and, in fact,
for any toric complex. Hence, we obtain the following corollaries.

\begin{corollary}\label{cor: toric_formula}
Let $T_L$ be a toric complex. Then,
\[
\Omega^i_A(T_L) = \Gamma(H,A) \setminus \bigcup_{\sW}
q^{-1} ( \sigma_{\oA}(\xi_{\sW}))
\]
where each $\sigma_{\oA}(\xi_{\sW})\subseteq \Grass_{n-r}(\Q^n)$ is
the special Schubert variety corresponding to the coordinate plane
$\xi_{\sW}\otimes \Q$, and the union is taken over all subsets $\sW \subseteq \sV$
for which there is a simplex $\sigma \in L_{\sV\setminus \sW}$
and an index $j\le i$ such that
$\tilde{H}_{j-1-\abs{\sigma}}(\lk_{L_{\sW}}(\sigma), \C)\ne 0$.
\end{corollary}

\begin{corollary}
\label{cor: toric}
Let $T_L$ be a toric complex. If $T_L^{\bar\nu}$ is a free abelian
cover with finite Betti numbers up to some degree $i$, then all regular,
finite abelian covers of $T_L^{\bar\nu}$ also have finite Betti
numbers up to degree $i$.
\end{corollary}

\subsection{When the translation order is coprime to $\abs{\Tors(A)}$}
\label{subsec:coprime}

We now return to the general situation, where the characteristic
variety is a union of torsion-translated subgroups.
In the next proposition, we identify a condition
on the order of translation of these subgroups, insuring that
diagram \eqref{commutative diagram} is a pull-back diagram.

As usual, let $A$ be a quotient of $H=H_1(X,\Z)$, and let
$\oA=A/\Tors(A)$ be its maximal torsion-free quotient.
Recall that the canonical projection, $q: \Gamma(H,A)\to \Gamma(H,\oA)$,
restricts to a map $q|_{\Omega^i_{A}(X)} \colon \Omega^i_{A}(X) \to \Omega^i_{\oA}(X)$,
and that resulting commuting square is a pull-back diagram if and only if
$\Omega_A^i(X)$ is the full pre-image of $q$.

Using Proposition \ref{prop:not pback1}  and Theorem \ref{main theorem},
we obtain the following consequence.

\begin{prop}
\label{prop:pull-back2}
Suppose the characteristic variety $V^i(X)$ is of the form
$\bigcup_{j} \rho_j T_j$, where each $T_j\subset \widehat{H}$
is an algebraic subgroup, and each $\bar{\rho_j}\in \widehat{H}/T_j$
has finite order, coprime to the order of $\Tors(A)$.  Then
$\Omega_A^i(X)=q^{-1}\Big(\Omega_{\oA}^i(X)\Big)$.
\end{prop}

Here is an application. As usual, let $X$ be a connected
CW-complex with finite $k$-skeleton.  Assume $H=H_1(X,\Z)$
has no torsion, and identify the character torus $\widehat{H}$
with $(\C^*)^n$, where $n=b_1(X)$.

\begin{corollary}
\label{cor:omaix}
Suppose that, for some $i\leq k$, there is an $(n-1)$-dimensional 
subspace $L \subseteq H^1(X;\Q)$ such that 
$V^i(X) = (\bigcup_{\alpha}\rho_\alpha T ) \cup Z$, where $Z$ is 
a finite set, $T = \exp(L \otimes \C)$ and $\rho_\alpha$ is a torsion
element in $(\C^*)^n$ of order coprime to that of $\T(A)$, for 
each $\alpha$.  Let $r=\rank A$. Then,
\[
\Omega_A^i(X)=
\begin{cases}
\: \Gamma(H, A), & \text{if $r = 1$;} \\
\: q^{-1}(\Grass_r(L)),  & \text{if $1 < r < n$;}\\
\: \emptyset,  &  \text{if $r \geq n$.}
\end{cases}
\]
\end{corollary}

\begin{proof}
From \cite[Proposition 8.6]{Su14}, we have that
\[
\Omega_r^i(X)=
\begin{cases}
\: \QP^{n-1}, & \text{if $r = 1$;} \\
\: \Grass_r(L),   & \text{if $1 < r < n$;}\\
\: \emptyset,  & \text{if $r \geq n$.}
\end{cases}
\]
On the other hand, Proposition \ref{prop:pull-back2} shows that
$\Omega_A^i(X)=q^{-1}\big(\Omega_{r}^i(X)\big)$,
and so the desired conclusion follows.
\end{proof}

\subsection{When the translation order divides $\abs{\Tors(A)}$}
\label{subsec:divides}

To conclude this section, we give some sufficient conditions
on the groups $H$ and $A$, and on the order of translation
of the subgroups comprising $V^i(X)$, insuring that
diagram \eqref{commutative diagram} is {\em not}\/
a pull-back diagram.

Proposition \ref{prop:not pback2} and Theorem \ref{main theorem}
yield the following immediate application.

\begin{corollary}
\label{single}
Let $X$ be a connected CW-complex with finite $k$-skeleton.
Assume that
\begin{romenum}
\item The group $H=H_1(X,\Z)$ is torsion-free, and $A$ is a quotient of $H$.
\item There is a degree $i\le k$ such that the positive-dimensional
components of $V^i(X)$ form a torsion-translated subgroup $\rho T$
inside $\wH$.
\item The rank of $A$ is less than the rank of $H$, and 
$\ord(\rho)$ divides $c(A)$.
\end{romenum}
Then $\Omega_A^i(X)  \subsetneqq  q^{-1}\Big(\Omega_{\oA}^i(X)\Big)$.
\end{corollary}

Note that the hypothesis of Corollary \ref{single} are satisfied in
Examples \ref{ex:421} and \ref{strata infinite}, thus explaining why,
in both cases, diagram \eqref{commutative diagram}
is not a pull-back diagram.  Here is one more situation
when that happens.

\begin{corollary}
\label{prop:transverse}
Suppose $V^i(X)= \rho_1T_1\cup \cdots \cup \rho_sT_s$,
with each $T_j$ an algebraic
subgroup of $\wH$ and each $\rho_j$ a torsion element in
$\widehat{H}\setminus T_j$.  Furthermore, suppose that
\begin{romenum}
\item \label{tt1} The identity component of $T_1$ is
not contained in $T'=T_2 \cdots T_s $ (internal product in $\wH$).
\item \label{tt2} The order of $\rho_1$ divides $c(A)$.
\item \label{tt3} $\rank A < \rank H -\dim T'$.
\end{romenum}
Then
$\Omega_A^i(X) \subsetneqq q^{-1}\Big(\Omega_{\oA}^i(X)\Big)$.
\end{corollary}

\begin{proof}
Split $H$ as a direct sum, $H' \oplus H''$, so that $\wH'=T'$ and
$\wH'' = \wH/T'$.  Let $p\colon H \surj H''$ be the
canonical projection, and let $\hat{p}\colon \wH'' \to \wH$
be the induced morphism. By assumption \eqref{tt1},
we have that $\hat{p}^{-1}(T_1)$ is a
positive-dimensional algebraic subgroup of $\wH''$.
Thus, the positive-dimensional components of
$W=V^i(X) \cap \wH''$ form a torsion-translated
subgroup of $\wH''$, namely, $\hat{p}^{-1} (\rho_1 T_1)$.
Moreover, assumptions \eqref{tt2} and \eqref{tt3} imply that
$\rank A < \rank H''$ and $\ord( \hat{p}^{-1} (\rho_1))$
divides $c(A)$.

Applying now Proposition \ref{prop:not pback2}
to the algebraic group $\wH''$ and to the subvariety $W$
yields an epimorphism $\mu\colon H'' \surj A$
such that $\im(\hat{\bar{\mu}})\cap \hat{p}^{-1} (\rho_1 T_1)$ is finite,
and $\im(\hat{\mu}) \cap \hat{p}^{-1} (\rho_1 T_1)$ is infinite.
Setting  $\nu=\mu\circ p$, we see that $\im(\hat{\nu})=\im(\hat\mu)$.
Thus, $\im(\hat{\bar{\nu}}) \cap V^i(X)$
is finite, while $\im(\hat\nu) \cap V^i(X)$ is infinite.
\end{proof}

\section{Quasi-projective varieties}
\label{sect:quasi-proj}

We conclude with a discussion of the generalized
Dwyer--Fried sets of smooth, quasi-projec\-tive varieties.

\subsection{Characteristic varieties}
\label{subsec:cv qproj}

A space $X$ is said to be a {\em (smooth) quasi-projective variety}\/
if there is a smooth, complex projective variety $\overline{X}$ and a
normal-crossings divisor $D$ such that $X=\overline{X}\setminus D$.
For instance, $X$ could be the complement of an algebraic
hypersurface in $\CP^{d}$.

The structure of the characteristic varieties of such spaces was
determined through the work of Beauville, Green and Lazarsfeld,
Simpson, Campana, and Arapura in the 1990s, and further refined
in recent years.  We summarize these results, essentially in
the form proved by Arapura.

\begin{theorem}[\cite{Ar}]
\label{thm:arapura}
Let $X=\overline{X}\setminus D$, where $\overline{X}$
is a smooth, projective variety and $D$ is a normal-crossings
divisor.
\begin{romenum}
\item \label{ar1}
If either $D=\emptyset$ or $b_1(\overline{X})=0$,
then each characteristic variety $V^i(X)$
is a finite union of unitary translates of algebraic
subtori of $T=H^1(X,\C^{*})$.
\item \label{ar2}
In degree $i=1$, the condition that
$b_1(\overline{X})=0$ if $D\ne \emptyset$ may be lifted.
Furthermore, each positive-dimensional component
of $V^1(X)$ is of the form $\rho \cdot S$,
where $S$ is an algebraic subtorus of $T$, and
$\rho$ is a {\em torsion}\/ character.
\end{romenum}
\end{theorem}

For instance, if $C$ is a connected, smooth complex curve
of negative Euler characteristic, then $V^1(C)$ is the full
character group $H^1(C,\C^{*})$.
For an arbitrary smooth, quasi-projective variety $X$,
each positive-dimensional component of $V^1(X)$
arises by pullback along a suitable orbifold fibration (or, pencil).
More precisely, if $\rho \cdot S$ is such a component,
then $S=f^*(H^1(C,\C^{*}))$, for some curve $C$,
and some holomorphic, surjective map
$f\colon X \to C$ with connected generic fiber.

Using this interpretation, together with recent work of
Dimca, Artal-Bartolo, Cogolludo, and Matei (as
recounted in \cite{Su14}), we can describe the
variety $V^1(X)$, as follows.

\begin{theorem}
\label{thm:vxi qp}
Let $X$ be a smooth, quasi-projective variety. Then
\[
V^1(X)=\bigcup_{\xi\in \Lambda}V(\xi) \cup
\bigcup_{\xi\in \Lambda'} \big( V(\xi)\setminus V(\overline{\xi}) \big)
\cup Z,
\]
where $Z$ is a finite subset of $T=H^1(X,\C^*)$, and $\Lambda$
and $\Lambda'$ are certain (finite) collections of subgroups of
$H=H_1(X,\Z)$.
\end{theorem}

\subsection{Dwyer--Fried sets}
\label{subsec:df qproj}

Using now Theorem \ref{thm:vxi qp} (and keeping the notation therein),
Proposition \ref{prop:uprho4} yields the following structural result for
the degree~$1$ Dwyer--Fried sets of a smooth, quasi-projective variety.

\begin{theorem}
\label{thm:omega qp}
Let $A$ be a quotient of $H$.  Then
\[
\Omega^1_A(X)=\Gamma(H,A)\setminus \Bigg( \bigcup_{\xi\in \Lambda}
q^{-1}\big(\sigma_{\overline{A}} (\overline\xi)\big) \cup
\bigcup_{\xi\in \Lambda'}
\Big( q^{-1}\big(\sigma_{\overline{A}} (\overline\xi)\big) \cap
\theta_A(\xi)\Big) \Bigg).
\]
\end{theorem}

In certain situations, more can be said.
For instance, Proposition \ref{prop:pull-back2} yields the
following corollary.

\begin{corollary}
\label{cor:qp1}
Suppose the order of $\overline{\xi}/\xi$ is coprime to
$c(A)$, for each $\xi \in \Lambda'$.
Then $\Omega_A^1(X)=q^{-1}\big(\Omega_{\oA}^1(X)\big)$.
\end{corollary}

Similarly, Corollary~\ref{prop:transverse} has the following consequence.

\begin{corollary}
\label{cor:qp2}
Suppose $H$ is torsion-free, and there is a subgroup
$\chi \in \Lambda'$ such that
\begin{romenum}
\item
$V(\overline\chi)$ is not contained in
$T':=V(\bigcap_{\xi\in \Lambda\cup \Lambda' \setminus \{\chi\}} \xi)$.
\item
There is a non-zero element in
$\overline{\chi}/\chi$ whose order divides $c(A)$.
\item $\rank A < \codim T'$.
\end{romenum}
Then
$\Omega_A^1(X)  \subsetneqq  q^{-1}\Big(\Omega_{\oA}^1(X)\Big)$.
\end{corollary}

For the remainder of this section, we shall give some concrete examples
of quasi-projective manifolds $X$ for which the computation of the sets
$\Omega^1_A(X)$ can be carried out explicitly.

\subsection{Brieskorn manifolds}
\label{subsec:Sigma(pqr)}

Let $(a_1, \dots, a_n)$ be an $n$-tuple of integers,
with $a_j \geq 2$. Consider the variety $X$ in $\C^n$ defined by the
equations $c_{j 1}x^{a_1}_1 + \cdots + c_{j n}x^{a_n}_n=0$,
for $1 \leq j \leq n-2$. Assuming all maximal minors of the
matrix $\big(c_{j k}\big)$ are non-zero, $X$ is a
quasi-homogeneous surface, with an isolated singularity at $0$.

The space $X$ admits a good $\C^*$-action.
Set $X^*=X\setminus{0}$, and let $p\colon X^* \surj C$ be the
corresponding projection onto a smooth projective curve.
Then $p^*\colon H^1(C, \C)\to H^1(X^*, \C)$ is an isomorphism,
the torsion subgroup of $H=H_1(X^*,\Z)$ coincides with the
kernel of $p_*\colon H_1(X^*, \Z) \to H_1(C, \Z)$.

By definition, the Brieskorn manifold $M= \Sigma (a_1, \dots, a_n)$
is the link of the quasi-ho\-mogenous singularity $(X, 0)$.
As such, $M$ is a closed, smooth, oriented $3$-manifold
homotopy equivalent to $X^*$.  Put
\[
l=\lcm(a_1, \dots, a_n), \quad
l_j=\lcm (a_1, \dots,\widehat{a_j} ,\dots,a_n), \quad
a=a_1 \cdots a_n.
\]
The $S^1$-equivalent
homeomorphism type of $M$ is determined by
the following Seifert invariants associated to the
projection $\left. p\right|_M\colon M \to C$:
\begin{itemize}
\item The exceptional orbit data,
$(s_1(\alpha_1, \beta_1), \cdots, s_n(\alpha_n, \beta_n))$,
with $\alpha_j=l/l_j$, $\beta_j l\equiv a_j \mod \alpha_j$
and $s_j=a / (a_j l_j)$, where $s_j=(\alpha_j \beta_j)$
means $(\alpha_j \beta_j)$ repeated $s_j$ times,
unless $\alpha_j=1$, in which case $s_j=(\alpha_j \beta_j)$
is to be removed from the list.
\item The genus of the base curve, given by
$g=\frac{1}{2}\left( 2+(n-2)a/l-\sum^n _{j=1}s_j \right)$.
\item The (rational) Euler number of the Seifert fibration,
given by $e=-a/l^2$.
\end{itemize}

The group $H=H_1(M, \Z)$ has rank $2g$, and torsion part
of order $\alpha_1^{s_1} \cdots \alpha_n^{s_n} \cdot \abs{e}$.
Identify the character group $\widehat{H}$ with a disjoint
union of copies of $\widehat{H}_0=(\C^*)^{2g}$, indexed by
$\Tors(H)$, and set
$\alpha=\alpha_1^{s_1} \cdots \alpha_n^{s_n} / 
\lcm(\alpha_1,\dots ,\alpha_n)$.

\begin{prop}[\cite{DPS-mz}]
\label{prop:v1seifert}
The positive-dimensional components of $V^1(M)$ are as follows:
\begin{romenum}
\item
$\alpha -1$ translated copies of $\widehat{H}_0$,
if $g=1$.
\item
$\widehat{H}_0$, together with $\alpha -1$ translated
copies of $\widehat{H}_0$,  if $g>1$.
\end{romenum}
\end{prop}

Denote the elements in $\Tors(H)$ corresponding to the
$\alpha-1$ translated copies of $\widehat{H}_0$ by
$h_1, \dots , h_{\alpha-1}$.  We then have the following corollaries.

\begin{corollary}
\label{cor:Brieskorn}
Let $M=\Sigma (a_1, \dots, a_n)$ be a Brieskorn
manifold, and let $A$ be a quotient of $H=H_1(M,\Z)$, 
with $r=\rank A$.
\begin{romenum}
\item
If $g>1$, then $\Omega_{A}^1(M)=\Omega_{\oA}^1=\emptyset$.

\item
If $g=1$, then $\Omega_{\oA}^1(M)=\Grass_r(\Q^{2g})$,
while
\[
\Omega_{A}^1(M)=\big\{[\nu]\in \Gamma(H, A) \mid \text{$\nu(h_i)=0$
for $i=1, \dots, \alpha-1$}\big\}.
\]
\end{romenum}
\end{corollary}

\begin{corollary}
\label{cor:Brieskorn prime}
Suppose $g=1$ and $\alpha>1$.  Then  $\Omega_{\oH}^1(M)={\pt}$,
yet $\Omega_{H}^1(M)=\emptyset$; that is,
$b_1(M^{\fab})< \infty$, yet $b_1(M^{\ab})= \infty$.
\end{corollary}

\begin{example}
\label{ex:sigma248}
Consider the Brieskorn manifold $M=\Sigma(2,4,8)$.
Using the algorithm described by Milnor in \cite{Mi},
we see that the fundamental group of $M$ has presentation
\[
G=\langle x_1, x_2, x_3\mid
x_1x_3^2=x_3^2x_1, \:
x_2x_3^2=x_3^2x_2, \:
x_3^{2} ( x_3 x_1  x_2 x_1^{-1} x_2^{-1})^2 =  1
\rangle,
\]
while its abelianization is $H=\Z^2 \oplus \Z_4$. Identifying
$\widehat{H}= (\C^*)^2 \times \{\pm 1, \pm i\}$, we find that
$V^1(M) = \{1\} \cup  (\C^*)^2 \times \{-1\}$.
(The positive-dimensional component in $V^1(M)$ arises
from an elliptic orbifold fibration $X\to \Sigma_1$
with two multiple fibers, each of multiplicity $2$.)  
By Proposition \ref{prop:fab f ab inf}, then, $b_1(M^{\fab})< \infty$,
while $b_1(M^{\ab})= \infty$.

Now take $A=\Z\oplus \Z_4$.  Applying Corollary \ref{cor:Brieskorn}
(with $g=1$ and $\alpha=2$), we conclude that $\Omega_{\oA}^1(M)=\QP^1$,
while $\Omega_{A}^1(M)$ consists of two copies of $\Gamma(\Z^2, A)$, 
naturally embedded in $\Gamma(H, A)$.  
\end{example}

\subsection{The Catanese--Ciliberto--Mendes Lopes surface}
\label{subsec:kahler not open}
We now give an example of a smooth, complex projective
variety $M$ for which the generalized Dwyer--Fried sets
exhibit the kind of subtle behavior predicted by Corollary \ref{cor:qp2}.

Let $C_1$ be a (smooth, complex) curve of genus $2$
with an elliptic involution $\sigma_1$ and $C_2$ a curve of
genus $3$ with a free involution $\sigma_2$.
Then $\Sigma_1=C_1/\sigma_1$ is a curve of genus $1$,
and $\Sigma_2=C_2/\sigma_2$ is a curve of genus $2$.
The group $\Z_2$ acts freely on the product $C_1 \times C_2$
via the involution $\sigma_1\times \sigma_2$; let
$M$ be the quotient surface.  This variety, whose
construction goes back to Catanese, Ciliberto, and
Mendes Lopes \cite{CCM}, is a minimal surface
of general type with $p_g(M) = q(M) = 3$ and $K_M^2=8$.

The projection $C_1 \times C_2\to C_1$  descends to an orbifold
fibration $f_1 \colon M\to \Sigma_1$ with two multiple fibers, each
of multiplicity $2$, while the projection $C_1 \times C_2\to C_2$
descends to a holomorphic fibration $f_2\colon M\to \Sigma_2$.
It is readily seen that $H=H_1(M,\Z)$ is isomorphic to $\Z^6$;
fix a basis $e_1,\dots , e_6$ for this group.
A computation detailed in \cite{Su14} shows
that the characteristic variety $V^1(M)\subset (\C^*)^6$ has two
components, corresponding to the above two pencils;
more precisely,
\begin{equation}
\label{eq:v1 ccm}
V^1(M)=V(\xi_1) \cup \big(V(\xi_2)\setminus V(\overline{\xi}_2)\big),
\end{equation}
where $\xi_1=\spn \{e_1, e_2\}$ and
$\xi_2=\spn \{2e_3, e_4,e_5,e_6\}$.

Now suppose $A$ is a quotient of $H=\Z^6$, and
let $q\colon \Gamma(H,A) \to  \Gamma(H,\overline{A})$
be the canonical projection.  By Theorem \ref{thm:omega qp},
\begin{equation}
\label{eq:ccm}
\Omega^1_A(M)  =
\Gamma(H,A)\setminus
\left(
q^{-1}(\sigma_{\overline{A}}(\xi_1) ) \cup \big(
 q^{-1}(\sigma_{\overline{A}}(\overline\xi_2) ) \cap
\theta_A(\xi_2) \big)
\right).
\end{equation}
Let us describe explicitly this set in a concrete situation.

\begin{example}
\label{ex:ccm}
Let $A=\Z\oplus \Z_2$, and identify $\overline{A}=\Z$ and
$\Gamma(H,\Z)=\QP^5$.  The fiber of $q$ is the set
$\Gamma=(\Z_2^5)^*$. Given an epimorphism
$\nu \colon H \to \Z\oplus \Z_2$, let $\bar\nu\colon H \to \Z$
and $\nu'\colon H\to \Z_2$ be the composites of $\nu$
with the projections on the respective factors.
The terms on the right-side of \eqref{eq:ccm}
are as follows:

\begin{itemize}
\item
$\sigma_{\Z}(\xi_1)$ is
the projective subspace $\QP^3=\bP((H/\xi_1)^{\vee}\otimes \Q)$
spanned by $e_3,\dots, e_6$.

\item
$\sigma_{\Z}(\overline\xi_2)$
is the projective line $\QP^1=\bP((H/\xi_2)^{\vee}\otimes \Q)$
spanned by $e_1$ and $e_2$.

\item
$q^{-1}(\QP^1)\cap\theta_A(\xi_2)$ consists of those $[\nu]$
satisfying $\overline\nu(e_3)=\cdots=\overline\nu(e_6)=0$, and
$\nu'(e_3)=1, \nu'(e_4)=\nu'(e_5)=\nu'(e_6)=0$.
\end{itemize}

This completes the description of the set $\Omega^1_{\Z\oplus \Z_2}(M)$.
Clearly, $\Omega^1_1(M)=\QP^5 \setminus \QP^3$.  Furthermore,
note that the restriction map $q^{-1}(\QP^1)\cap\theta_A(\xi_2) \to \QP^1$
is a $2$-to-$1$ surjection.  Thus, the restriction map
$\Omega^1_{\Z \oplus \Z_2}(M) \to \Omega^1_1(M)$ is {\em not}\/
a set fibration: the fiber over $\Omega^1_1(M) \setminus \QP^1$
has cardinality $31$, while the fiber over $\QP^1$ has cardinality $29$.
\end{example}

\newcommand{\arxiv}[1]
{\texttt{\href{http://arxiv.org/abs/#1}{arXiv:#1}}}
\newcommand{\doi}[1]
{\texttt{\href{http://dx.doi.org/#1}{doi:#1}}}
\renewcommand{\MR}[1]
{\href{http://www.ams.org/mathscinet-getitem?mr=#1}{MR#1}}

\end{document}